\documentclass[final,leqno,onefignum,onetabnum]{siamltex1213}

\usepackage{amsmath}
\usepackage{amsfonts}
\usepackage{amssymb}
\usepackage{latexsym}
\usepackage{multirow}

\newcommand{\stinvF}{{C}_{\rm inv}^1}

\newcommand{\stinvb}{{C}_{\rm inv}^2}

\usepackage[usenames,dvipsnames,svgnames,table]{xcolor}

\usepackage{dsfont}

\newcommand{\be}{\begin{equation}}
\newcommand{\ee}{\end{equation}}
\newcommand{\bea}{\begin{eqnarray}}
\newcommand{\eea}{\end{eqnarray}}

\newcommand{\nno}{\nonumber}

\newcommand{\vecL}{\bold{\Pi}_2}
\newcommand{\mbf}[1]{\mbox{\boldmath$\rm{#1}$}}

\newcommand{\el}{ \kappa \in \mathcal{T}}

\newcommand{\diam}{\operatorname{diam}}

\newcommand{\stfes}{S^{\bf p}(\mathcal{U} ; \mathcal{T})}

\newcommand{\timesub}{\mathcal{U}_n(\mathcal{T})}

\newcommand{\stnfessub}[1]{V^{\bf p}(I_{#1};\mathcal{T}; \timesub)}

\newcommand{\stp}{\lambda_{\kappa_n}}

\newcommand{\tmesh}{\tau_{n}}

\newcommand{\tsubmesh}[1]{\tau_{n,#1}}

\newcommand{\stF}{\mathcal{F}}

\newcommand{\mean}[1]{\{\!\!\{#1\}\!\!\}}                
\newcommand{\jump}[1]{[\![#1]\!]}                        
\newcommand{\ujump}[1]{\lfloor #1\rfloor}                

\renewcommand{\k}{\kappa}

\newcommand{\su}{\sum_{\el}}
\newcommand{\ud}{\,\mathrm{d}}
\newcommand{\ndg}[1]{| \kern -.25mm \|{#1}| \kern -.25mm \|}
\newcommand{\ncdg}[1]{| \kern -.25mm \|{#1}| \kern -.25mm \|_{\rm DG}}
\newcommand{\nsdg}[1]{| \kern -.25mm \|{#1}| \kern -.25mm \|_{\rm s}}
\newcommand{\nstdg}[1]{| \kern -.25mm \|{#1}| \kern -.25mm \|_{L_2 (J; \mathcal{D} )}}
\newcommand{\ltwo}[2]{\|{#1}\|_{{#2}}}
\newcommand{\linf}[2]{\|{#1}\|_{L_{\infty}({#2})}}

\newcommand{\no}{[\kern -.8mm [}
\newcommand{\nc}{]\kern -.8mm ]}

\newcommand{\stelem}{ \mathcal{U} \times \mathcal{T} }

\newcommand{\stelemsub}{ \timesub \times \mathcal{T}}

\newcommand{\norm}[2]{\| {#1} \|_{#2}}

\newcommand*\samethanks[1][\value{footnote}]{\footnotemark[#1]}

\newtheorem{remark}[theorem]{Remark}
\newtheorem{assumption}[theorem]{Assumption}

\title{$\boldsymbol{\lowercase{hp}}$-Version space-time discontinuous Galerkin methods for parabolic problems on prismatic meshes \\ (Corrected version 01/11/2024)}

\author{ANDREA CANGIANI\thanks{Department of Mathematics, University of Leicester,   Leicester LE1 7RH, United Kingdom (\email{Andrea.Cangiani@le.ac.uk},
\email{zd14@le.ac.uk}).}
\and ZHAONAN DONG\samethanks
\and Emmanuil H. Georgoulis\thanks{Department of Mathematics, University of Leicester,   Leicester LE1 7RH, United Kingdom (\email{Emmanuil.Georgoulis@le.ac.uk}) \& Department of Mathematics, School of Applied Mathematical and Physical Sciences, National Technical University of Athens, Zografou 15780, Greece.}
}

\begin{document}
\maketitle
\slugger{sisc}{xxxx}{xx}{x}{x--x}

{\small \bf **This is a modified/corrected version of the published paper [SISC 39(4) pp.A1251 --A1279 (2017)]. In particular, the proof of Theorem 5.11 in the published paper contains two inaccuracies. These are now corrected in the present version. A new result, Theorem 5.12 treating the case of time-dependent diffusion coefficient has also been added for completeness. The modified text starts right after Assumption 5.10. and concludes in the last paragraph of Section 5.**}

\begin{abstract}
We present a new $hp$-version space-time discontinuous Galerkin (dG) finite element method for the numerical approximation of parabolic evolution equations on general spatial meshes consisting of polygonal/polyhedral (polytopic) elements, giving rise to prismatic space-time elements. A key feature of the proposed method is the use of space-time elemental polynomial bases of \emph{total} degree, say $p$,
defined in the physical coordinate system, as opposed to standard dG-time-stepping methods whereby spatial elemental bases are tensorized with temporal basis functions. This approach leads to a fully discrete $hp$-dG scheme using fewer degrees of freedom for each time step, compared to dG time-stepping schemes employing tensorized space-time basis, with acceptable deterioration of the approximation properties. A second key feature of the new space-time dG method is the incorporation of very general spatial meshes consisting of possibly polygonal/polyhedral elements with \emph{arbitrary} number of faces. A priori error bounds are shown for the proposed  method in various norms. An extensive comparison among the new space-time dG method, the (standard) tensorized space-time dG methods, the classical dG-time-stepping, and conforming finite element method in space, is presented in a series of numerical experiments.
\end{abstract}



\pagestyle{myheadings}
\thispagestyle{plain}
\markboth{A.~CANGIANI, Z.~DONG, E.H.~GEORGOULIS}{\sc $hp$-space-time dg methods for parabolic problems on prismatic meshes}

\section{Introduction}
The discontinuous Galerkin (dG) method can be traced back to \cite{reedhill}, where it was introduced as a nonstandard finite element scheme for solving the neutron transport equation. This dG method was analyzed in \cite{MR58:31918}, where it was also applied as a time stepping scheme for initial value problem for ordinary differential equations, and was shown to be equivalent to certain implicit Runge-Kutta methods. Jamet \cite{jamet1978galerkin} introduced a dG time-stepping scheme for parabolic problems on evolving domains, later extended and analysed in  \cite{eriksson1985time,eriksson1991adaptive,eriksson1995adaptive, eriksson1995adaptiveNonlinear,eriksson1998adaptive}. For an introduction, we refer to the classic monograph \cite{thomee1984galerkin} and the references therein. In \cite{makridakis1997stability}, the quasioptimality of
the dG time-stepping method for parabolic problems in mesh-dependent norms is established. Also, dG time-stepping convergence analyses under minimal regularity were shown in \cite{MR2139223,MR2217386,MR2684359}. In all aforementioned literature, convergence of the discrete solution to the exact solution is achieved by reducing spatial mesh size $h$ and time step size $\tau$  at some fixed (typically low) order.

On the other hand, the $p$- and $hp$-version finite element method (FEM) appeared in the 1980s (see \cite{babuvska1987optimal,Babuska-Suri:RAIRO:1987}, and also the textbook \cite{schwab} for a extensive survey). $p$- and $hp$-version FEM can achieve exponential rates of convergence when the underlying solution is locally analytic by increasing the polynomial order $p$ and/or locally grading the meshsize towards corner or edge singularities.  In this vein, the analyticity in the time-variable in parabolic problems has given rise to the use of $p$- and $hp$-version FEM for time-stepping \cite{babuska1989h, babuska1990h}, followed by \cite{schotzau2000time}, where $hp$-version dG time-stepping in conjunction with FEM in space was shown to converge exponentially.

Space-time $hp$-version dG methods have also been popular during the last 15 years \cite{sudirham2006space,van2008space,MR2838303}, typically, employing space-time slabs with possibly anisotropic tensor-product space-time elemental polynomial basis. More recently, space-time hybridizable-dG methods have been developed for flow equations \cite{rhebergen2012space,rhebergen2013space} and for Hamilton-Jacobi-Bellman equations \cite{smears2014discontinuous}.

The aim of this work is to present a new $hp$-version space-time dG method for the numerical approximation of parabolic evolution equations. A key attribute of the new method is the use of space-time elemental polynomial bases of \emph{total} degree, say $p$,
defined in the \emph{physical} coordinate system, as opposed to standard dG-time stepping methods whereby spatial elemental bases (conforming or non-conforming) are tensorized with temporal basis functions and are mapped from a reference element. This approach leads to a fully discrete $hp$--dG scheme which uses fewer degrees of freedom for each time step, compared to dG-time stepping schemes employing tensorized space-time bases. On the other hand, the use of total degree space-time bases leads to  half an order loss in mesh size of the expected rate of convergence in $L_2(L_2)$--norm and in $L_{\infty}(L_2)$--norm. Nonetheless, the method is shown to converge optimally in the broken $L_{2}(H^1)$--norm, with the error dominated asymptotically by the spatial convergence rate. The marginal deterioration in the convergence properties, compared to the standard space-time tensorized basis paradigm, turns out to be an acceptable trade-off given the substantial reduction in the local elemental basis cardinality. For earlier use of linear space-time basis functions for large flow computations, we refer to \cite{van2002space1,van2002space2}.

A second key attribute of the proposed method, stemming from the use of \emph{physical} frame basis functions, is its immediate applicability to extremely general spatial meshes consisting of \emph{polytopic} elements (polygonal/polyhedral elements in two/three space dimensions), giving rise to \emph{prismatic space-time elements}. Finite element methods with general-shaped elements have enjoyed a strong recent interest in the literature, aiming to reduce the computational cost of standard approaches based on simplicial or box-type elements, see, e.g., \cite{DPE,MR3259024,di2015hybrid,DiPietroErn,dg_cfes_2012,BassiJCP2012,cangiani2013hp,MR3033077, MR3008290} for dG schemes, \cite{composite,conforming_FEM_poly,mimetic,VEM} for conforming schemes, and the references therein.

Here, we prove the unconditional stability of the new space-time dG method, via the proof of an inf-sup condition for space-time elements with arbitrary aspect ratio between the time-step $\tau$ and the local spatial mesh-size $h$; this is an extension of the respective result from \cite[Lemma 5.1]{cangiani2015hp}, where global shape-regularity was required. As in \cite{cangiani2013hp,cangiani2015hp}, the analysis allows for arbitrarily small/degenerate $(d-k)$-dimensional element facets, $k=1,\ldots,d-1$, with $d$ denoting the spatial dimension.
However, by considering different mesh assumptions compared to \cite{cangiani2013hp,cangiani2015hp}, the proposed method is proved to be stable also, independently of the number of $(d-1)$-dimensional faces per element. (Note that the elemental basis is {\em independent} of the element's geometry, and in particular of the number of faces.)
To the best of our knowledge, this is the first result in the literature whereby polytopic meshes  with arbitrary number of faces are allowed. This setting gives great flexibility in resolving complicated geometrical features without resorting to locally overly-refined meshes, and in designing multi-level solvers \cite{dg_cfes_2012,BassiJCP2012}. For instance, this result can be viewed as the theoretical justification for the numerical experiments in  \cite{antonietti2015review,antonietti2014multigrid}.

 Furthermore, under a \emph{space-time shape-regularity} assumption, $hp$-a priori error bounds are proven in the broken $L_2(H^1)$-- and  $L_{2}(L_2)$--norms, combining the classical duality approach with careful use of approximation arguments to circumvent the fundamental impossibility to apply `tensor-product' arguments (as is standard in this context \cite{thomee1984galerkin}) in the present setting. Instead, a new argument, based on judicious use of the space-time local degrees of freedom, eventually delivers the $L_2(H^1)$--norm and $L_{2}(L_2)$--norm error bound, with constants independent of number of faces per element.

 The remainder of this work is structured as follows.
In Section~\ref{mod}, we introduce the model problem and define the set of admissible subdivisions of the space-time computational domain while the new space-time dG method is formulated in
Section~\ref{DGFEM_ell_sec}.  In Section~\ref{Stability}, we prove an inf-sup  condition  for the dG scheme.  Section~\ref{Analysis} is devoted to the a priori error analysis. The practical performance of the new space-time dG method is studied through a series of numerical examples in Section \ref{numerics}, where extensive comparison among different combinations of the spatial and temporal discretizations and the new approach are given.


\section{Problem and method}\label{mod}
For a Lipschitz domain $\omega \subset {\mathbb R}^d$, $d=2,3$,
we denote by $H^s(\omega)$ the Hilbertian Sobolev space of  index $s\ge 0$ of real--valued functions defined on
$\omega$, with seminorm $|\cdot |_{H^s(\omega)}$ and norm $\|\cdot\|_{H^s(\omega)}$. For $s=0$, we have $H^0(\omega)\equiv L_2(\omega)$ with inner product $(\cdot, \cdot)_\omega$ and induced norm $\|\cdot\|_{\omega}$; when $\omega=\Omega$, the problem domain, we shall drop the subscript and write $(\cdot, \cdot)$ and  $\|\cdot\|$, respectively, for brevity.  We also let~$L_p(\omega)$, $p\in[1,\infty]$,
denote the standard Lebesgue space on $\omega$, equipped with the norm~$\|\cdot\|_{L_p(\omega)}$. Further, with $|\omega|$ we shall denote the $d$-dimensional Hausdorff measure of $\omega$. Standard Bochner spaces of functions which map a (time) interval $I$ to a Banach space $X$ will also be employed. $L_2(I;X)$ and  $H^s(I;X)$ are the corresponding Lebesgue and Sobolev spaces, while $C(\bar{I};X)$ denotes the space of continuous functions. 


\subsection{Model problem}\label{Parabolic problem}
Let $\Omega$ be a bounded open polyhedral domain in $\mathbb{R}^d$, $d=2,3$, and let $J:= (0,T)$ a time interval with $T>0$. We consider the linear parabolic problem:
\begin{equation}\label{Problem}
\begin{aligned} \partial_t u  - \nabla \cdot ({\bf a} \nabla u) &=f  \quad \text{in } J \times  \Omega , \\
u|_{t=0} = u_0 \quad \text{on }  \Omega, \quad &\text{and} \quad u = g_{\rm D}  \quad \text{on }  J \times \partial\Omega,
\end{aligned}
\end{equation}
for $f\in L_2(J ; L_2(\Omega))$ and ${\bf a}\in L_\infty(J \times \Omega)^{d\times d}$, symmetric with
\begin{equation}\label{uniform ellipticity}
 \xi^\top {\bf a}(t,x) \xi \geq \theta |\xi|^2 >0 \quad  \forall \xi \in  \mathbb{R}^d, \quad \text{a.e.} \quad  (t,x)\in  J \times \Omega,
\end{equation}
for some constant $\theta>0$. Note that the differential operator $\nabla := (\partial_1,\partial_2,\cdots,\partial_d)$, i.e., is applied to the spatial variables only. For $u_0\in L_2(\Omega)$ and $g_{\rm D}=0$ the problem \eqref{Problem} is well-posed and there exists a unique solution $u \in L_2(J;H^1_0(\Omega))$ with  $u \in C(\bar{J};L_2(\Omega))$ and $\partial_t u \in L_2(J; H^{-1}(\Omega))$ \cite{ladyzhenska1988linear,lions2012non}.


\subsection{Finite element spaces}\label{FEM space}
Let $\mathcal{U}$ be a partition of the time interval $J$ into $N_t$ time steps $\{I_n \}_{n=1}^{N_t}$, with  $I_n = (t_{n-1},t_n)$  with respective set of nodes $\{t_n\}_{n=0}^{N_t}$ defined so that  $0:= t_0 <t_1<\dots <t_{N_t}:=T$. Set also set  $\tmesh:=t_{n} - t_{n-1}$, the length of $I_n$.

For the spatial mesh, we shall adopt the setting from \cite{cangiani2013hp,cangiani2015hp} (albeit with different assumptions on admissible meshes as we shall see below), with $\mathcal{T}$ being a
subdivision of spatial domain $\Omega$ into
disjoint open polygonal $(d=2)$ or polyhedral $(d=3)$ elements $\kappa$ such that $\bar{\Omega}=\cup_{\el}\bar{\kappa}$.
In the absence of hanging nodes/edges, we define the {\em interfaces} of the mesh $\mathcal{T}$ to be the set of
$(d-1)$-dimensional facets of the elements $\kappa\in \mathcal{T}$. To facilitate the presence of hanging nodes/edges, which are permitted in $\mathcal{T}$,
the interfaces of $\mathcal{T}$ are defined to be the intersection of the $(d-1)$-dimensional facets of neighbouring elements.
Hence, for $d=2$, the interfaces of a given element $\k\in \mathcal{T}$ will consist of line segments (one-dimensional simplices), while for $d=3$, we assume that each interface of an element $\k\in \mathcal{T}$ may be subdivided into a set of co-planar triangles (two-dimensional simplices). We shall, therefore, use the terminology `face' to refer to a $(d-1)$-dimensional simplex which forms part of the interface of an element $\k\in \mathcal{T}$. For $d=2$,  the face and interface of an element $\k\in \mathcal{T}$ necessarily coincide. We also assume that, for $d=3$, a sub-triangulation of each interface into faces is given.  We shall denoted by $\mathcal{E}$ the union of all open mesh faces, i.e., $\mathcal{E}$ consists of $(d-1)$-dimensional simplices.  Further, we write $\mathcal{E}_\text{int}$ and $\mathcal{E}_\text{D}$ to denote the union of all open $(d-1)$-dimensional element faces $F\in \mathcal{E}$ that are contained in $\Omega$ and in $\partial \Omega$, respectively. Let also $\Gamma_\text{int}:=\{x\in \Omega:x \in F , F\in \mathcal{E}_\text{int} \}$ and $\Gamma_\text{D}:=\{x\in \Omega:x \in F , F\in \mathcal{E}_\text{D} \}$, while $\Gamma := \Gamma_{\rm D}\cup \Gamma_\text{int}$.

 \begin{figure}[t]
\begin{center}
\includegraphics[scale=0.26]{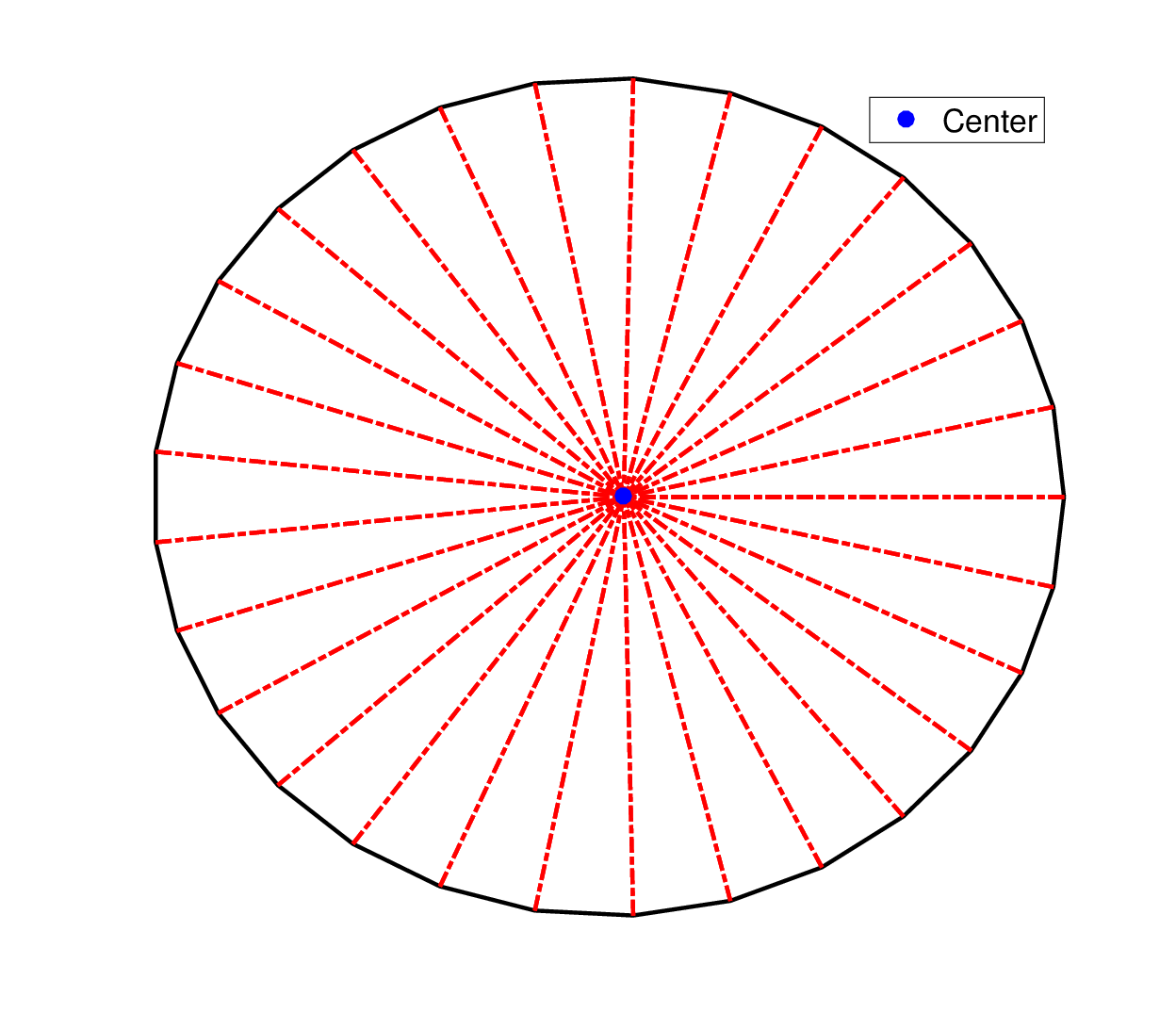}
\includegraphics[scale=0.26]{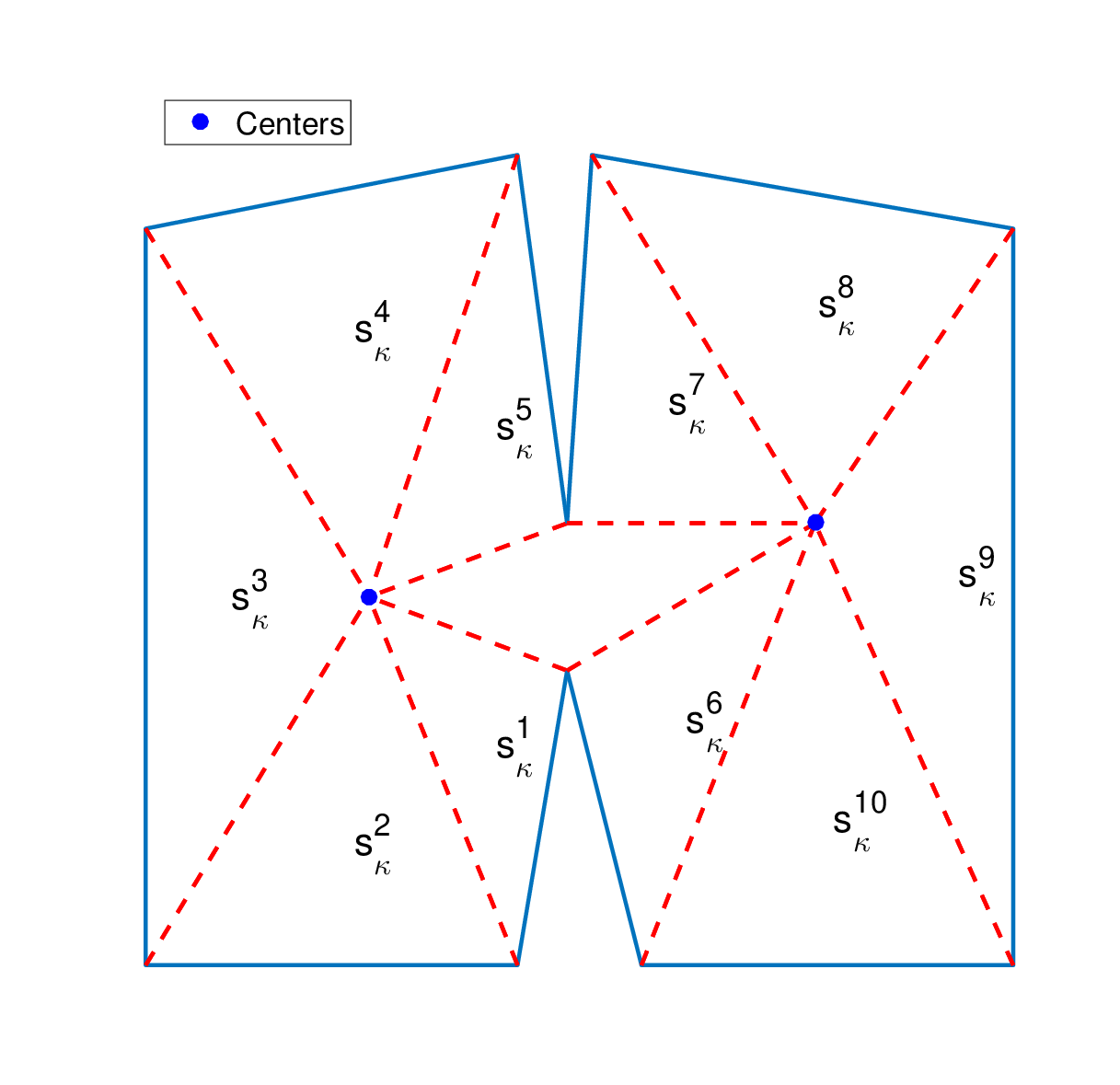}
\end{center}
\caption{ 30-gon with $\cup_{i=1}^{30} \bar{s}_\k^i = \bar{\kappa}$ (left); star shaped polygon with $\cup_{i=1}^{10} \bar{s}_\k^i \subsetneq \bar{\kappa}$ (right).}
  \label{polygons}
  \vspace{-0.5cm}
\end{figure}

\begin{assumption}[Spatial mesh] \label{A1}
For any $\kappa \in\mathcal{T}$, the element boundary $\partial \k$ can be sub-triangulated into non-overlapping $(d-1)$-dimensional simplices $\{F_\k^i\}_{i=1}^n$. Moreover, there exists a set of non-overlapping $d$-dimensional simplices $\{s_\k^i\}_{i=1}^n$ contained in $\kappa$, such that $\partial s_\k^i\cap \partial \k = F_k^i$, and
\begin{equation}\label{shape_relation1}
h_\k \leq C_s \frac{d|s_\k^i|}{|F_\k^i|},
\end{equation}
with $C_s>0$ constant independent of the discretization parameters, the number of faces per element, and the face measures.
\end{assumption}

\begin{remark}
Meshes made of polytopes which are finite union of polytopes with the latter being uniformly star-shaped with respect to the largest inscribed circle will satisfy Assumption \ref{A1}.
\end{remark}

In Figure~\ref{polygons}, we exemplify two different  polygons satisfying the above mesh regularity assumption.
We note that the assumption does not give any restrictions on neither the number nor the measure of the elemental faces.
Indeed, shape irregular simplices $s_\k^i$, with base $|F_\k^i|$ of small size compared to the corresponding height $d|s_\k^i|/|F_\k^i|$, are allowed: the height, however, has to be comparable to $h_\kappa$; cf., the left polygon on Figure~\ref{polygons}.
Further, we note that the union of the simplices $s_\k^i$ does not need to cover the whole element $\kappa$, as in general it is sufficient to assume that
\begin{equation} \label{shape_relation2}
\cup_{i=1}^N \bar{s}_\k^i \subseteq \bar{\kappa};
\end{equation}
cf., the right polygon on Figure~\ref{polygons}.
 In the following, we shall use $s_\k^F$ instead of  $s_k^i$ when no confusion is likely to occur.

\begin{figure}[t]
\begin{center}
\begin{tabular}{cc}
\vspace{-0.2cm}
\hspace{-1.2 cm} \includegraphics[scale=0.25]{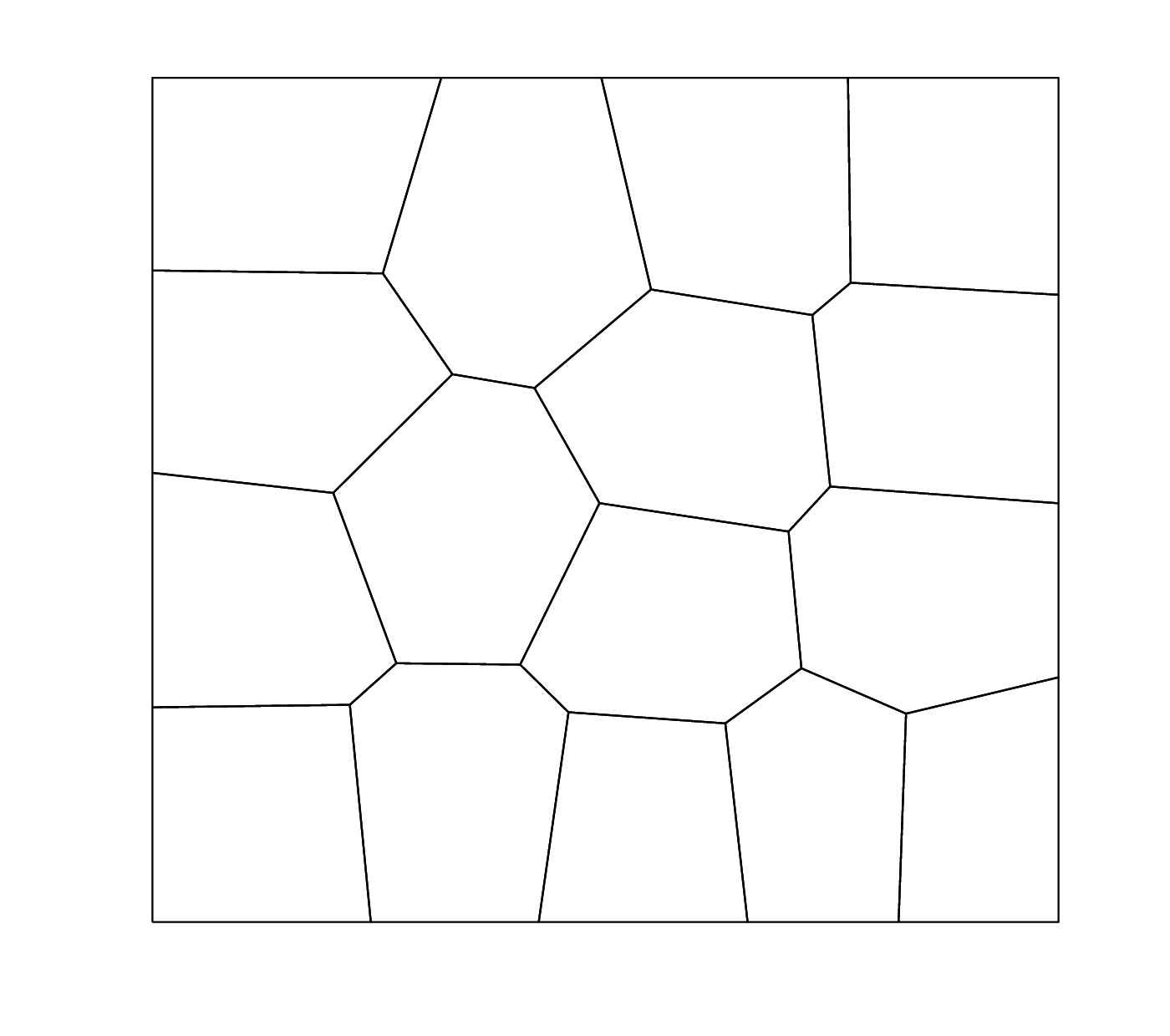} &
\hspace{-.5 cm} \includegraphics[scale=0.27]{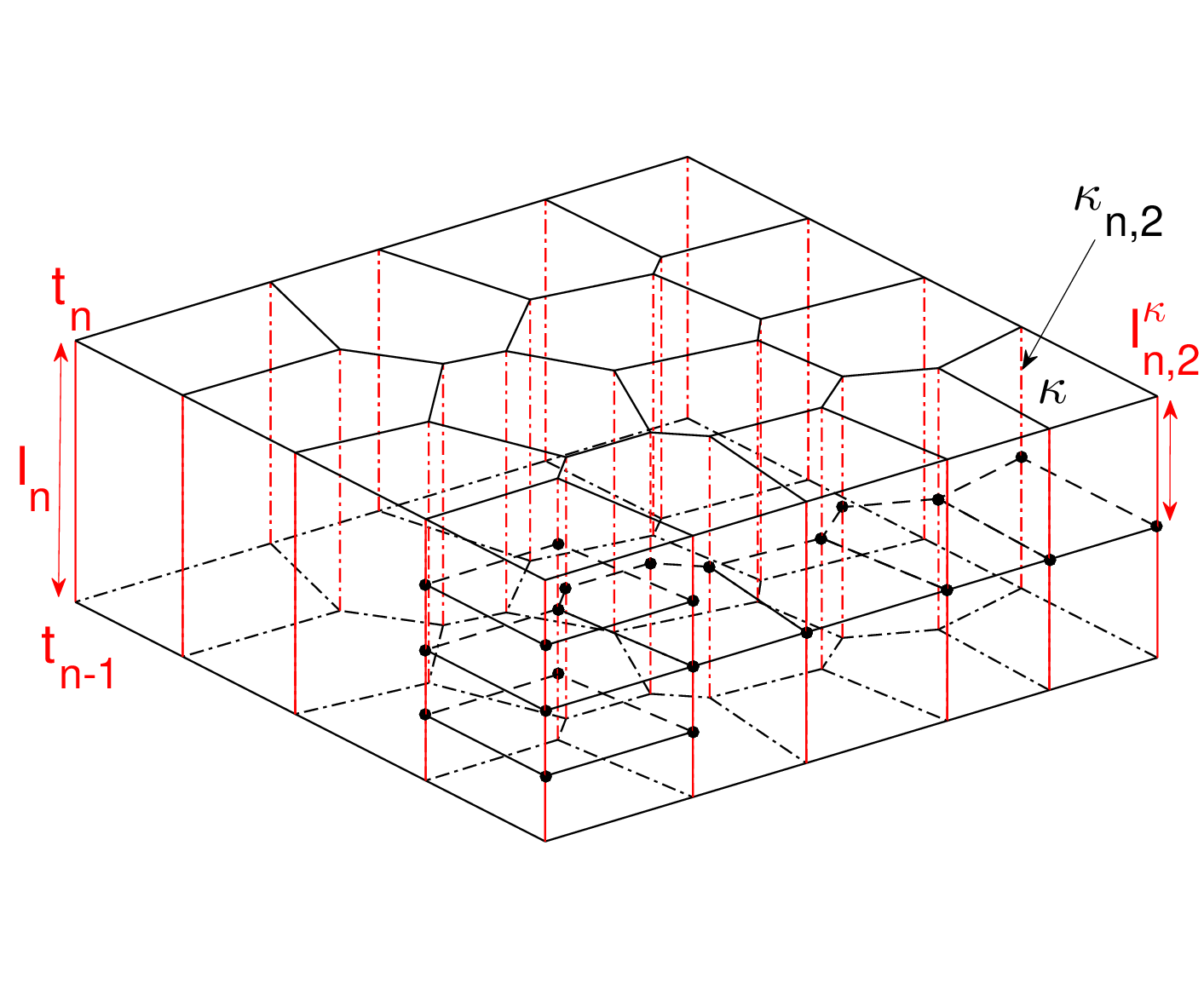}\\
\vspace{-0.2cm}(a) & (b)
\end{tabular}
\end{center}
\caption{(a). $16$ polygonal spatial elements over the spatial domain $\Omega =(0,1)^2$; (b) space-time elements over $I_n \times \Omega$ under the local time  partition $\mathcal{U}_n(\mathcal{T})$.} \label{spatial_temporal_elem}
\vspace{-0.4cm}
\end{figure}
The space-time mesh $\mathcal{U}\times \mathcal{T}$ is allowed to include locally smaller time-steps as follows.
Over each time interval $I_n$, $n=1,\dots, N_t$, we may consider the local time partition $\timesub$ that, for each space element $\el$, yields a subdivision of the time interval $I_n$  into $N_{n}^{\k}$ local time steps $I_{n,j}^\k = (t_{n,j-1}, t_{n,j})$, $j=1,\dots,N_{n}^{\k}$, with respect to the local time nodes $\{ t_{n,j}\}_{j=0}^{N_{n}^{\k}}$, defined so that  $t_{n-1}:= t_{n, 0} <t_{n,1}<\dots <t_{n,N_{n}^{\k}}:=t_{n}$. Further, we set  $\tsubmesh{j}:=t_{n,j} - t_{n,j-1}$ to be the length of $I_{n,j}^\k$.

For every time interval $I_n \in\mathcal{U}$ and every space element $\el$, with local time partition $\timesub$, we define the $(d+1)$-dimensional space-time \emph{prismatic} element $\kappa_{n,j} := I_{n,j}^\k \times \kappa$; see Figure \ref{spatial_temporal_elem} for an illustration. Let $p_{\kappa_{n,j}}$ denote the (positive) \emph{polynomial degree} of the space-time element $\kappa_{n,j}$, and collect $p_{\kappa_{n,j}}$ in
the vector ${\bf p}:=(p_{\kappa_{n,j}}:  \kappa_{n,j} \in \timesub\times \mathcal{T})$. We define the \emph{space-time finite element space} with respect to the time interval $I_n$, subdivision $\mathcal{T}$,  local time partition $\timesub $, and  polynomial degree ${\bf p}$ by
\[
\stnfessub{n}:=\{u\in L_2(I_n \times \Omega)
:u|_{\kappa_{n,j}}\in\mathcal{P}_{p_{\kappa_{n,j}}}(\kappa_{n,j}),  \kappa_{n,j} \in \timesub\times \mathcal{T}\},
\]
where
$\mathcal{P}_{p_{\kappa_{n,j}}}(\kappa_{n,j})$ denotes the space of polynomials of \emph{total degree} $p_{\kappa_{n,j}}$ on $\kappa_{n,j}$. The space-time finite element space $\stfes$ with
respect to $\mathcal{U}$, $\mathcal{T}$, ${\bf p}$, and, implicitly, $\timesub$, is defined as  $\stfes = \bigoplus_{n=1}^{N_t} \stnfessub{n}$.
As is standard in this context of local time-stepping, the resulting dG method is implicit with respect to all the local time-steps within the same time-interval $I_n$.

Note that the local elemental polynomial spaces employed in the definition
of $\stfes$ are defined in the \emph{physical coordinate system}, without the need to map from a given reference/canonical frame; cf.~\cite{cangiani2013hp}. This setting is crucial to retain full approximation of the finite element space, independently of the element shape. Note that $\stfes$ employs fewer degrees of freedom per space-time element as compared to  tensor-product polynomial bases of the same order in space and time.

We shall also make use of the broken space-time Sobolev space $H^\bold{l}(J \times \Omega,\mathcal{U};\mathcal{T})$,
up to composite order $\bold{l}:=(l_{\kappa_{n,j}}:  \kappa_{n,j} \in \stelemsub, n=1,\dots, N_t )$ defined by
\begin{equation}\label{brokenSob}
H^\bold{l} (J \times \Omega, \mathcal{U};\mathcal{T})=\{ u\in L_2(J \times \Omega):u|_{\kappa_{n,j}} \in H^{l_{\kappa_{n,j}}}(\kappa_{n,j}) , \kappa_{n,j} \in \stelemsub \}.
\end{equation}
Let $h_{\kappa_{n,j}}$ denote the diameter of the space-time element $\kappa_{n,j}$; for convenience, we collect the $h_{\kappa_{n,j}}$ in the vector ${\bf h}:=(h_{\kappa_{n,j}}:  \kappa_{n,j} \in \stelemsub, n=1,\dots, N_t )$. Moreover, we define the broken Sobolev space $H^1 ( \Omega, \mathcal{T})$ with respect to the subdivision $\mathcal{T}$  as follow
\begin{equation}\label{brokenSob-space}
H^1 ( \Omega, \mathcal{T})=\{ u\in L_2( \Omega):u|_{\kappa} \in H^{1}(\kappa) , \kappa \in \mathcal{T} \}.
\end{equation}
For $u \in  H^1 ( \Omega, \mathcal{T})$, we define  the broken spatial gradient $(\nabla_h u)|_\kappa=\nabla(u|_\kappa)$, $\kappa\in \mathcal{T}$, which will be used to
construct the forthcoming dG method.


\subsection{Trace operators}\label{Trace operator}

We denote by $\stF$ a generic $d$-dimensional face of a space-time element $\kappa_{n,j} \in \stelemsub$, which should be distinguished from the $(d-1)$-dimensional face $F$ of the spatial element $\kappa \in \mathcal{T}$. For any space-time element $\kappa_{n,j} \in \stelemsub$, we define $\partial\kappa_{n,j}$ to be the union of all $d$-dimensional open faces $\stF$ of $\kappa_{n,j}$. For convenience, we further subdivide $\stF$ into two disjoint subsets
\begin{equation}\label{space-time-face}
\stF^\parallel :=\stF \subset I_{n,j}^\k\times \partial \k, \quad \text{and } \quad  \stF^\perp :=\stF \subset \partial I_{n,j}^\k \times \k,
\end{equation}
i.e., the parallel and perpendicular to the time direction boundaries, respectively. Hence, for each $\kappa_{n,j}$, there exist exactly two $d$-dimensional faces $\stF^\perp$ and the number of $d$-dimensional faces $\stF^\parallel$ is equal to the number of $(d-1)$-dimensional spatial faces $F$ of the spatial element $\kappa$.

Let $\kappa_{n,j_1}^1$ and $\kappa^2_{n,j_2}$ be two
adjacent space-time elements sharing a face $\stF^\parallel =\partial \kappa_{n,j_1}^1 \cap \partial \kappa_{n,j_2}^2$ and $(t,x)\in \stF^\parallel \subset J\times \Gamma_{\rm int}$; let also $\bar{\bold{n}}_{\kappa_n^1}$ and $\bar{\bold{n}}_{\kappa_n^2}$ denote the
outward unit normal vectors on $\stF^\parallel$, relative to $\partial\kappa_{n,j_1}^1$ and $\partial\kappa_{n,j_2}^2$, respectively.
Then, for $v$ and $\bold{q}$,  scalar- and vector-valued functions, respectively, smooth enough for their traces on $\stF^\parallel$ to be well defined, we define the  \emph{averages}
$\mean{v}:=\frac{1}{2}(v|_{\kappa_{n,j_1}^1}+v|_{\kappa_{n,j_2}^2})$, $\mean{\bold{q}}:=\frac{1}{2}(\bold{q}|_{\kappa_{n,j_1}^1}+\bold{q}|_{\kappa_{n,j_2}^2} ),$
and the \emph{jumps}
$
\jump{v} :=v|_{\kappa_{n,j_1}^1}\,{\bar{\bold{n}}}_{\kappa_n^1}+v|_{\kappa_{n,j_2}^2}\,{\bar{\bold{n}}}_{\kappa_n^2},$ $\jump{\bold{q}} := \bold{q}|_{\kappa_{n,j_1}^1}\cdot \bar{\bold{n}}_{\kappa_n^1} +\bold{q}|_{\kappa_{n,j_2}^2}\cdot \bar{\bold{n}}_{\kappa_n^2},
$
respectively.
On a boundary face $\stF^\parallel \subset J\times \Gamma_{\rm D} $ and $\stF^\parallel  \subset \partial \kappa_{n,j}$,
we set
$
\mean{v}=v|_{\kappa_{n,j}}$,  $ \mean{\bold{q}}=\bold{q}|_{\kappa_{n,j}}$, $ \jump{v}=v|_{\kappa_{n,j}}{\bar{\bold{n}}}_{\kappa_{n,j}}$, $\jump{\bold{q}}=\bold{q}|_{\kappa_{n,j}}\cdot {\bar{\bold{n}}}_{\kappa_{n,j}},
$
with~$\bold{n}_{\kappa_n}$ denoting the unit outward normal
vector on the boundary. Upon defining
\[
u_n^+ := \lim_{s\rightarrow 0^+} u(t_n+s), \ 0\leq n \leq N_t-1, \quad
u_n^- := \lim_{s\rightarrow 0^+} u(t_n-s), \ 1\leq n \leq N_t,
\]
the \emph{time-jump} across $t_{n}$, $n=1, \dots,N_t-1$ is given by
$\ujump{u}_{n}  := u_{n}^+-u_{n}^-$. Similarly, the \emph{time-jump} across the  interior time nodes  $t_{n,j}$, $j=1, \dots, N_{n}^\k-1$,  $n=1, \dots,N_t$ is given by
$\ujump{u}_{n,j}  := u_{n,j}^+-u_{n,j}^-.$


\section{Space-time dG method}\label{DGFEM_ell_sec}

Equipped with the above notation, we can now describe the space-time discontinuous Galerkin method for the problem \eqref{Problem}, reading: find $u_h\in \stfes$ such that
\begin{equation}\label{adv_galerkin_dg}
B(u_h,v_h)=\ell(v_h) , \quad\text{ for all }v_h\in \stfes,
\end{equation}
where $B:\stfes \times \stfes \to \mathbb{R}$ is defined as
\begin{eqnarray}\label{new-bilinear}
B(u,v) :=&& \sum_{n=1}^{N_t} \int_{I_n}  \big(  ( \partial_t u, v) + a(u,v) \big)  \ud  t + \sum_{n=2}^{N_t}(\ujump{u}_{n-1},v^+_{n-1})+(u^+_{0},v^+_{0}) \\
&+& \su  \sum_{n=1}^{N_t} \sum_{j=2}^{N_n^\k}(\ujump{u}_{n,j-1},v^+_{n,j-1})_\k  ,\nno
\end{eqnarray}
with the spatial bilinear form $a(\cdot, \cdot)$ given by
\[
a(u,v) := \su \int_\kappa{\bf a}\nabla u \cdot \nabla v   \ud x - \int_{ \Gamma } \Big( \mean{{\bf {a}} \nabla u }\cdot \jump{v} + \mean{{\bf a} \nabla v }\cdot \jump{u}
 -   \sigma \jump{u} \cdot \jump{v}  \Big)  \ud s,
\]
and the linear functional $\ell:\stfes \to \mathbb{R}$ given by
\[
\ell(v)
:= \sum_{n=1}^{N_t} \int_{I_n}  \Big( (f , v)
-\int_{\Gamma_{\rm D}} g_{\rm D} \Big( ({\bf a} \nabla_h v ) \cdot \bold{n} -\sigma v \Big)\ud s  \Big) \ud t +(u_0,v^+_{0}).
\]
The non-negative function $\sigma \in L_\infty(J\times\Gamma)$ appearing in $a$ and $\ell$ above is referred to as the \emph{discontinuity-penalization parameter}; its precise definition, depending on the diffusion tensor ${\bf a}$ and on the discretization parameters, will be given in Lemma~\ref{norm-relation} below.

The use of prismatic meshes is key in that it permits to solve for each time-step separately: for each time interval $I_n \in \mathcal{U}$, $n=2,\dots,N_t$, the solution $U_n  = u_h |_{I_n}  \in \stnfessub{n}$ is given by:
\begin{eqnarray}\label{timestep-bilinear}
&& \int_{I_n}    ( \partial_t U_n, V_n)  + a(U_n,V_n)   \ud  t + (U^+_{n-1},V^+_{n-1})
+ \su  \sum_{j=2}^{N_n^\k}(\ujump{u}_{n,j-1},v^+_{n,j-1})_\k
\\
&&~~~= \int_{I_n}  \Big( (f , V_n)
-\int_{\Gamma_{\rm D}} g_{\rm D} \big( ({\bf a} \nabla_h V_n ) \cdot \bold{n} -\sigma V_n \big)\ud s  \Big) \ud t +(U^-_{n-1},V^+_{n-1}), \nno
 \end{eqnarray}
 for all $V_n\in \stnfessub{n}$,
with $U_{n-1}^-$  serving as the initial datum at time step $I_n$; for $n=1$, we set $U_0^- = u_0$.


In the interest of simplicity of the presentation, we shall not explicitly carry through the local time-steps in the stability and the a-priori error analysis that follows. Indeed, the general case including local time partitions can be derived by slightly modifying the analysis in a straightforward fashion. Removing the local time-step notation, the last term in the dG bilinear form~\eqref{new-bilinear} and the last term on the left hand side of \eqref{timestep-bilinear} vanish.

\section{Stability}\label{Stability}
We shall establish the unconditional stability of the above space-time dG method, via the derivation of an inf-sup condition for arbitrary aspect ratio between the time-step and the local spatial mesh-size. The proof circumvents the global shape-regularity assumption, required in the respective result from \cite[{Theorem} 5.1]{cangiani2015hp} for the case of parabolic problems, and also removes the assumption of uniformly bounded number of faces per element.


\subsection{Inverse estimates}\label{Inverse-estimation}
We review some $hp$--version inverse estimates.

\begin{lemma}\label{lemmal_inv}
Let $\kappa_n \in \stelem$, and let Assumption \ref{A1} to hold. Then, for each $v\in\mathcal{P}_{p_{\kappa_n}}(\kappa_n)$,
we have
\begin{equation}\label{inv_est_gen}
\norm{v}{\stF^\parallel}^2 \le
\frac{(p_{\kappa_n}+1)(p_{\kappa_n}+d)}{d}\frac{|F|}{|s^F_\kappa|}\norm{v}{L_2(I_n; L_2(s^F_\kappa))}^2,
\end{equation}
with $\stF^\parallel = I_n\times F$, $F\subset\partial \kappa\cap \partial s^F_\k$ and $s^F_\k$ as in Assumption \ref{A1} sharing $F$ with $\kappa$.
\end{lemma}
\begin{proof}
The proof follows from the tensor-product structure of $\kappa_n = I_n\times \kappa$, together with \cite[Theorem 3]{warburton2003constants}.
\end{proof}
\begin{lemma}\label{space-time-lemma_inv}
Let $v \in\mathcal{P}_{p_{\kappa_n}}(\kappa_n)$, $\kappa_n  \in \stelem$ and $\mathcal{G}\in\{\kappa_n,\stF^\parallel\}$. Then, there exist positive constants $\stinvF$ and $\stinvb$, independent of $v$, $\kappa_n$, $\tmesh$ and $p_{\kappa_n}$, such that
\begin{equation}\label{anisotropic-trace-inv}
\ltwo{v}{\stF^\perp}^2 \leq \stinvF \frac{p_{\kappa_n}^2}{ \tmesh} \ltwo{v}{\kappa_n}^2,
\end{equation}
\vspace{-0.6cm}
\begin{equation}\label{anisotropic-H-1-kappa}
\ltwo{\partial_t v}{\mathcal{G}}^2 \leq \stinvb \frac{p_{\kappa_n}^4}{\tmesh^2} \ltwo{v}{\mathcal{G}}^2.
\end{equation}
\end{lemma}
\begin{proof}
 The proofs are immediate upon observing the prismatic (tensor-product) structure of $\kappa_n = I_n\times \kappa$; see, e.g., \cite{thesis,georgoulis2007discontinuous,schotzau2013hp} for similar arguments.
\end{proof}

\subsection{Coercivity and continuity}
For the error analysis, we introduce an inconsistent bilinear form $\tilde{a}(\cdot,\cdot)$, (cf.~\cite{perugia_schotzau}): for $u,v\in\mathcal{S}:= L_2(J ;H^1(\Omega))\cap H^1(J ; H^{-1}(\Omega))  +\stfes$, we set
\begin{eqnarray}\label{in-bilinear}
\tilde{B}(u,v) &:=& \sum_{n=1}^{N_t} \int_{I_n}    \big( ( \partial_t u, v) + \tilde{a}(u,v)  \big) \ud  t + \sum_{n=2}^{N_t}(\ujump{u}_{n-1},v^+_{n-1})+(u^+_{0},v^+_{0}),
\end{eqnarray}
where
\[\tilde{a}(u,v) := \su \int_\kappa{\bf a}\nabla u \cdot \nabla v   \ud x - \int_{ \Gamma } \Big( \mean{{\bf {a}} \vecL(\nabla u) }\cdot \jump{v} + \mean{{\bf a} \vecL(\nabla v)}\cdot \jump{u}  -  \sigma \jump{u} \cdot \jump{v} \Big)   \ud s,
\]
and a modified linear functional $\tilde{\ell}:\mathcal{S}\to\mathbb{R}$, given by
\[
\tilde{\ell}(v)
:= \sum_{n=1}^{N_t} \int_{I_n}  \Big( (f , v)
-\int_{\Gamma_{\rm D}} g_{\rm D} \Big( {\bf a} \vecL(\nabla_h v ) \cdot \bold{n} -\sigma v \Big)\ud s  \Big) \ud t +(u_0,v^+_{0});
\]
here, $\vecL:[L_2(J; L_2 (\Omega))]^d \rightarrow [\stfes]^d$ denotes the vector-valued $L_2$--projection onto  $[ \stfes]^d$. It is immediately clear, therefore, that
$B(u_h,v_h)=\tilde{B}(u_h,v_h) $ and that  $l(v_h) = \tilde{l}(v_h)$, for all $ v_h\in \stfes.$

Let $\sqrt{{\bf a}}$ be the square root of ${\bf a}$ and set $\bar{\bold{a}}_{\kappa_n} = |\sqrt{{\bf a}} |_2^2|_{\kappa_n}$, for $\kappa_n\in \stelem$, with $|\cdot |_2$ denoting the matrix $l_2$--norm. We introduce the dG--norm $\ncdg{\cdot}$:
\begin{equation} \label{DG_norm}
\ncdg{v} := \Big(  \int_{J}\ndg{v}_{\rm d}^2 \ud t + \frac{1}{2} \ltwo{v^+_0}{}^2+ \sum_{n=1}^{N_t-1} \frac{1}{2} \ltwo{\ujump{v}_n}{}^2+ \frac{1}{2} \ltwo{v^-_{N_t}}{}^2\Big)^{1/2},
\end{equation}
with $\ndg{v}_{\rm d}: = \big(\ltwo{\sqrt{\bf a} \nabla_h v}{} ^2 +\ltwo{\sqrt{\sigma} \jump{v}}{\Gamma}^2 \big)^{1/2}.$

\begin{lemma}\label{norm-relation}
Let  Assumption~\ref{A1} hold and let $\sigma: J\times \Gamma \to\mathbb{R}_+$ be defined face-wise over all $\stF^\parallel $ by
\begin{equation}
\sigma({t, x}) :=
\begin{array}{ll}
C_{\sigma}{\displaystyle   \max_{\kappa_n: {\stF^\parallel }\cap \bar{\kappa}_n\neq \emptyset}\Big\{
\frac{\bar{\bold{a}}_{\kappa_n}^2 (p_{\kappa_n}+1)(p_{\kappa_n}+d) }{h_{\k}} \Big\} }, &  \stF^\parallel  \subset J \times \Gamma_{\rm}
\end{array}, \label{eq:IPdef}
\end{equation}
with $C_{\sigma}>0$ sufficiently large, independent of discretization parameters and the number of faces per element.  Then, for all $v\in \mathcal{S}$, we have
\begin{equation}\label{coer}
\int_{J} \tilde{a}(v,v)\ud{t}  \geq C_{\rm d}^{\rm coer} \int_{J} \ndg{v}_{\rm d} ^2 \ud{t},
\end{equation}
\vspace{-0.4cm}
\begin{equation}\label{cont}
\int_{J}  \tilde{a}(w,v)\ud{t}  \le C_{\rm d}^{\rm cont} \int_{J} \ndg{w}_{\rm d}\,\ndg{v} _{\rm d} \ud{t},
\end{equation}
\vspace{-0.4cm}
\begin{equation}\label{bilinear-DG-norm}
\tilde{B}(v,v)\geq  \bar{C}\ncdg{v}^2,
\end{equation}
 for all $v\in \mathcal{S}$, with the positive constants $C_{\rm d}^{\rm coer}$, $C_{\rm d}^{\rm cont}$ and  $\bar{C}$, independent of the discretization parameters, the number of faces per element, and of $v$.
\end{lemma}
\begin{proof}
The proof of these inequalities is standard (see, e.g., \cite{DiPietroErn}) for the most part and, thus, we shall focus on the treatment of the face terms in view of the new definition of the discontinuity penalization parameter \eqref{eq:IPdef}. To this end, using  \eqref{inv_est_gen}, the stability of the $L_2$-projection $\mathbf{\Pi}_2$, the uniform ellipticity \eqref{uniform ellipticity} together with \eqref{shape_relation1} and \eqref{shape_relation2}, we have
 we deduce that
\begin{eqnarray}\label{coercivity-3}
&&\int_{J} \int_{\Gamma}\mean{\mathbf{a\Pi}_2(\nabla_h v)}\cdot\jump{v}\ud s  \ud{t} \nonumber  \\
&&\leq
\epsilon \sum_{\k_n \in \stelem}\sum_{ \stF^{\parallel} \subset \partial \k_n}
\sigma^{-1}  \bar{\bold{a}}_{\kappa_n}^2 \frac{(p_{\kappa_n}+1)(p_{\kappa_n}+d)}{d}\frac{|F|}{|s^F_\kappa|}\norm{\mathbf{\Pi}_2 \nabla v}{L_2(I_n; L_2(s^F_\kappa))}^2     \nonumber \\
&&+  \frac{1}{4\epsilon}  \sum_{ \stF^{\parallel} \subset J\times \Gamma}    \ltwo{\sigma^{1/2} \jump{v}}{\stF^\parallel}^2   \nonumber \\
&&\leq
\frac{\epsilon C_s}{\theta C_\sigma } \sum_{\k_n \in \stelem}
\norm{ \sqrt{\bold{a}} \nabla v}{\k_n}^2
+  \frac{1}{4\epsilon}  \sum_{ \stF^{\parallel} \subset J\times \Gamma}    \ltwo{\sigma^{1/2} \jump{v}}{\stF^\parallel}^2.
\end{eqnarray}
Thereby, we deduce
\begin{eqnarray}
\int_{J} \tilde{a}(v,v) \ud{t}
&\geq& ( 1-\frac{2\epsilon C_s}{\theta C_\sigma}  )
\sum_{\k_n \in \stelem}   \norm{\sqrt{\bold{a}}  \nabla v}{\k_n}^2   +(1 -\frac{1}{2\epsilon})  \int_{J} \int_{\Gamma}     \ltwo{\sigma^{1/2} \jump{v}}{}^2 \ud{s}\ud{t}. \nonumber
\end{eqnarray}
Hence, the bilinear form $\tilde{a} (\cdot,\cdot)$ is coercive over ${\cal S} \times {\cal S}$  for
$\epsilon>1/2$ and $C_\sigma >  2 C_s\epsilon/ \theta$. $C_\sigma$ depends on  constant $C_s$, but is independent of number of faces per element.  The continuity relation \eqref{cont} can be proved by the similar arguments.
For \eqref{bilinear-DG-norm}, integration by parts on the first term on the right-hand side of \eqref{in-bilinear} along with (\ref{coer}) yield
\begin{eqnarray*}
\tilde{B}(v,v)
&\geq& C_{\rm d}^{\rm coer} \int_{J}   \ndg{v}_d^2 \ud t + \frac{1}{2} \ltwo{v^+_0}{}^2+ \sum_{n=1}^{N_t-1} \frac{1}{2} \ltwo{\ujump{v}_n}{}^2+ \frac{1}{2} \ltwo{v^-_{N_t}}{}^2 \geq \bar{C} \ncdg{v}^2,
\end{eqnarray*}
with $\bar{C}=\min\{1, C_{\rm d}^{\rm coer}\}$.  
\end{proof}

We stress that  the above  proof of coercivity of the elliptic part of the bilinear form follows different arguments from those used in  \cite{cangiani2013hp}. Our approach is dictated by the mesh regularity Assumption~\ref{A1} allowing for \emph{arbitrary} number of faces per element. In contrast, in \cite{cangiani2013hp}, no shape regularity was explicitly assumed at the expense of imposing a uniform bound on the number of faces per element. Clearly, the two approaches can be combined to produce admissible discretisations on even more general mesh settings; we refrain from doing so here in the interest of brevity and we refer to the forthcoming \cite{book} for the complete treatment. Nonetheless, all convergence and stability theory presented in this work is also valid under the mesh assumptions from \cite{cangiani2013hp}.

\begin{remark}
The coercivity constant may depend on the shape regularity constant $C_s$ and on the uniform ellipticity constant $\theta$. To avoid the dependence on the latter, it is possible to combine the present developments with the dG method proposed in \cite{GL05}; we refrain from doing so here, in the interest of simplicity of the presentation.
\end{remark}

\subsection{Inf-sup condition}\label{inf-sup}

 We shall now prove an inf-sup condition for the space-time dG method for a stronger \emph{streamline-diffusion} norm given by
\begin{equation}\label{SDG-norm}
\nsdg{v}^2:=\ncdg{v}^2+\sum_{\kappa_n \in \stelem} \stp \ltwo{ \partial_t {v}}{\kappa_n}^2,
\end{equation}
where $\stp:= \tmesh/\hat{p}_{\kappa_n}^2$,
for $p_{\kappa_n} \geq 1$ and $\hat{p}_{\kappa_n}$ defined as
\begin{equation}
\nonumber
\hat{p}_{\kappa_n}
:= \max_{\stF^\parallel \subset \partial\kappa_n} \Big\{ {\displaystyle
\max_{ \substack{\tilde{\kappa}_n \in\{\kappa_n ,\kappa^\prime_n\} \\
\stF^\parallel \subset  \partial\kappa_n \cap\partial \kappa^\prime_n  } }
\{ p_{\tilde{\kappa}_n}\} }
\Big\} \quad   \forall \kappa_n \in \stelem.
\end{equation}

\begin{theorem} \label{infsup  theorem}
Given Assumption \ref{A1}, there exists a constant $\Lambda_s>0$, independent of the temporal and spatial mesh sizes $\tmesh$, $h_\kappa$, of the polynomial degree $p_{\kappa_n}$ and of the number of faces per element, such that:
\begin{equation}\label{Inf-sup}
\inf_{{\nu\in \stfes\backslash \{0\}}} \sup_{{\mu\in \stfes
\backslash \{0\}}} \frac{\tilde{B}(\nu,\mu)}{\nsdg{\nu} \nsdg{\mu}}\geq \Lambda_s.
\end{equation}
\end{theorem}
\begin{proof}
For $\nu \in \stfes$, we select $\mu :=\nu+\alpha \nu_s$, with
$\nu_s|_{\kappa_n} := \stp\partial_t \nu$, $\kappa_n \in \stelem$, with
$0<\alpha\in\mathbb{R}$, at our disposal.
Then, \eqref{Inf-sup} follows if both the following:
\begin{equation}\label{upper}
\nsdg{\mu} \leq C^* \nsdg{\nu},
\end{equation}
and
\begin{equation}\label{lower}
\tilde{B}(\nu,\mu)\geq C_*\nsdg{\nu}^2,
\end{equation}
hold, with $C^*>0$ and $C_*>0$ constants independent of $h_\kappa$, $\tmesh$, $p_{\kappa_n}$, number of faces per element, and $\Lambda_s =C_*/C^* $.

To show~\eqref{upper}, we start by considering the jump terms at time nodes $\{t_n\}_{n=0}^{N_t}$. Employing \eqref{anisotropic-trace-inv}, we have
\begin{eqnarray}
 && \hspace*{2cm}  \frac{1}{2} \ltwo{(\nu^+_s)_0}{\Omega}^2+ \sum_{n=1}^{N_t-1} \frac{1}{2} \ltwo{\ujump{\nu_s}_n}{\Omega}^2+ \frac{1}{2} \ltwo{(\nu^-_s)_{N_t}}{\Omega}^2 \qquad  \\
 &\leq&\sum_{\kappa_n \in \stelem} \stp^2 \sum_{\stF^\perp \subset \partial \kappa_n }  \ltwo{\partial_t \nu}{\stF^\perp}^2  \leq  \sum_{\kappa_n \in \stelem} 2 \stinvF  \frac{ \stp p_{\kappa_n}^2}{\tmesh}  \Big( \stp \ltwo{\partial_t \nu}{\kappa_n}^2 \Big)  \leq C_1 \nsdg{\nu}^2  \nonumber.
\end{eqnarray}
Using \eqref{anisotropic-H-1-kappa} with $\mathcal{G}=\kappa_n$, the second term on the right-hand side of \eqref{SDG-norm} is estimated by
\begin{eqnarray}
\sum_{\kappa_n\in \stelem} \stp \ltwo{\partial_t \nu_s}{\kappa_n}^2 \leq \sum_{\kappa_n\in \stelem} \stinvb  \frac{\stp^2 p_{\kappa_n}^4}{\tmesh^2}  \Big( \stp \ltwo{\partial_t \nu}{\kappa_n}^2 \Big)  \leq C_2 \nsdg{\nu}^2  .
\end{eqnarray}
Next, for the first term on the right-hand side of \eqref{SDG-norm}, employing \eqref{anisotropic-H-1-kappa} with $\mathcal{G}=\kappa_n$, the uniform ellipticity condition \eqref{uniform ellipticity}, together with Fubini's theorem, we have
\begin{eqnarray}
 \sum_{\kappa_n\in \stelem}
 \ltwo{\sqrt{\bf{a}}\nabla \nu_s}{\kappa_n}^2  &\leq &
 \sum_{\kappa_n\in \stelem} \bar{\bold{a}}_{\kappa_n} \stp^2  \ltwo{\partial_t (\nabla \nu)}{\kappa_n}^2 \leq \sum_{\kappa_n\in \stelem} \bar{\bold{a}}_{\kappa_n}  \stinvb  \frac{\stp^2 p_{\kappa_n}^4}{\tmesh^2}   \ltwo{ \nabla \nu}{\kappa_n}^2  \nonumber \\
&\leq&   \sum_{\kappa_n\in \stelem}  \stinvb \frac{ \bar{\bold{a}}_{\kappa_n}}{\theta} \frac{ \stp^2 p_{\kappa_n}^4}{\tmesh^2}   \ltwo{ \sqrt{\bf{a}}\nabla \nu}{\kappa_n}^2
  \leq C_3 \nsdg{\nu}^2.
 \label{C3_diffusion}
\end{eqnarray}
Finally, employing \eqref{anisotropic-H-1-kappa} with $\mathcal{G}=\stF^\parallel$, we have
\begin{eqnarray}
\nonumber
\int_J\int_{\Gamma}  \sigma | \jump{\nu_s} |^2 \ud{s}  \ud{t}
&\leq& \!\! \sum_{\stF^\parallel \subset J\times  \Gamma}\!\!   \sigma   \stinvb\frac{ \stp^2
}{\tmesh^2}
\big({\displaystyle  \max_{ \substack{\tilde{\kappa}_n \in\{\kappa_n ,\kappa^\prime_n\} \\
\stF^\parallel \subset  \partial\kappa_n \cap\partial \kappa^\prime_n  } }\{ p_{\tilde{\kappa}_n}\} }\big)^4
\ltwo{ \jump{ \nu}}{\stF^\parallel}^2  \leq  C_4 \nsdg{\nu}^2.
\label{C4_diffusion}
\end{eqnarray}
 Combining the above, we have
$
 \nsdg{\nu_s}\leq \hat{C}\nsdg{\nu}
 $, where $\hat{C}=\sqrt{\sum_{i=1}^4C_i}$, or
\begin{equation} \label{first_part}
\nsdg{\mu}\leq \nsdg{\nu}+\alpha\nsdg{\nu_s} \leq (1+ \alpha\hat{C})\nsdg{\nu} \equiv C^*(\alpha)\nsdg{\nu}.
\end{equation}

For (\ref{lower}), we start by noting that
$\tilde{B}(\nu,\mu)= \tilde{B}(\nu,\nu)+\alpha \tilde{B}(\nu,\nu_s)$;
we observe
\[
\tilde{B}(\nu,\nu_s) = \sum_{n=1}^{N_t}\Big(\!\!\sum_{\kappa_n \in \stelem}\!\!\! \stp \ltwo{ \partial_t {v}}{\kappa_n}^2+\int_{I_n}  \!\! \tilde{a}(\nu,\nu_s)   \ud  t\Big) + \sum_{n=2}^{N_t}(\ujump{\nu}_{n-1},(\nu_s)^+_{n-1})+(\nu^+_{0},(\nu_s)^+_{0}).
\]
Further, using~\eqref{anisotropic-trace-inv} and the arithmetic mean inequality, we have
\begin{eqnarray}
\sum_{n=2}^{N_t}(\ujump{\nu}_{n-1},(\nu_s)^+_{n-1})+(\nu^+_{0},(\nu_s)^+_{0})
\leq
\sum_{\kappa_n\in \stelem} \ltwo{\ujump{\nu}}{\stF^\perp \subset\partial \kappa_n }
 \Big( \stp  \sum_{\stF^\perp \subset \partial \kappa_n }  \ltwo{\partial_t \nu}{\stF^\perp} \Big)  \nonumber \\
\leq   \sum_{\kappa_n\in \stelem} \ltwo{\ujump{\nu}}{\stF^\perp \subset\partial \kappa_n }
 (  2 \sqrt{\stinvF  \frac{\stp p_{\kappa_n}^2}{\tmesh} } )(   \stp^{1/2} \ltwo{\partial_t \nu}{\k_n} )    \nno \\
\label{relation1}
  \leq   \sum_{\kappa_n\in \stelem}
\frac{\stp}{4}\ltwo{\partial_t \nu}{\kappa_n}
^2 + 4\stinvF
\Big(\ltwo{\nu^+_0}{}^2+
\ltwo{\ujump{\nu}_n}{}^2+
 \ltwo{\nu^-_{N_t}}{}^2 \Big),
\end{eqnarray}
where, with slight abuse of notation, we have extended the definition of the time jump $\ujump{\nu}$ to time boundary faces.
Next, from \eqref{cont}, together
with \eqref{C3_diffusion} and \eqref{C4_diffusion}, we get
\begin{eqnarray}\label{relation2}
&&\sum_{n=1}^{N_t} \int_{I_n}   \tilde{a}(\nu,\nu_s)   \ud  t \leq \sum_{n=1}^{N_t} \int_{I_n}    C_{\rm d}^{\rm cont} \ndg{\nu}_{\rm d} \ndg{\nu_s}_{\rm d} \ud  t  \\
&\leq& \frac{(C_{\rm d}^{\rm cont})^2}{2}\int_{J}\ndg{\nu}_{\rm d}^2 \ud t + \frac{1}{2} \int_{J}\ndg{\nu_s}_{\rm d}^2 \ud t     \leq \frac{(C_{\rm d}^{\rm cont})^2 +C_3+C_4}{2} \int_{J}\ndg{\nu}_{\rm d}^2 \ud t. \nonumber
\end{eqnarray}
Combining \eqref{bilinear-DG-norm} with \eqref{relation1} and \eqref{relation2}, we arrive to
\begin{eqnarray*} \label{infsupcoercive}
\tilde{B}(\nu, \mu) &=&\tilde{B}(\nu, \nu) + \alpha \tilde{B}(\nu, \nu_s) \nonumber \\
&\geq &  \Big( \frac{1}{2}-4\alpha \stinvF  \Big)  \Big( \ltwo{\nu^+_0}{}^2+ \sum_{n=1}^{N_t-1}  \ltwo{\ujump{\nu}_n}{}^2+ \ltwo{\nu^-_{N_t}}{}^2  \Big)  \nonumber \\
&+&  \Big( C_{\rm d}^{\rm coer}-\alpha  \frac{(C_{\rm d}^{\rm cont})^2 +C_3+C_4}{2} \Big) \int_{J}\ndg{\nu}_{\rm d}^2 \ud t + \kern-.2cm \sum_{\kappa_n\in \stelem}\kern-.2cm\alpha\Big( \stp- \frac{\stp}{4}   \Big)\ltwo{\partial_t \nu}{\kappa_n}^2.
\end{eqnarray*}
The coefficients in front of the norms arising on the right hand side of the above bound are all positive if
$
\alpha < \min \{1/( 8  \stinvF) , 2 C_{\rm d}^{\rm coer}/((C_{\rm d}^{\rm cont})^2+C_3+C_4)\},$ with the latter independent of the discretization parameters and the number of faces per element.
\end{proof}

The above result shows that the space-time dG method based on the reduced \emph{total-degree-$p$} space-time basis is well posed. It extends the stability proof from \cite{cangiani2015hp} to space-time elements with arbitrarily large aspect ratio between the time-step $\tmesh$ and local mesh-size $h_\k$ for parabolic problems.
Moreover, the above inf-sup condition holds \emph{without any  assumptions on the number of faces per spatial mesh}, too.  Therefore, the scheme is shown to be stable for extremely general, possibly anisotropic, space-time meshes.

The above inf-sup condition will be instrumental in the proof of the a priori error bounds below, as the total-degree-$p$ space-time basis does not allow for classical space-time tensor-product arguments \cite{thomee1984galerkin} to be employed.
\begin{remark}
Crucially, the stability constant in  \eqref{Inf-sup}  is independent of both the temporal mesh size and polynomial degree. However, the respective continuity constant for the bilinear form would scale proportionally to $\tau_n^{-1}$. Nonetheless, this is of \emph{no} consequence in the a priori error analysis presented below.
\end{remark}


\section{A priori error analysis}\label{Analysis}


\subsection{Polynomial approximation}\label{Approximation}

In view of using known approximation results, we shall require a shape-regularity assumption for the space-time elements.

\begin{assumption}\label{space-time-shape regular} We assume the existence of a constant $c_{reg}>0$ such that
\[
c_{reg}^{-1}\le h_{\kappa}/\tmesh\le c_{reg},
\] uniformly for all $\kappa_n\in\mathcal{U}\times\mathcal{T}$, i.e., the space-time elements are also shape-regular.
\end{assumption}

Following \cite{cangiani2013hp,cangiani2015hp}, we assume the existence of certain spatial mesh coverings.
\begin{definition}\label{meshes2}
A \emph{covering} $\mathcal{T}_{\sharp} = \{ \mathcal{K} \}$ related to the polytopic mesh $\mathcal{T}$ is a
set of shape-regular $d$--simplices or hypercubes  $\mathcal{K}$, such that for each $\el$, there exists a $\mathcal{K}\in\mathcal{T}_{\sharp}$,
with $\kappa\subset\mathcal{K}$. We refer to Figure \ref{spatial_temporal_covering}(a) for an illustration.
Given $\mathcal{T}_{\sharp}$, we denote by $\Omega_{\sharp}$ the \emph{covering domain} given by
$\Omega_{\sharp}:=\left(\cup_{\mathcal{K}\in\mathcal{T}_{\sharp}}\bar{\mathcal{K}}\right)^{\circ}$,   with $D^{\circ}$ denoting the interior of a set  $D \subset {\mathbb R}^d$.
\end{definition}

\begin{assumption}\label{assumption_overlap}
There exists a covering $\mathcal{T}_{\sharp}$ of $\mathcal{T}$ and a positive constant ${\cal O}_{\Omega}$, independent of the mesh parameters, such that the subdivision $\mathcal{T}$ satisfies
$
\max_{\kappa \in {\cal T}} {\cal O}_\k \leq {\cal O}_{\Omega},
$
where, for each $\el$,
$$
{\cal O}_\k :=  \mbox{\rm card}
      \left\{ \kappa^\prime \in {\cal T} :
              \kappa^\prime \cap \mathcal{K} \neq \emptyset, ~\mathcal{K}\in\mathcal{T}_{\sharp} ~\mbox{ such that } ~\kappa\subset\mathcal{K} \right\} . \nno
$$
\end{assumption}
As a consequence, we deduce that
$
\diam(\mathcal{K})\le C_{\diam} h_{\kappa},
$
for each pair $\el$, $\mathcal{K}\in\mathcal{T}_{\sharp}$, with $\kappa\subset\mathcal{K}$, for a constant $C_{\diam}>0$, uniformly with respect to the mesh size.

\begin{figure}[t]
\begin{center}
\begin{tabular}{cc}
\hspace{-0 cm} \includegraphics[scale=0.23]{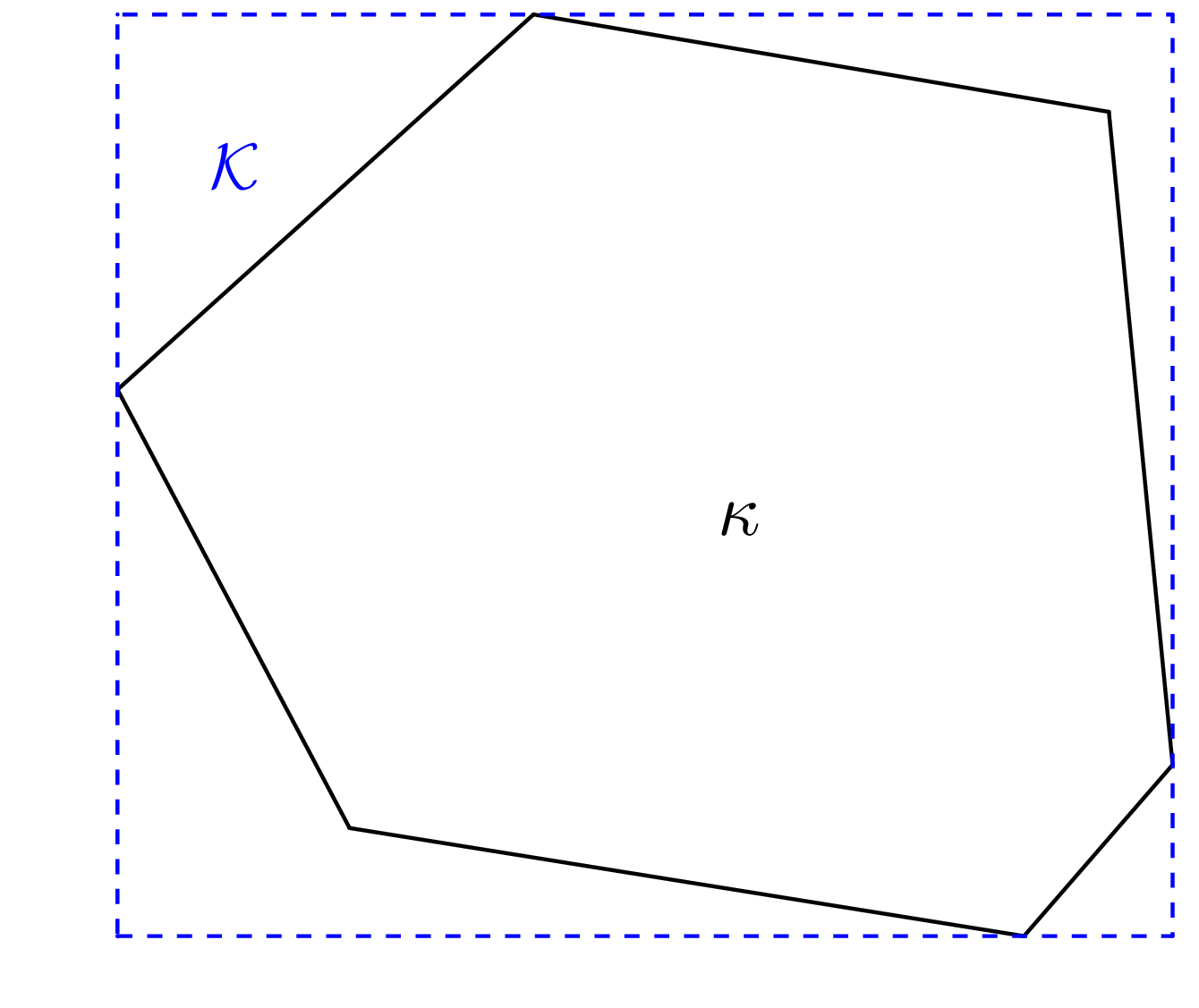} &
\hspace{-0 cm} \includegraphics[scale=0.24]{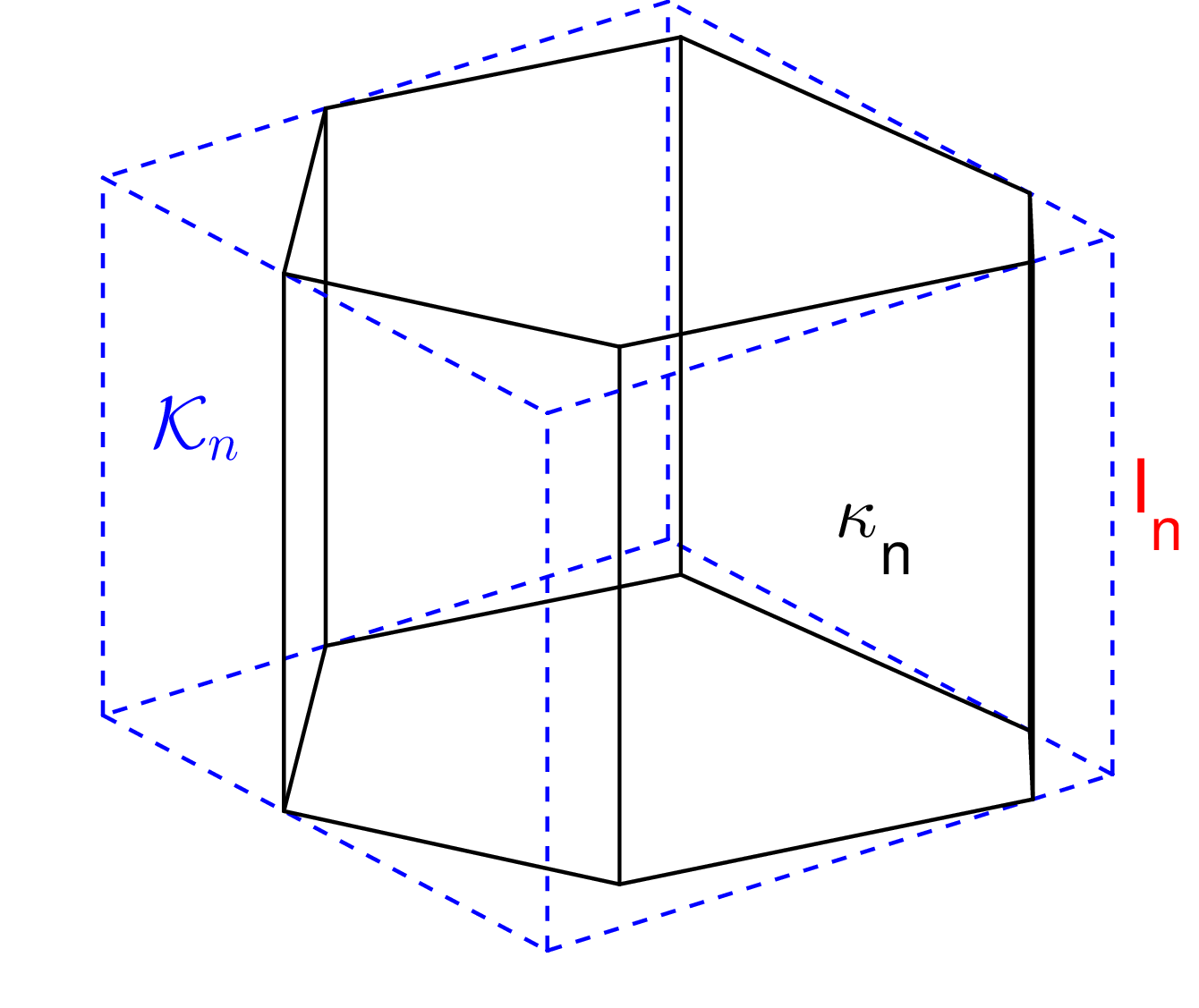}\\
(a) & (b)
\end{tabular}
\end{center}
\caption{(a). Polygonal spatial element $\kappa$ and covering ${\cal K}$; (b) space-time element $\kappa_n=I_n\times\kappa$ and covering ${\cal K}_n:=I_n\times{\cal K}$.} \label{spatial_temporal_covering}
\vspace{-.3cm}
\end{figure}

\begin{theorem}[\cite{stein}]\label{thm-extension}
Let $\Omega$ be a domain with a Lipschitz boundary. Then, there exists a linear extension
operator $\frak{E}:H^s(\Omega) \rightarrow H^s({\mathbb R}^d)$, $s \in {\mathbb N}_0$, such that
$\frak{E}v|_{\Omega}=v$ and
$
\| \frak{E} v \|_{H^s({\mathbb R}^d)} \leq C \| v \|_{H^s(\Omega)},
$
with $C>0$ constant depending only on $s$ and $\Omega$.
\end{theorem}

Moreover, we shall also denote by $\frak{E}v$ the (trivial) space-time extension $\frak{E}v:L_2(J;H^s(\Omega))\to L_2(J;H^s(\mathbb{R}^d))$ defined as the spatial extension above, for every $t\in J$.

\begin{lemma}\label{approx_lemma}
Let $\kappa_n \in \stelem$, $\stF \subset \partial\kappa_n$ a face, and $\mathcal{K}\in\mathcal{T}_{\sharp}$ as in Definition~\ref{meshes2} and let ${\cal K}_n=I_n\times {\cal K}$ (see Figure \ref{spatial_temporal_covering}(b) for an illustration).
Let $v\in L_2(J \times \Omega)$, such that $\frak{E}v|_{\mathcal{K}_n}\in H^{l_{\kappa_n}}(\mathcal{K}_n)$, for some $l_{\kappa_n}\ge 0$. Suppose also that
Assumptions \ref{space-time-shape regular} and \ref{assumption_overlap}, hold. Then, there exists
$\tilde{\Pi}v|_{\kappa_n} \in \mathcal{P}_{p_{\kappa_n}}(\kappa_n)$, such that
\begin{equation}\label{approxH_k}
\norm{v - \tilde{\Pi} v}{H^q(\kappa_n)}
\le C \frac{h_{\kappa_n}^{s_{\kappa_n}-q}}{p_{\kappa_n}^{l§_{\kappa_n}-q}}\norm{\frak{E}v}{H^{l_{\kappa_n}}(\mathcal{K}_n)},\quad l_{\kappa_n}\ge 0,
\end{equation}
for $0\le q\le l_{\kappa_n}$,
\begin{equation}\label{approx-inf_k}
\norm{v - \tilde{\Pi} v}{L_2(\partial {\k_n} \cap \stF^\perp)}
\le C  \frac{h_{\kappa_n}^{s_{\kappa_n}-1/2}}{ p_{\kappa_n}^{l_{\kappa_n}-1/2}}
\norm{\frak{E} v}{H^{l_{\k_n}}(\mathcal{K}_n)}, \quad l_{\kappa_n}>1/2,
\end{equation}
and
\begin{equation}\label{approx-inf_k2}
\norm{v - \tilde{\Pi} v}{L_2(\partial {\k_n}\cap \stF^\parallel)}
\le C \frac{h_{\kappa_n}^{s_{\kappa_n}-1/2}}{p_{\kappa_n}^{l_{\kappa_n}-1/2}}
\norm{\frak{E} v}{H^{l_{\k_n}}(\mathcal{K}_n)},  \quad l_{\kappa_n}>1/2,
\end{equation}
with $s_{\kappa_n}=\min\{p_{\kappa_n}+1, l_{\kappa_n}\}$, and $C>0$ constant,
depending on the shape-regularity of ${\mathcal{K}_n}$, but is
independent of $v$, $h_{\kappa_n}$, $p_{\kappa_n}$ and number of faces per element.
\end{lemma}
\begin{proof}
The bound \eqref{approxH_k} can be proved in completely analogous fashion to the bounds appearing in \cite{cangiani2013hp,schotzau2013hp}. The proof of \eqref{approx-inf_k} also follows using an anisotropic version of the classical trace inequality (see, e.g., \cite{thesis}) and \eqref{approxH_k} for $q=0,1$. We give detailed  proof for \eqref{approx-inf_k2}. By employing Assumption \ref{A1}, relation \eqref{shape_relation1}, \eqref{shape_relation2}, the trace inequality over simplices, arithmetic mean inequality, and \eqref{approxH_k}, we have
\begin{eqnarray}
\qquad&& \norm{v - \tilde{\Pi} v}{L_2(\partial {\k_n} \cap \stF^\parallel)}^2
=  \sum_{\stF^\parallel \subset \partial \k_n} \norm{v - \tilde{\Pi} v}{L_2(\stF^\parallel)}^2  \nno \\
&\le&  C_{tr}
\sum_{\stF^\parallel \subset \partial \k_n} \Big(
\frac{p_{\k_n}}{h_{\k_n}} \norm{v - \tilde{\Pi} v}{L_2(I_n; L_2(s^F_\k))}^2+
\frac{h_{\k_n}}{p_{\k_n}}\norm{v - \tilde{\Pi} v}{L_2(I_n; H^1(s^F_\k))}^2   \Big)\nno \\
&\le& C \frac{h_{\kappa_n}^{2s_{\kappa_n}-1}}{p_{\kappa_n}^{2l_{\kappa_n}-1}}
\norm{\frak{E} v}{H^{l_{\k_n}}(\mathcal{K}_n)}^2,  \quad l_{\kappa_n}>1/2,\nno
\end{eqnarray}
where the  constant $C$ depending the on constant from trace inequality and $C_s$ in \eqref{shape_relation1}, but independent of the discretization parameters and number of faces per element, see \cite[Lemma 1.49]{DiPietroErn}.
\end{proof}
\begin{remark}
Assumption \ref{space-time-shape regular} is the result of the use of the projector $\tilde{\Pi}$ onto the space-time total degree $\mathcal{P}_p$ elemental basis in the analysis. Nonetheless, we emphasise that the dG scheme introduced in this work allows for the use of different type of space-time basis over general shaped space-time prismatic elements. For problems with singularities in time, bases with  anisotropic tensor product space-time basis could  be used instead, while elements with total degree $\mathcal{P}_p$ basis would be used in spatio-temporal regions where the solution is characterised by isotropic behaviour.
\end{remark}

\subsection{Error-analysis}\label{Error-analysis}
We first give an a priori error bound for the space-time dG scheme \eqref{adv_galerkin_dg} in the $\nsdg{\cdot}$--norm, before using this bound to prove a respective $L_2(L_2)$--norm a priori error bound.
\begin{theorem} \label{thm:apriori}
Let Assumptions~\ref{A1}, \ref{space-time-shape regular} and~\ref{assumption_overlap} hold, and let $u_h\in \stfes$ be the space-time dG approximation to the exact solution $u \in L_2(J ; H^1(\Omega))\cap H^1(J ; H^{-1}(\Omega)) $,
with the discontinuity-penalization parameter given by~\eqref{eq:IPdef},
and suppose that  $u|_{\k_n} \in  H^{l_{\k_n}}(\k_n)$, $l_{\k_n}\geq 1$, for each $\k_n \in \stelem$, such that
$\frak{E}u|_{\mathcal{K}_n}\in H^{l_{\k_n}}(\mathcal{K}_n)$.   Then,
the following error bound holds:
\begin{eqnarray}\label{Error_bounds}
 \nsdg{u-u_{h}}^2 \le C \hspace{-0.3cm} \sum_{\kappa_n \in \stelem} \frac{h_{\kappa_n}^{2s_{\kappa_n}}}{p_{\kappa_n}^{2l_{\kappa_n}}}
\big(  && {\cal G}_{\kappa_n}(h_{\k_n},p_{\k_n}) + {\cal D}_{\kappa_n}(h_{\k_n},p_{\k_n}) \big)\|\frak{E} u\|_{H^{l_{\k_n}}(\mathcal{K}_n)}^2,
\end{eqnarray}
where
$
 {\cal G}_{\kappa_n} (h_{\k_n},p_{\k_n})
=
\stp^{-1}+\stp {p_{\kappa_n}^{2}}{h_{\kappa_n}^{-2}}  + p_{\kappa_n} h_{\kappa_n}^{-1}  +\bar{\bold{a}}_{\kappa_n}{p_{\kappa_n}^{2}}{h_{\kappa_n}^{-2}}   +p_{\kappa_n} h_{\kappa_n}^{-1}
\max_{{\stF^\parallel \subset \partial\kappa_n  }}
\sigma ,
$
and
\begin{eqnarray}\label{incon-err}
{\cal D}_{\kappa_n}(h_{\k_n},p_{\k_n})
&=& \bar{\bold{a}}_{\kappa_n}^2
\big(p_{\kappa_n}^{3}h_{\k_n}^{-3}
\max_{{\stF^\parallel\subset \partial\kappa_n  }}\sigma^{-1}  +{p_{\kappa_n}^{4} h_{\kappa_n}^{-3}}
\max_{{\stF^\parallel\subset \partial\kappa_n  }}\sigma^{-1}
\big ),
\end{eqnarray}
with
$s_{\kappa}=\min\{p_{\kappa}+1,l_{\kappa}\}$ and $p_\k \geq1$.
Here, the positive constant C  is independent of the discretization parameters, number of faces per element and $u$.
\end{theorem}

\begin{proof}
After noting that $\tmesh \le c_{reg} h_{\kappa}$ by assumption, a priori bound can be derived following similar ways in \cite{cangiani2015hp} where a priori bound for  general second order linear problems is presented. However, we detail here a different treatment of the trace terms to take advantages of the different mesh assumption used here. Let $\rho = u-\tilde{\Pi}u$, $\tilde{\Pi}$ is the projector defined in Lemma \ref{approx_lemma}. By employing relation \eqref{approx-inf_k2} in approximation Lemma \ref{approx_lemma}, we have
\begin{eqnarray}\label{trick}
&&\quad \int_J\int_{\Gamma}  \sigma | \jump{\rho} |^2 \ud{s}  \ud{t}=\!\! \sum_{\stF^\parallel \subset J\times  \Gamma}\!\!  \sigma  \ltwo{\jump{ \rho}}{\stF^\parallel}^2 \\  \nno \\
&\leq& 2\sum_{{\k_n}\in \stelem} ( \max_{\stF^\parallel\subset \partial \k_n} \sigma)  \ltwo{ { \rho}}{L_2(\partial \k_n \cap \stF^\parallel)}^2
\leq C\sum_{{\k_n}\in \stelem} ( \max_{\stF^\parallel\subset \partial \k_n} \sigma)   \frac{h_{\kappa_n}^{2s_{\kappa_n}-1}}{p_{\kappa_n}^{2l_{\kappa_n}-1}} \|\frak{E} u\|_{H^{l_{\k_n}}(\mathcal{K}_n)}^2;  \nno
\end{eqnarray}
the constant $C>0$ is independent of number of faces per elements. Bounds for remaining trace and inconsistency terms can be derived in completely analogous fashion to the respective result in \cite{cangiani2015hp}.
\end{proof}

The above a priori bound holds \emph{without} any assumptions on the relative size of the spatial faces $F$, $F\subset \partial \kappa$, and number of faces  of a given spatial polytopic element $\k \in {\cal T}$, i.e., elements with arbitrarily small faces and/or arbitrary number of faces are permitted, as long as they satisfy Assumption \ref{A1}.

\begin{remark}
For later reference, we note that  ${\cal D}_{\kappa_n}(h_{\k_n},p_{\k_n})$ given in \eqref{incon-err}, estimates the inconsistency part of the error; we refer to \cite{cangiani2015hp} for details.
\end{remark}

\begin{corollary}\label{cor_conv_l2h1}
Assume the hypotheses of Theorem \ref{Error_bounds} and consider uniform elemental polynomial degrees $p_{\k_n} =p \geq1$. Assume also that $h=\max_{\k_n \in\stelem} h_{\k_n}$, $s_{\kappa_n}=s$ and
$s=\min\{p+1,l\}$, $l\geq1$. Then, we have the bound
\begin{equation*}
\ltwo{u-u_{h}}{L_2(J; H^1(\Omega, \mathcal{T}))} \leq C \frac{h^{s-1}}{p^{l-\frac{3}{2}}} \| u\|_{H^{l}(J \times \Omega)},
\end{equation*}
for $C>0$ constant, independent of $u$, $u_h$, and of the mesh parameters.
\end{corollary}

The above bound is, therefore, $h$-optimal and $p$-suboptimal by $p^{1/2}$.

Next, we derive an error bound in the $L_2(J; L_2(\Omega))$--norm using a duality argument. To this end, the backward adjoint problem of \eqref{Problem} is defined by
\begin{equation}\label{Adjoint Problem}
\begin{aligned} -\partial_t z  - \nabla \cdot ({\bf a} \nabla z) &=\phi  \quad \text{in } J \times  \Omega , \\
z|_{t=T} = g \quad \text{on }  \Omega, \quad &\text{and} \quad u = 0  \quad \text{on }  J \times \partial\Omega.
\end{aligned}
\end{equation}
Assume that $g\in H^1_0(\Omega)$ and $\phi \in L_2(J; L_2(\Omega))$. Then we have
\begin{equation} \label{regularity}
z \in  L_2(J; H^2(\Omega)) \cap L_\infty(J; H^1_0(\Omega)), \quad  \partial_t z \in L_2(J; L_2(\Omega)).
\end{equation}
We assume that $\Omega$ and ${\bold a}$ are such that the parabolic regularity estimate
\begin{eqnarray}\label{regularity estimate}
\linf{z}{J; H^1_0(\Omega)} + \ltwo{z}{L_2(J; H^2(\Omega))}+ &&\ltwo{z}{H^1(J; L_2(\Omega))}\leq \\ \nonumber
 && ~~~~~~~~C_r( \ltwo{\phi}{L_2(J; L_2(\Omega))} + \| g\|_{H^1_0(\Omega)}),
\end{eqnarray}
holds with the constant $C_r>0$ depending only on $\Omega$, $T$ and $\bold{a}$; cf.~\cite[p.360]{evans2010partial} for smooth domains, and the parabolic regularity results can be extended to convex  domains by using results in \cite[Chapter 3]{grisvard2011elliptic}.
\begin{assumption}\label{bounded-variation}
For any two $d$-dimensional spatial elements $\kappa$, $\kappa^\prime\in \mathcal{T}$ sharing the same $(d-1)-$face, we have:
\begin{equation}
\max(h_\k,h_{\k^\prime})\leq c_h\min(h_\k,h_{\k^\prime}), \quad\max(p_{\k_n},p_{\k_n^\prime})\leq c_p\min(p_{\k_n},p_{\k_n^\prime}),
\end{equation}
for $n=1,\dots,N_t$, $c_h>0$, $c_p>0$ constants, independent of discretization parameters.
\end{assumption}

For each given time interval $I_n$, we define orthogonal projection in time variable $\pi^t_{\bar{p}_n}$, with $\bar{p}_n: = \min_{\kappa_n\in\mathcal{U}\times \mathcal{T}} \lfloor {\frac{p_{\kappa_n}}{2}} \rfloor $, i.e., the integer part of the minimum elemental polynomial degree of the space-time slab $I_n\times \Omega$ divided by two. We collect all the $\bar{p}_n$ into a vector $\bar{p}:=(\bar{p}_1,\dots, \bar{p}_{N_t})$.

We now present the theorem under the assumption that $a\equiv a(\mbf{x})$, i.e., the diffusion coefficient depends only on the spatial variables.

\begin{theorem}[$a\equiv a(\mbf{x})$]\label{L2_L2 norm}
Consider the setting of  Theorem~{\ref{thm:apriori}}, and assume the parabolic regularity estimate \eqref{regularity estimate}  holds along with  Assumption \ref{bounded-variation}. Then, we have the bound
\begin{eqnarray*}
 \ltwo{u-u_{h}}{L_2(J; L_2(\Omega))}^2 &\le& Cr_{\max} \max_{\kappa_n\in\stelem}\!\! h_{\k_n}\!\!\!\!\sum_{\kappa_n \in \stelem}\! \frac{h_{\kappa_n}^{2s_{\kappa_n}}}{p_{\kappa_n}^{2l_{\kappa_n}}}
 \big(   {\cal G}_{\kappa_n}(h_{\k_n},p_{\k_n}) \\
 &&~~~~~~~~~~~~~~~~~~~~~~~~~ +{\cal D}_{\kappa_n}(h_{\k_n},p_{\k_n}) \big)\|\frak{E} u\|_{H^{l_{\k_n}}(\mathcal{K}_n)}^2,
\end{eqnarray*}
with  $r_{\max}:=(\max_{\kappa_n\in\stelem} p_{\kappa_n}/\bar{p}_{n})^2$, with the constant $C>0$, independent of $u$, $u_h$, of the discretization parameters and of number of faces per element.
\end{theorem}
\begin{proof}
Let $\pi^t_{q}$, for $q=(q_1,\dots,q_{N_t})$, denote the $L_2-$orthogonal projection onto (discontinuous, in general) polynomials of $q_n$ over each $I_n$ with respect to the time variable.  Also, we set $e:=u-u_h$ and $e_{\bar{p}}:= \pi^t_{\bar{p}}e$ the $L_2-$orthogonal projection of the error $e$ in the time variable, for brevity. We have, respectively,
\begin{equation}\label{relation 1}
\begin{split}
 \ltwo{e}{L_2(J; L_2(\Omega))}^2  =&\  \ltwo{e - e_{\bar{p}}}{L_2(J; L_2(\Omega))}^2  +  \ltwo{  e_{\bar{p}}}{L_2(J; L_2(\Omega))}^2 \\
\leq&\  C \sum_{\kappa_n \in \stelem} \frac{\tau_n^2}{\bar{p}_{n}^2} \ltwo{\partial_t  e}{\kappa_n}^2  +  \ltwo{  e_{\bar{p}}}{L_2(J; L_2(\Omega))}^2 \\
\leq&\    C r_{\max} \max_{\kappa_n\in\stelem}\!\! h_{\k_n}\nsdg{e}^2+  \ltwo{ e_{\bar{p}}}{L_2(J; L_2(\Omega))}^2 ,
\end{split}
\end{equation}
with $r_{\max}:=(\max_{\kappa_n\in\stelem} p_{\kappa_n}/\bar{p}_{n})^2$,
from standard estimation and the space-time mesh shape-regularity.

We begin by setting $g = 0 $ and $\phi =  e_{\bar{p}}$ in \eqref{Adjoint Problem}.  We test \eqref{Adjoint Problem} against $e_{\bar{p}}$, we perform integration by parts, noting that, since $z\in L_2(J;H^2(\Omega))$, we have $\jump{ \textbf{a}\nabla z}=0$ on $F^{\parallel}\subset  J\times \Gamma_{\rm int}$  and $\jump{z}=0$ on $F^{\parallel}\subset  J\times \Gamma_{}$, to deduce
\begin{equation}\label{Step 2}
\begin{split}
\ltwo{e_{\bar{p}}}{L_2(J; L_2(\Omega))}^2
&= \sum_{n=1}^{N_t} \int_{I_n}\Big(   - (  e_{\bar{p}},\partial_t z)_{} - ( e_{\bar{p}},\nabla (\textbf{a}\nabla z)) \Big)\ud  t     \\
& = \sum_{n=1}^{N_t} \int_{I_n}  \Big(   - ( e_{\bar{p}},\partial_t z)_{}
+ \tilde{a}( e_{\bar{p}},z)\Big)
\ud t    - R(z,e_{\bar{p}}) ,
\end{split}
\end{equation}
with
$
R(z,e_{\bar{p}}) : =  \int_J \int_{\Gamma}  \mean{\bold{a} ( \nabla z-  \vecL( \nabla z)) }\cdot \jump{e_{\bar{p}}} \ud{s} \ud{t}.
$
 We note that by the defining property of the $L_2-$orthogonal projection, we have
 $ \int_{I_n}  \tilde{a}(   \pi^t_{\bar{p}}e,z)  \ud t=  \int_{I_n} \tilde{a}( e, \pi^t_{\bar{p}}z)\ud t $ and $R(z,e_{\bar{p}}) = R(\pi^t_{\bar{p}}z,e) $, upon observing that $\jump{e_{\bar{p}}}=\pi_{\bar{p}}^t\jump{e}$ by construction of $\bar{p}$.

 For convenience, we now employ the $H^1$-projection operator in the time variable defined over each $I_n$, $n=1,\dots,N_t$, by
$
\mathcal{I}^t_{\bar{p}+1}w(t):=\int_{t_{n-1}}^t \pi^t_{\bar{p}}\partial_t w(s)\ud s + w(t_{n-1}),
$
so that $\partial_t (\mathcal{I}^t_{\bar{p}+1}w)=\pi^t_{\bar{p}}\partial_t w$; see, e.g., \cite{schwab} for details. Note that $\mathcal{I}^t_{\bar{p}+1}$ is continuous in $J$ and interpolates $w$ at the time nodes. We have
\begin{equation} \label{H1 projection}
\int_{I_n}  (  e_{\bar{p}},\partial_t z)_{}  \ud t= \int_{I_n}  ( e,\pi^t_{\bar{p}} \partial_t {z})_{}  \ud t  = \int_{I_n}  (e, \partial_t \mathcal{I}^t_{\bar{p}+1}{z})_{}  \ud t.
\end{equation}

We now integrate by parts over each $I_n$, use that $z(t_{N_t}) = 0$, and the interpolation-at-the-nodes property of $\mathcal{I}_{\bar{p}+1}$, and the orthogonality of the $\pi_{\bar{p}}^t-$orthogonal projection, to deduce
\begin{equation}\label{error equaiton 1}
\begin{split}
\ltwo{e_{\bar{p}}}{L_2(J; L_2(\Omega))}^2
=&\ \sum_{n=1}^{N_t} \int_{I_n}  \Big( - ( e, \partial_t \mathcal{I}^t_{\bar{p}+1}{z})_{}
+ \tilde{a}(e, \pi^t_{\bar{p}} z) \Big)  \ud t
  - R(\pi^t_{\bar{p}}z, e) \\
 =&\   \sum_{n=1}^{N_t} \int_{I_n}   \Big( (  \partial_t e,\mathcal{I}^t_{\bar{p}+1} z)_{}  +\tilde{a}(e, \pi^t_{\bar{p}} z)\Big) \ud t
 \\
&+ \sum_{n=1}^{N_t-1}(\ujump{e}_{n},z(t_{n}))_{}+(e(t_{0}^+),z(t_{0}^+))_{}
   - R(\pi^t_{\bar{p}}z, e).
\end{split}
\end{equation}

Further, for any $z_h\in \stfes$,  Galerkin orthogonality implies
$\tilde{B}( e,z_h) =R(u,z_h)= -R(u,z-z_h)$ since $R(u,z)=0$.  Subtracting this now from \eqref{error equaiton 1}, results to
\begin{equation}\label{error equation}
	\begin{split}
		\ltwo{e_{\bar{p}}}{L_2(J; L_2(\Omega))}^2
		=&\   \sum_{n=1}^{N_t} \int_{I_n}   \Big( (  \partial_t e,\mathcal{I}^t_{\bar{p}+1} z-z_h)_{}  +\tilde{a}(e, \pi^t_{\bar{p}} z-z_h)\Big) \ud t
		\\
		&+ \sum_{n=1}^{N_t-1}(\ujump{e}_{n},(z-z_h)(t_n^+))_{}+(e(t_{0}^+),(z-z_h)(t_0^+))_{}\\
	&	- R(\pi^t_{\bar{p}}z, e) + R(u,z-z_h)\\
  \leq &\  C \Big(  \sum_{\kappa_n \in \stelem}  \stp^{-1}\ltwo{ ( \mathcal{I}^t_{\bar{p}+1} z - z_h)}{\kappa_n}^2  + \int_{J}  \ndg{ \pi_{\bar{p}}^t z - z_h}_{\rm d}^2 \ud  t\\
&  +\sum_{n=0}^{N_t-1}   \ltwo{( z - z_h)(t_{n}^+)}{}^2  \Big)^\frac{1}{2} \nsdg{e}	- R(\pi^t_{\bar{p}}z, e) + R(u,z-z_h).
\end{split}
\end{equation}
Let $z_h\in \stfes$ defined on each element $\kappa_n \in\stelem$ by $z_h|_{\kappa_n} :=\pi^t_{\bar{p}_n}  \tilde{\Pi}_{ p_{\kappa_n}-\bar{p}_n }   z$
with $\tilde{\Pi}_{q}$ is the projector defined in Lemma \ref{approx_lemma} over $d$-dimensional spatial variables. Note that this choice ensures that $z_h \in \stfes$. For notational convenience, we collect all indices into a piecewise constant function $p'$ over the space-time mesh, defined by $p'|_{\kappa_n}:=p_{\kappa_n}-\bar{p}_n$. Note that, by definition, $p’\ge 1$.

We estimate each term of the right-hand side of \eqref{error equation} separately. Recalling standard $hp$-approximation bounds (see, e.g., $\cite{newpaper}$), we have for ${p'}=p_{\kappa_n}-\bar{p}_n$,
\begin{equation}\label{term1}
\begin{split}
& \sum_{\kappa_n \in \stelem} \!\! \!\!\! \stp^{-1} \ltwo{\mathcal{I}^t_{\bar{p}_n+1} z - z_h }{\kappa_n}^2
\leq
2\!\!\!\!\!\sum_{\kappa_n \in \stelem} \!\!\! \!\!\stp^{-1}  \Big(  \ltwo{\mathcal{I}^t_{\bar{p}_n+1} -  z}{\kappa_n}^2  + \ltwo{  z-    \pi_{\bar{p}_n}^t \tilde{\Pi}_{{p'}} z}{\kappa_n}^2  \Big)   \\
&\leq   C\sum_{\kappa_n \in \stelem} \stp^{-1}  \Big(  \ltwo{z-\mathcal{I}^t_{\bar{p}_n+1}    z}{\kappa_n}^2 + \ltwo{z-\pi_{\bar{p}_n}^t     z}{\kappa_n}^2  + \ltwo{  \pi_{\bar{p}_n}^t (z-   \tilde{\Pi}_{{p'}} z)}{\kappa_n}^2  \Big)   \\
&\leq   C\sum_{\kappa_n \in \stelem}   \stp^{-1} \Big(   \frac{\tmesh^2}{\bar{p}_{\kappa_n}^2} \ltwo{\partial_t  z}{\kappa_n}^2  +  \frac{ h_\kappa^4}{p_{\kappa_n}^4} \ltwo{  \frak{E}z}{L_2(I_n; H^2(\mathcal{K}))}^2  \Big)   \\
&\leq  C \max_{\kappa_n}h_{\kappa_n} \Big(r_{\max}\ltwo{ z}{H^1(J; L_2(\Omega))}^2 +\max_{\kappa_n} \frac{ h_{\kappa_n}^2}{p_{\kappa_n}^2} \ltwo{z}{L_2(J; H^2(\Omega))}^2 \Big),
\end{split}
\end{equation}
using the triangle inequality, the stability of $L_2-$orthogonal projection, Assumptions \ref{space-time-shape regular} and \ref{bounded-variation} and, finally,  Theorem 5.4 \eqref{thm-extension}, respectively. Next, we have
\begin{equation} \label{term2}
\begin{split}
\int_{J}  \ndg{\pi^t_{\bar{p}_n} z - z_h}_{\rm d}^2 \ud  t
=&\  \int_{J}  \ndg{\pi^t_{\bar{p}_n} (z - \tilde{\Pi}_{{p'}} z)}_{\rm d}^2 \ud  t\\
\leq&\    \int_{J}  \ndg{z - \tilde{\Pi}_{{p'}} z}_{\rm d}^2 \ud  t
\leq   C \max_{\kappa_n} \frac{h^2_{\kappa_n}}{p_{\kappa_n}}  \ltwo{z}{L_2(J; H^2(\Omega))}^2.
\end{split}
\end{equation}
Also,
\begin{equation}\label{term3}
	\begin{aligned}
\sum_{n=0}^{N_t-1}   \ltwo{( z - z_h)(t_{n}^+)}{}^2
 \leq &\ 2  \sum_{n=0}^{N_t-1} \!\!\su \! \Big(\ltwo{(z-\pi_{\bar{p}}^t    z)(t_{n}^+)}{\kappa}^2   +  \ltwo{\pi_{\bar{p}_n}^t ( z-    \tilde{\Pi}_{{p'}} z)(t_{n}^+)}{\kappa}^2 \Big)     \\
\leq &\ C  \sum_{\kappa_n\in \stelem}   \Big( \frac{\tmesh}{\bar{p}_{n}}  \ltwo{\partial_t  z}{\kappa_n}^2   + \frac{\bar{p}_{n}^2}{\tmesh} \ltwo{\pi_{\bar{p}}^t  (z-   \tilde{\Pi}_{r} z)}{\kappa_n}^2 \Big)     \\
\leq &\ C\sum_{\kappa_n\in \stelem}  \Big( \frac{\tmesh}{\bar{p}_{n}}  \ltwo{\partial_t  z}{\kappa_n}^2   +  \frac{h_\kappa^4}{\tmesh p_{\kappa_n}^2} \ltwo{  \frak{E}z}{L_2(J; H^2(\mathcal{K}))}^2 \Big)     \\
\leq &\ C \max_{\kappa_n} \frac{ h_{\kappa_n}}{\bar{p}_{n}} \Big(\ltwo{ z}{H^1(J; L_2(\Omega))}^2 +\max_{\kappa_n}\frac{ h_{\kappa_n}^2}{p_{\kappa_n}} \ltwo{z}{L_2(J; H^2(\Omega))}^2 \Big),
\end{aligned}
\end{equation}
using an $hp$-version trace inverse estimate and working as before.

Combining the above estimates, the first term on the right-hand side of \eqref{error equation} is shown to be bounded by
\begin{eqnarray} \label{result 1}
CC_r \sqrt{r_{\max}}    \max_{\kappa_n}  \sqrt{ h_{\kappa_n} }\nsdg{e} \ltwo{e}{L_2(J; L_2(\Omega))}.
\end{eqnarray}
Moving on to the second term on the right-hand side of \eqref{error equation}, using the stability of $\pi^t_{\bar{p}}$, we have
\begin{equation} \label{inconsistent term}
\begin{aligned}
R(\pi^t_{\bar{p}}z,e) &=   \int_{J} \int_{\Gamma}  \mean{\bold{a} ( \nabla \pi^t_{\bar{p}} z-  \vecL( \nabla \pi^t_{\bar{p}} z)) }\cdot \jump{e} \ud{s} \ud{t} \\
& \le \ltwo{ \sigma^{-1/2}\mean{\bold{a} ( \nabla z- \boldsymbol{ \tilde{\Pi}_2}( \nabla z)) }}{L_2(J,L_2(\Gamma))}^2\nsdg{e}.
\end{aligned}
\end{equation}
To bound further $R(\pi^t_{\bar{p}}z,e)$, it is sufficient to bound \mbox{I}+\mbox{II} instead, whereby
\begin{equation*}
	\begin{aligned}
	\mbox{I}:=&\  2\ltwo{\sigma^{-1/2}   \mean{\bold{a} ( \nabla z- \boldsymbol{ \tilde{\Pi}_{r}}( \nabla z)) }}{L_2(J,L_2(\Gamma))}^2, \\
	\mbox{II}:=&\ 2\ltwo{\sigma^{-1/2}  \mean{\bold{a} \vecL( \boldsymbol{ \tilde{\Pi}_{r}}( \nabla z)-   \nabla z) }}{L_2(J,L_2(\Gamma))}^2.
\end{aligned}
\end{equation*}
For  \mbox{I}, using Lemma~5.1 and, working as before, gives
\begin{eqnarray}\label{inconsistent term 1}
\mbox{I} &\leq& C  \max_{\kappa_n} \frac{h_{\k_n}^{2}}{p_{\k_n}^3}   \ltwo{ z}{L_2(J; H^2(\Omega))}^2
 \end{eqnarray}

By using the inverse estimation \eqref{inv_est_gen} and stability of $\mathbf{\Pi}_2$, and working as above, we also have
\begin{eqnarray}\label{inconsistent term 2}
\mbox{II} &\leq& C C_r \max_{\kappa_n} \frac{h_{\k_n}^{2}}{p_{\k_n}^2}   \ltwo{ z}{L_2(J; H^2(\Omega))}^2.
 \end{eqnarray}

Therefore, \eqref{inconsistent term 1} and \eqref{inconsistent term 2}, together with the maximal parabolic regularity estimate \eqref{regularity estimate} give
\begin{equation}\label{result 2}
R(\pi_{\bar{p}}^tz,e)  \leq  CC_r    \max_{\kappa_n} \frac{h_{\kappa_n}}{p_{\kappa_n}}  \nsdg{e} \ltwo{e}{L_2(J; L_2(\Omega))}.
\end{equation}

Next, we bound the last term $R(u,z-z_h)$ on the right-hand side of \eqref{error equation}. A key observation is the following $\jump{z-z_h}=-\jump{z_h}=\jump{\pi_{\bar{p}}^tz- z_h}=\jump{\pi_{\bar{p}}^t(z- \Pi_{{p'}}z)}$, for the continuity of $\pi_{\bar{p}}^tz$ with respect to the spatial variables. Hence, we have
\begin{equation}  \label{result 3}
\begin{aligned}
& R(u,z-z_h)=  \int_{J} \int_{\Gamma}  \mean{\bold{a} ( \nabla_h u-  \vecL( \nabla_h u)) }\cdot \jump{ \pi^t_{\bar{p}} z-z_h} \ud{s} \ud{t}  \\
&\leq  \Big( \int_{J} \int_{\Gamma} \sigma^{-1}  |\mean{\bold{a} ( \nabla_h u-  \vecL( \nabla_h u)) }|^2 \ud{s} \ud{t}\Big)^{\frac{1}{2}} \Big(  \int_{J} \int_{\Gamma} \sigma \jump{ \pi^t_{\bar{p}} (z-  \tilde{\Pi}_{\bar{p}}  )}^2 \ud{s} \ud{t}\Big)^{\frac{1}{2}}  \\
 &\leq CC_r   \max_{\kappa_n}  \frac{h_{\kappa_n}}{\sqrt{p_{\kappa_n}}} \ltwo{e}{L_2(J; L_2(\Omega))}
 \Big( \sum_{\kappa_n \in \stelem} \frac{h_{\kappa_n}^{2s_{\kappa_n}}}{p_{\kappa_n}^{2l_{\kappa_n}}}  {\cal D}_{\kappa_n}(h_{\k_n},p_{\k_n}) \|\frak{E} u\|_{H^{l_{\k_n}}(\mathcal{K}_n)}^2  \Big)^{1/2}.
\end{aligned}
\end{equation}
Combining \eqref{result 1}, \eqref{result 2},  \eqref{result 3} and \eqref{error equation} with \eqref{relation 1}, the result follows.
\end{proof}

Next, we discuss the extension to time-dependent diffusion coefficient $\mbf{a}\in W^{1,\infty}(J\times\Omega)$, assuming the validity of the maximal parabolic regularity estimate; we refer to Wloka~\cite{Wloka} for a detailed discussion on parabolic regularity in this setting.

\begin{theorem}[$a\equiv a(t,\mbf{x})$]\label{L2_L2 norm_td}
	Consider the setting of  Theorem~{\ref{thm:apriori}}, and assume the parabolic regularity estimate \eqref{regularity estimate}  holds along with  Assumption \ref{bounded-variation} for $\mbf{a}\in W^{1,\infty}(J\times\Omega)$. Then, we have the bound
	\begin{eqnarray*}
		\ltwo{u-u_{h}}{L_2(J; L_2(\Omega))}^2 &\le& Cr_{\max} \max_{\kappa_n\in\stelem}\!\! h_{\k_n}\!\!\!\!\sum_{\kappa_n \in \stelem}\! \frac{h_{\kappa_n}^{2s_{\kappa_n}}}{p_{\kappa_n}^{2l_{\kappa_n}}}
		\big(   {\cal G}_{\kappa_n}(h_{\k_n},p_{\k_n}) \\
		&&~~~~~~~~~~~~~~~~~~~~~~~~~ +{\cal D}_{\kappa_n}(h_{\k_n},p_{\k_n}) \big)\|\frak{E} u\|_{H^{l_{\k_n}}(\mathcal{K}_n)}^2,
	\end{eqnarray*}
	with  $r_{\max}:=(\max_{\kappa_n\in\stelem} p_{\kappa_n}/\bar{p}_{n})^2$, with the constant $C>0$, independent of $u$, $u_h$, of the discretization parameters and of number of faces per element.
\end{theorem}
\begin{proof} The proof and notation follows that of the previous theorem. Here we only comment on the differences in the proof.

Working as before, \eqref{relation 1} holds and, using the same arguments, we arrive to
	\begin{equation}\label{Step 2_td}
		\begin{split}
			\ltwo{e_{\bar{p}}}{L_2(J; L_2(\Omega))}^2
			& = \sum_{n=1}^{N_t} \int_{I_n}  \Big(   - ( e_{\bar{p}},\partial_t z)_{}
			+ \tilde{a}( e_{\bar{p}},z)\Big)
			\ud t    - R(\mbf{a}; z,e_{\bar{p}}) ,
		\end{split}
	\end{equation}
	with
	$
	R(\mbf{a};z,e_{\bar{p}}) : =  \int_J \int_{\Gamma}  \mean{\bold{a} ( \nabla z-  \vecL( \nabla z)) }\cdot \jump{e_{\bar{p}}} \ud{s} \ud{t},
	$
	denoting explicitly the diffusion coefficient as variable also.
	From this point we have to modify the argument above, since we cannot ``move'' $\pi^t_{\bar{p}}$ from $a$ into $z$ using $L_2-$orthogonality property in both $\tilde{a}$ and $R$.

We now consider the bilinear form $\tilde{a}_0(\cdot,\cdot)$ which arises by replacing the diffusion coefficient $\mbf{a}$ by $\mbf{a}_0:=\pi_0^t\mbf{a}$, i.e., taking the average with respect to the $t$-variable on each time interval. Then, we have, respectively,
	\[
	\begin{aligned}
	 \int_{J}  \tilde{a}(  e_{\bar{p}},z)  \ud t &=  \int_{J}  (\tilde{a}-\tilde{a}_0)(  e_{\bar{p}},z)  \ud t+  \int_{J} \tilde{a}_0( e, \pi^t_{\bar{p}}z)\ud t\\
	 &= \int_{J}  \Big( (\tilde{a}-\tilde{a}_0)(   e_{\bar{p}},z) -  (\tilde{a}-\tilde{a}_0)( e, \pi^t_{\bar{p}}z)\Big) \ud t+ \int_{J} \tilde{a}( e, \pi^t_{\bar{p}}z)\ud t
	 \\
	  &=:  \mathcal{R}_{\tilde{a}}+ \int_{J} \!\tilde{a}( e, \pi^t_{\bar{p}}z)\ud t.
  \end{aligned}
	  \]
	  Completely analogously, upon observing that $\jump{e_{\bar{p}}}=\pi_{\bar{p}}^t\jump{e}$ by construction of $\bar{p}$, we have
	  	\[
	  \begin{aligned}
	  	R(\mbf{a};z,e_{\bar{p}})& = 	R(\mbf{a}-\mbf{a}_0;z,e_{\bar{p}})+	R(\mbf{a}_0;\pi_{\bar{p}}^t z,e)\\
	  &	= \Big(	R(\mbf{a}-\mbf{a}_0;z,e_{\bar{p}}) - R(\mbf{a}-\mbf{a}_0;\pi^t_{\bar{p}}z,e) \Big)+R(\mbf{a};\pi_{\bar{p}}^t z,e)\\
	  &	=:  \mathcal{R}_{R} +R(\mbf{a};\pi_{\bar{p}}^t z,e).
	  \end{aligned}
	  \]
	So, working as before, and using the last two identities, we arrive at
%
	\begin{equation}\label{error equation_td}
		\begin{split}
			\ltwo{e_{\bar{p}}}{L_2(J; L_2(\Omega))}^2
		\leq &\  C \Big(  \sum_{\kappa_n \in \stelem}  \stp^{-1}\ltwo{ ( \mathcal{I}^t_{\bar{p}+1} z - z_h)}{\kappa_n}^2  + \int_{J}  \ndg{ \pi_{\bar{p}}^t z - z_h}_{\rm d}^2 \ud  t\\
		&\qquad  +\sum_{n=0}^{N_t-1}   \ltwo{( z - z_h)(t_{n}^+)}{}^2  \Big)^\frac{1}{2} \nsdg{e}	\\
		&\	- R(\mbf{a};\pi^t_{\bar{p}}z, e) + R(\mbf{a};u,z-z_h) +\mathcal{R}_{\tilde{a}}- \mathcal{R}_{R}.
	\end{split}
\end{equation}
We select $z_h\in \stfes$ as in the previous section and we estimate each term on the right-hand side of \eqref{error equation_td} separately; all terms but the last two can be estimated identically to the proof in the previous section. To complete the proof, we estimate $R_{\tilde{a}}$ and $\mathcal{R}_{R}$ from above.

For $R_{\tilde{a}}$, using the stability of $\pi_{\bar{p}}^t$ in $L_2$-norm in time, along with standard estimation, we have, respectively,
\[
\begin{aligned}
	\mathcal{R}_{\tilde{a}}=&\ \int_J\bigg(\int_\Omega (\mbf{a}-\mbf{a}_0)\nabla_h e_{\bar{p}}\cdot\nabla z\ud x -\int_\Gamma \mean{(\mbf{a}-\mbf{a}_0)\vecL \nabla z} \cdot \jump{e_{\bar{p}}} \ud s\bigg)\ud t \\
	&\ -\int_J\bigg(\int_\Omega (\mbf{a}-\mbf{a}_0)\nabla_h e\cdot\nabla \pi_{\bar{p}}^t z\ud x -\int_\Gamma \mean{(\mbf{a}-\mbf{a}_0)\vecL \nabla \pi_{\bar{p}}^t z} \cdot \jump{e}\ud s \bigg)\ud t\\
	\le &\ C\norm{\mbf{a}-\mbf{a}_0}{L_2(J;L_\infty(\Omega))}\ndg{e}  \norm{z}{L_\infty(J;H^1(\Omega))},
\end{aligned}
\]
with $C>0$ independent of $h$ and of $p$.

Now, using the stability of $\pi_{\bar{p}}^t$, along with standard estimation, we  also have
 \[
 \begin{aligned}
 		\mathcal{R}_{R}=&\  \int_J \int_{\Gamma}  \mean{(\bold{a}-\bold{a}_0) ( \nabla z-  \vecL( \nabla z)) }\cdot  \pi_{\bar{p}}^t \jump{e} \ud{s} \ud{t} \\
 	&\ - \int_J \int_{\Gamma}  \mean{(\bold{a}-\bold{a}_0) \pi_{\bar{p}}^t ( \nabla z-  \vecL( \nabla z)) }\cdot \jump{e} \ud{s} \ud{t}\\
 	\le &\ C\norm{\mbf{a}-\mbf{a}_0}{L_\infty(J;L_\infty(\Omega))}\ndg{e}  \norm{z}{L_2(J;H^2(\Omega))},
 \end{aligned}
 \]
with $C>0$ independent of $h$ and of $p$.

Observing now, the estimate
\[
\norm{\mbf{a}-\mbf{a}_0}{L_\infty(J;L_\infty(\Omega))}\le C \max_{n} \tau_n \norm{\mbf{a}}{W^{1,\infty}(J\times \Omega)},
\]
the space-time shape-regularity, the above bounds and the parabolic regularity estimate \eqref{regularity estimate}  imply the result.
\end{proof}

The $L_2(J;L_2(\Omega))$--norm error bound in Theorem \ref{L2_L2 norm} (and Theorem \ref{L2_L2 norm_td}) is suboptimal with respect to the mesh size $h$ by half an order of $h$, and sub-optimal in $p$ by $3/2$ order. (The respective space-time tensor-product basis dG method, using the same approach can be shown to be $h$-optimal and $p$-suboptimal one order of $p$.) The numerical experiments in the next section confirm the suboptimality in $h$ for the proposed method, but at the same time highlight its competitiveness with respect to standard (optimal) methods.

An interesting further development would be the use of different polynomial degrees in space and in time as done, e.g., in \cite{sudirham2006space,van2008space} in this context of \emph{total degree} space-time basis. The exploration of a number of index sets for space-time polynomial basis, including this case, will be discussed elsewhere. Nevertheless, the above proof of the $L_2(J;L_2(\Omega))$--norm error bound would carry through, with minor modifications only, for various choices of space-time basis function index sets.


\section{Numerical examples}\label{numerics}
We shall present a series of numerical experiments to investigate the asymptotic convergence behavior of the proposed space-time dG method.
We shall also make comparisons with known methods on space-time hexahedral meshes, such as the tensor-product space-time dG method and the dG-time stepping scheme combined with conforming finite elements in space. Furthermore, an implementation using prismatic space-time meshes with polygonal bases is presented and its convergence is assessed. In all experiments we choose $C_\sigma=10$. The polygonal spatial mesh giving rise to the space-time prismatic elements is generated through the PolyMesher MATLAB library \cite{polymesher}. The High Performance Computing facility ALICE of the University of Leicester was used for the numerical experiments.

\subsection{Example 1}
We begin by considering a smooth problem for which  $u_0$ and $f$ are chosen such that the exact solution $u$ of \eqref{Problem} is given by:
\begin{equation}\label{example1}
u(x,y,t)= \sin(20\pi t)e^{-5((x-0.5)^2+(y-0.5)^2)} \quad \text{in } J \times  \Omega ,
\end{equation}
for $J=(0,1)$ and $\Omega  =  (0,1)^2$, and $\bold{a}(x,y,t)$ is the identity matrix. Notice that the solution oscillates in time. To asses the convergence rate with respect to the space-time mesh diameter $h_{\kappa_n}$ on (quasi)uniform meshes, we fix the ratio between the spatial and temporal mesh sizes to be $h_{\kappa}/\tmesh  =  10$.

The convergence rate with respect to decreasing space-time mesh size $h_{\k_n}$ in three different norms is given in Figure \ref{Ex1_h_refinement} for space-time prismatic elements with rectangular bases (standard hexahedral space-time elements) and for prismatic meshes with quasiuniform  polygonal bases: all computations are performed over $16$, $64$, $256$, $1024$,  $4096$ spatial rectangular or polygonal elements and for $40$, $80$, $160$, $320$, $640$ time-steps, respectively.

\begin{figure}[t]
\begin{center}
\begin{tabular}{cc}
\vspace{-0.1cm}
\includegraphics[scale=0.29]{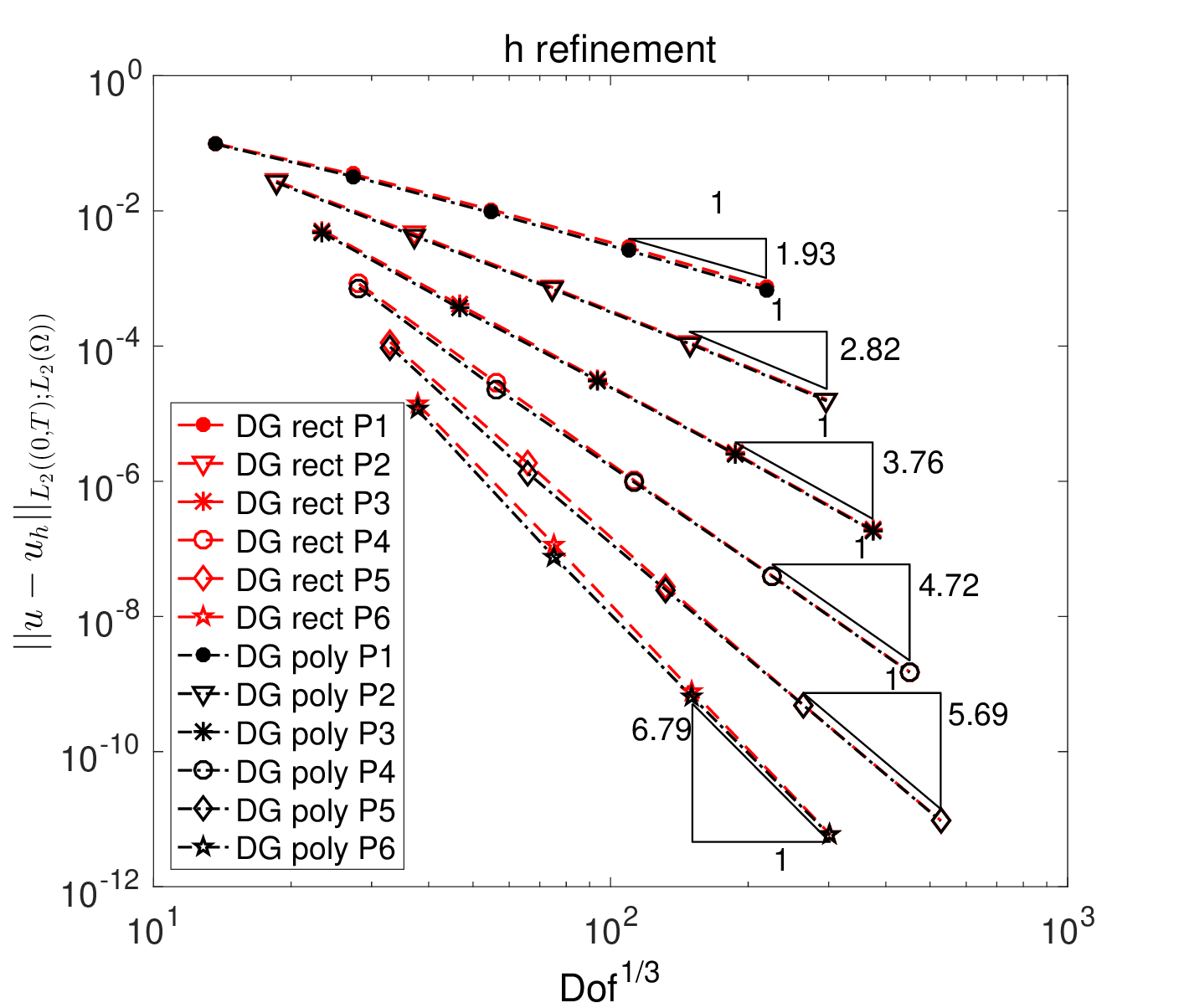}&
\hspace{-1cm}
\includegraphics[scale=0.29]{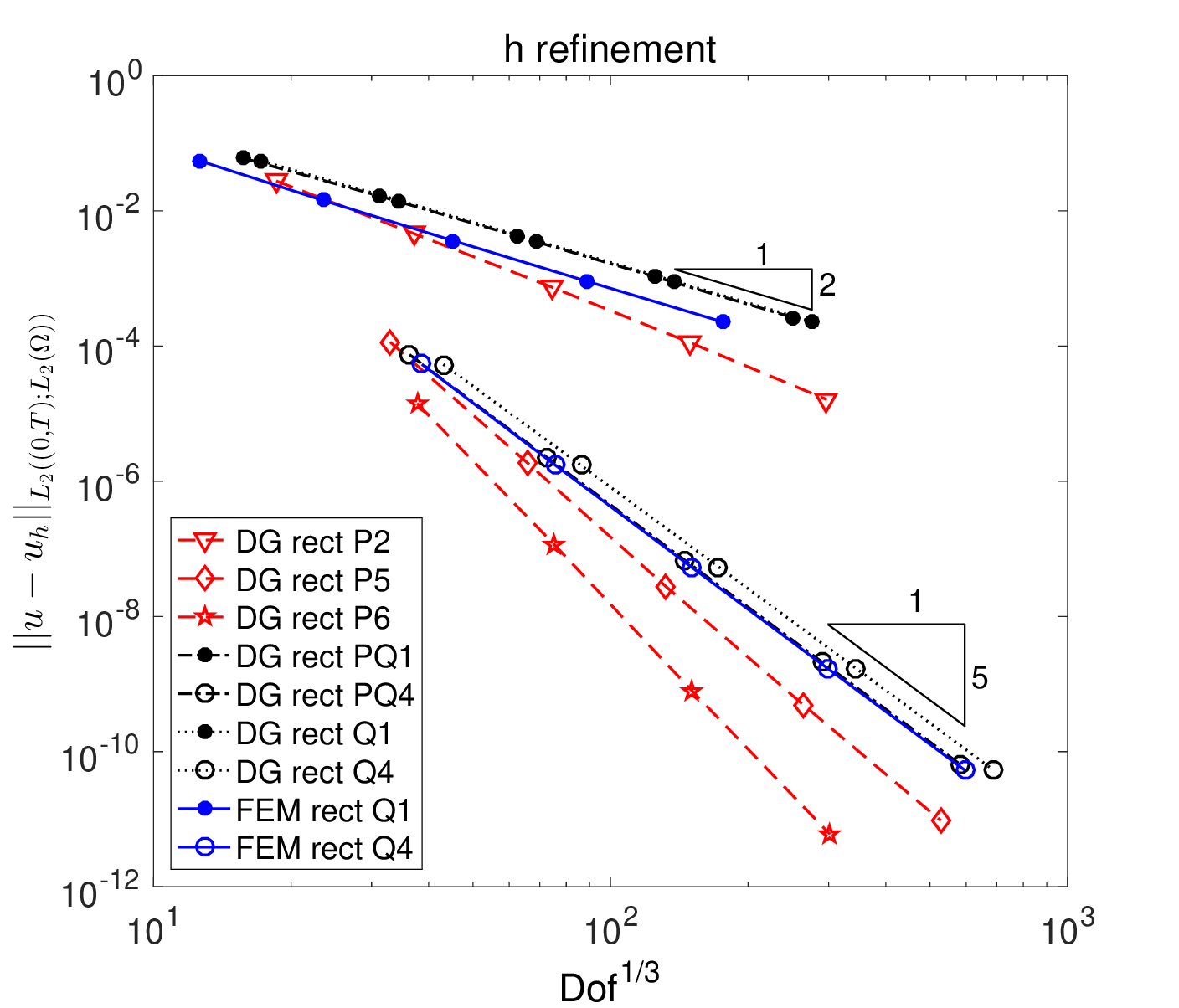} \\
\vspace{-0.1cm}
\includegraphics[scale=0.29]{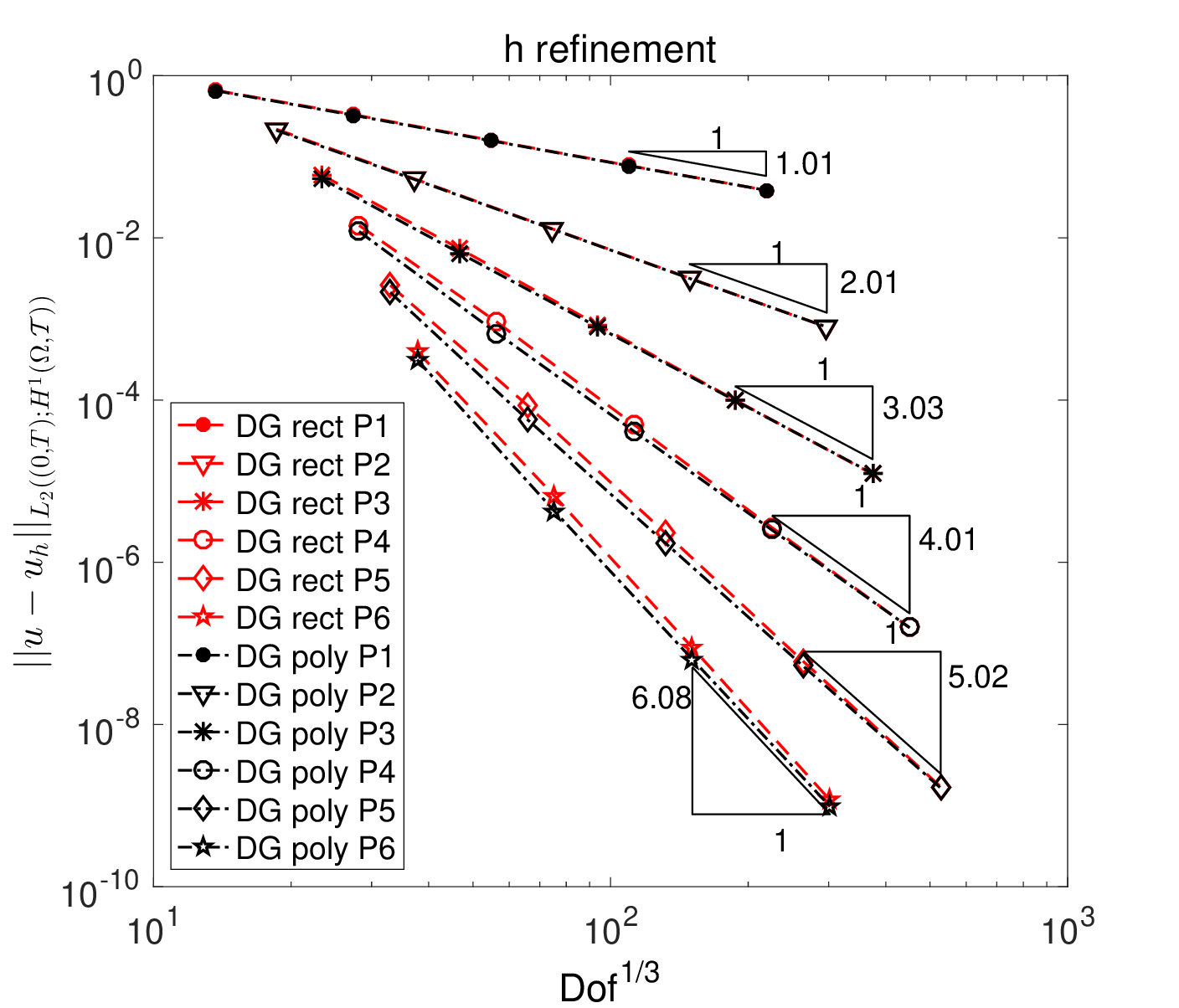}&
\hspace{-1cm}
\includegraphics[scale=0.29]{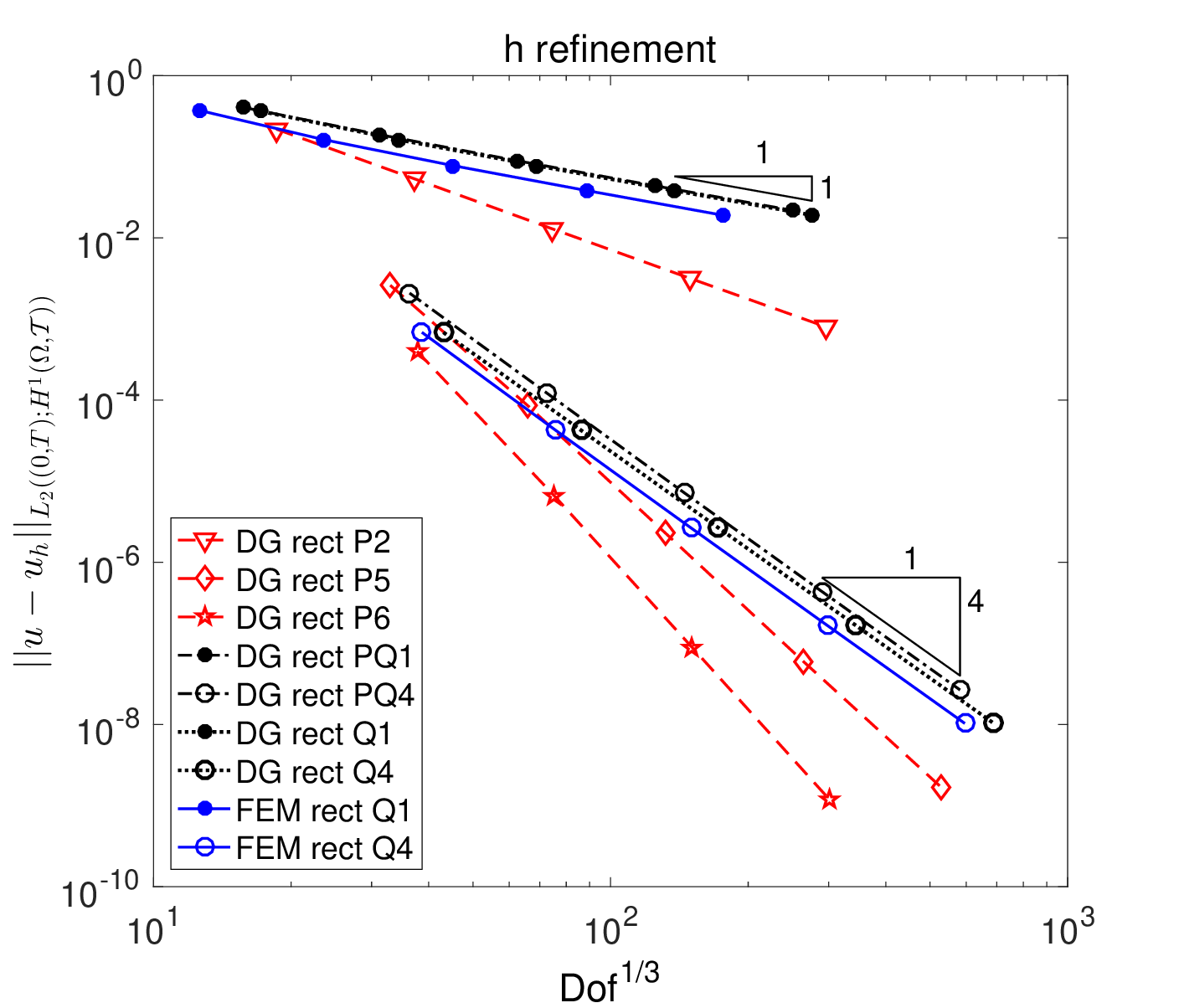} \\
\vspace{-0.3cm}
\includegraphics[scale=0.29]{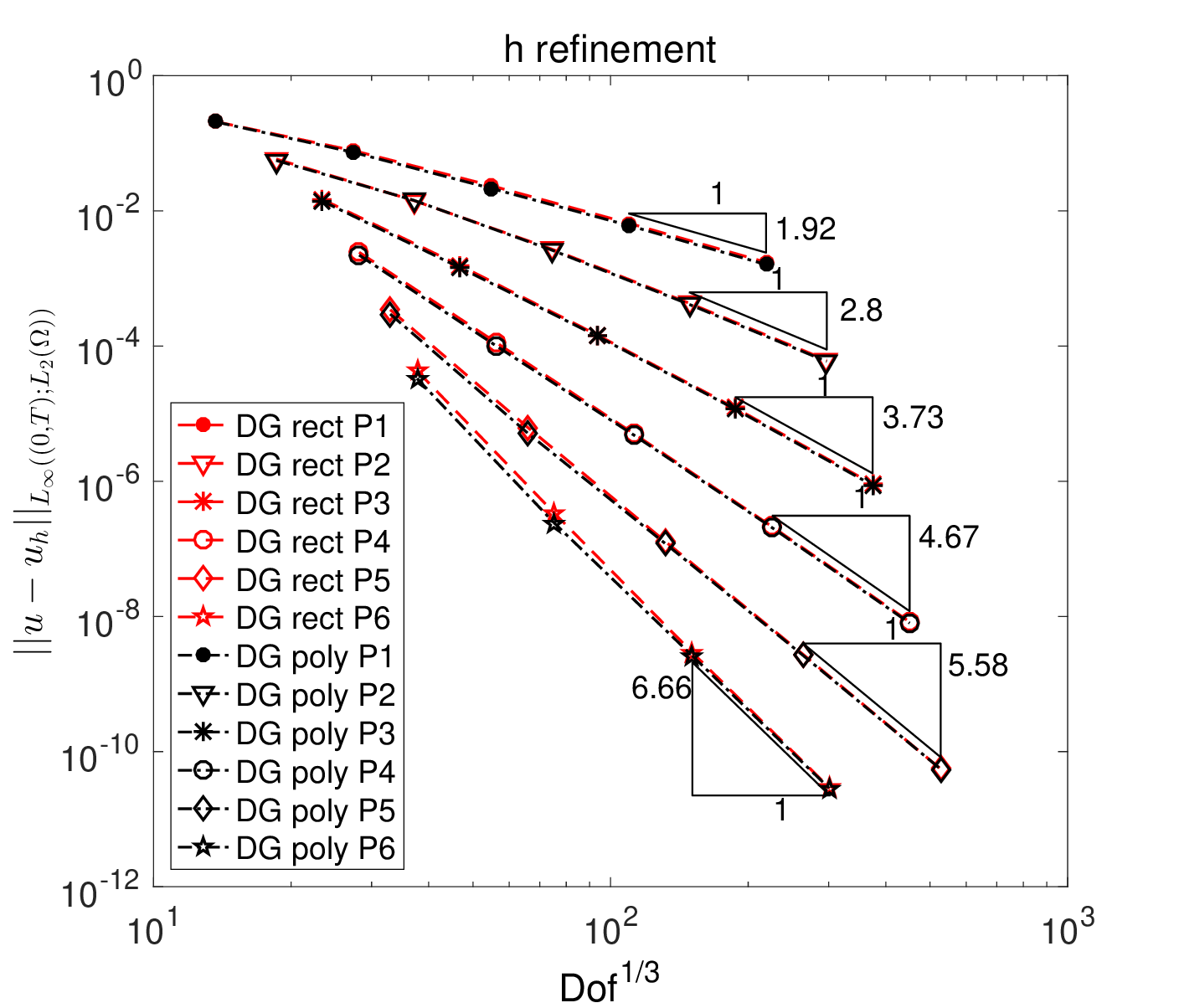}&
\hspace{-1cm}
\includegraphics[scale=0.29]{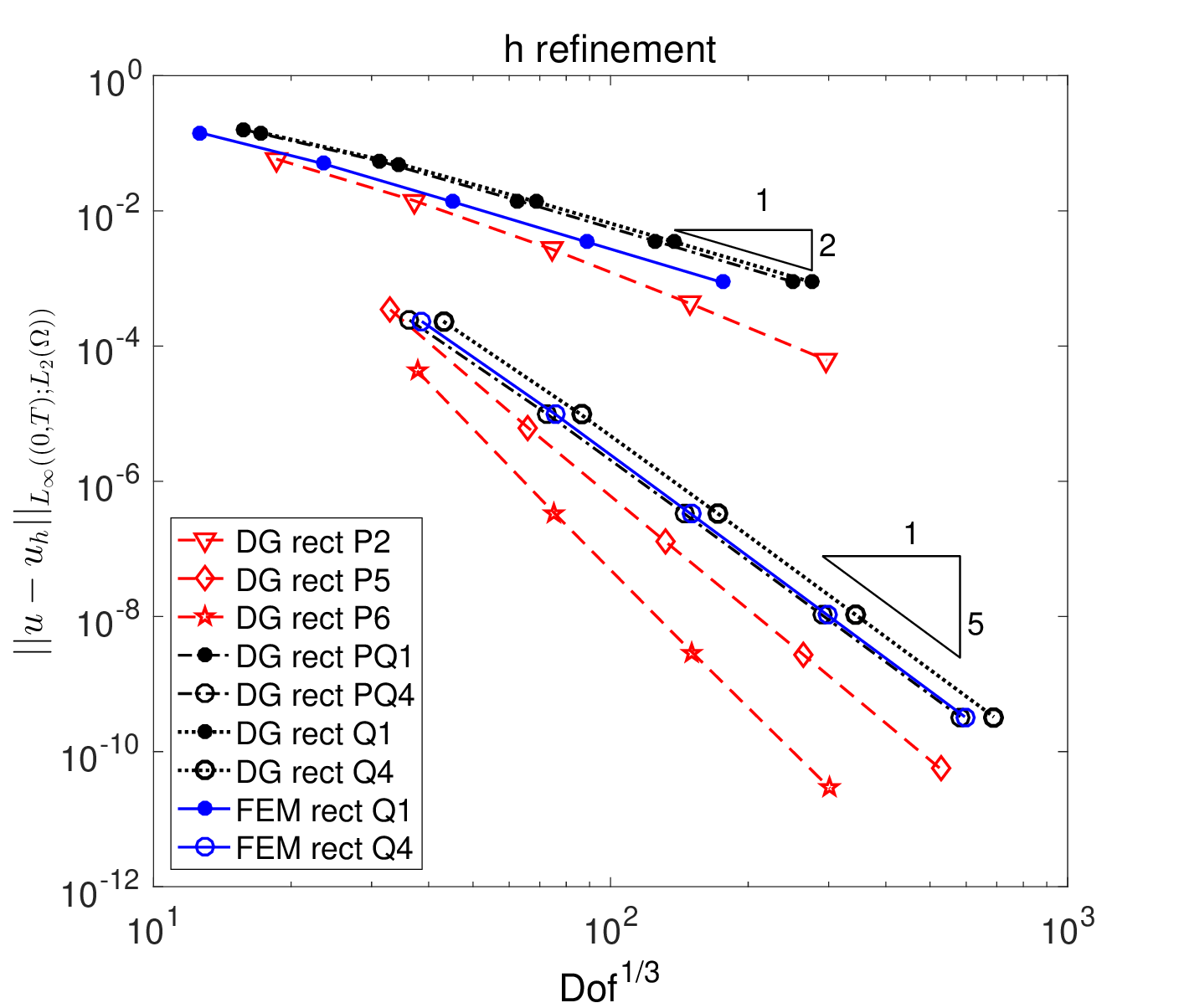} \\
\end{tabular}
\end{center}
\caption{Example 1. DG(P) under $h$--refinement (left) and comparison with other methods (right) for three different norms. }\label{Ex1_h_refinement}
\vspace{-0.8cm}
\end{figure}
\begin{figure}[h]
\begin{center}
\begin{tabular}{cc}
\vspace{-0.15cm}
\hspace{-0.5cm}
\includegraphics[scale=0.29]{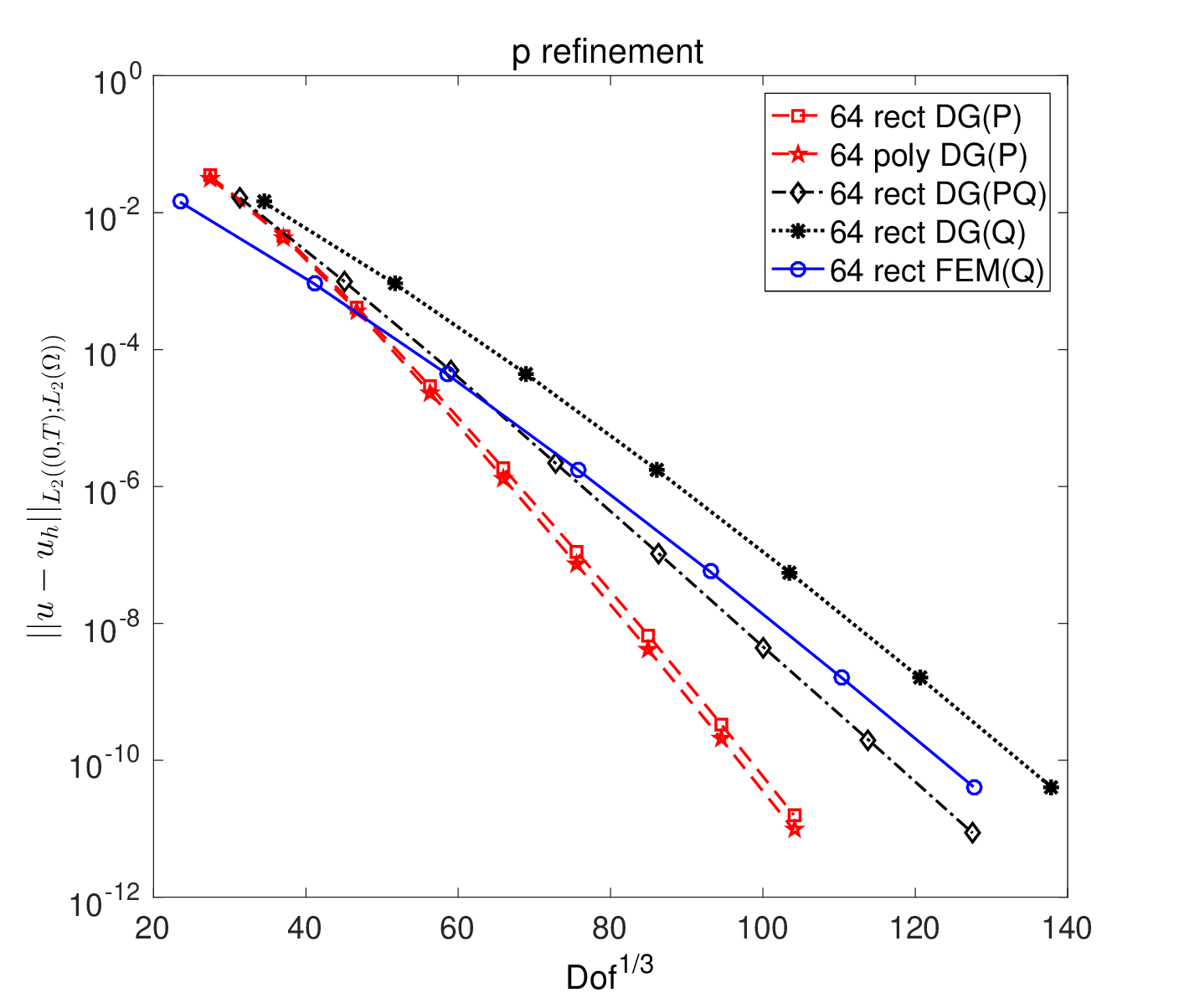}&
\hspace{-0.8cm}
\includegraphics[scale=0.29]{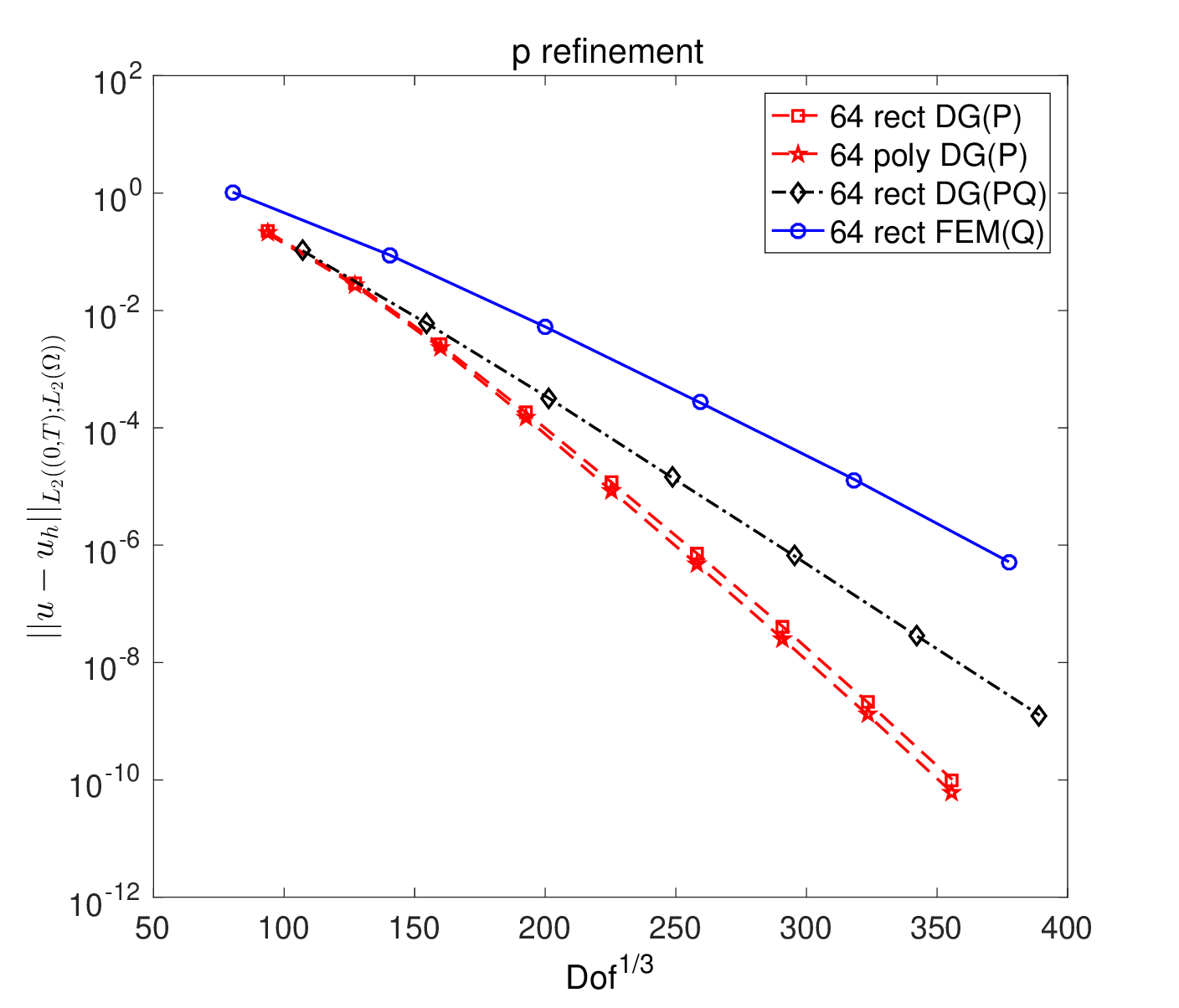} \\
\vspace{-0.15cm}
\hspace{-0.5cm}
\includegraphics[scale=0.29]{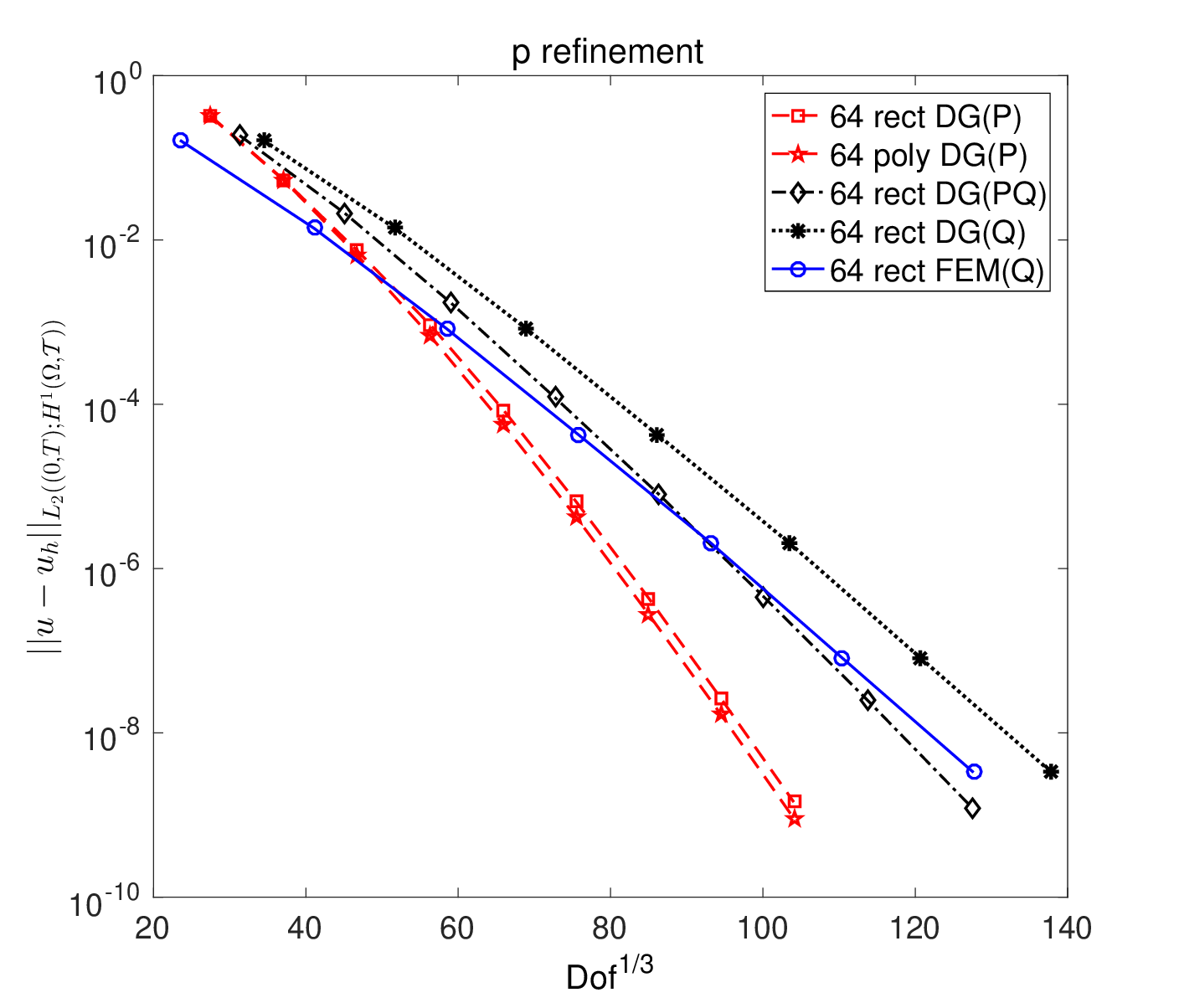}&
\hspace{-0.8cm}
\includegraphics[scale=0.29]{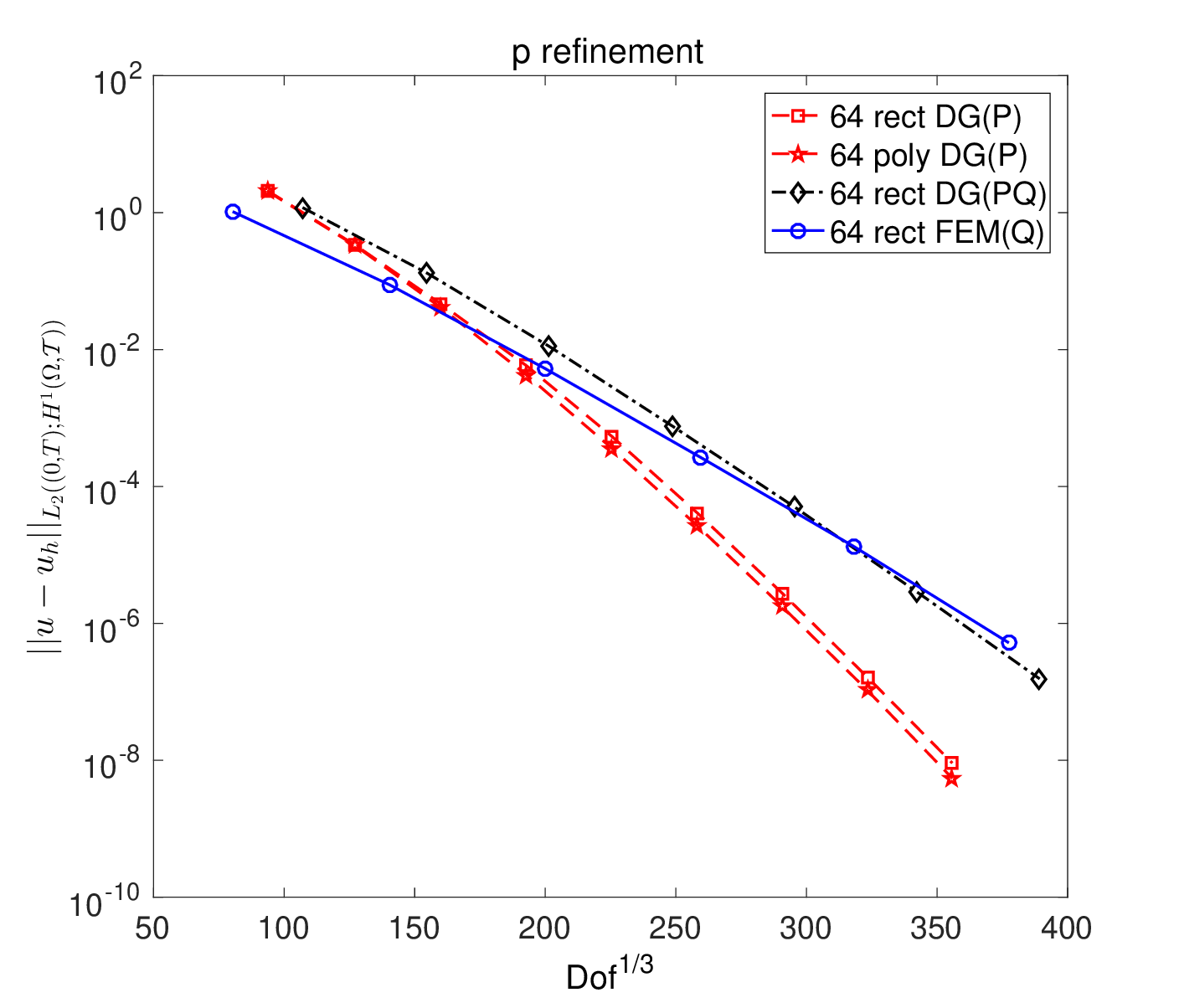} \\
\vspace{-0.5cm}
\hspace{-0.5cm}
\includegraphics[scale=0.29]{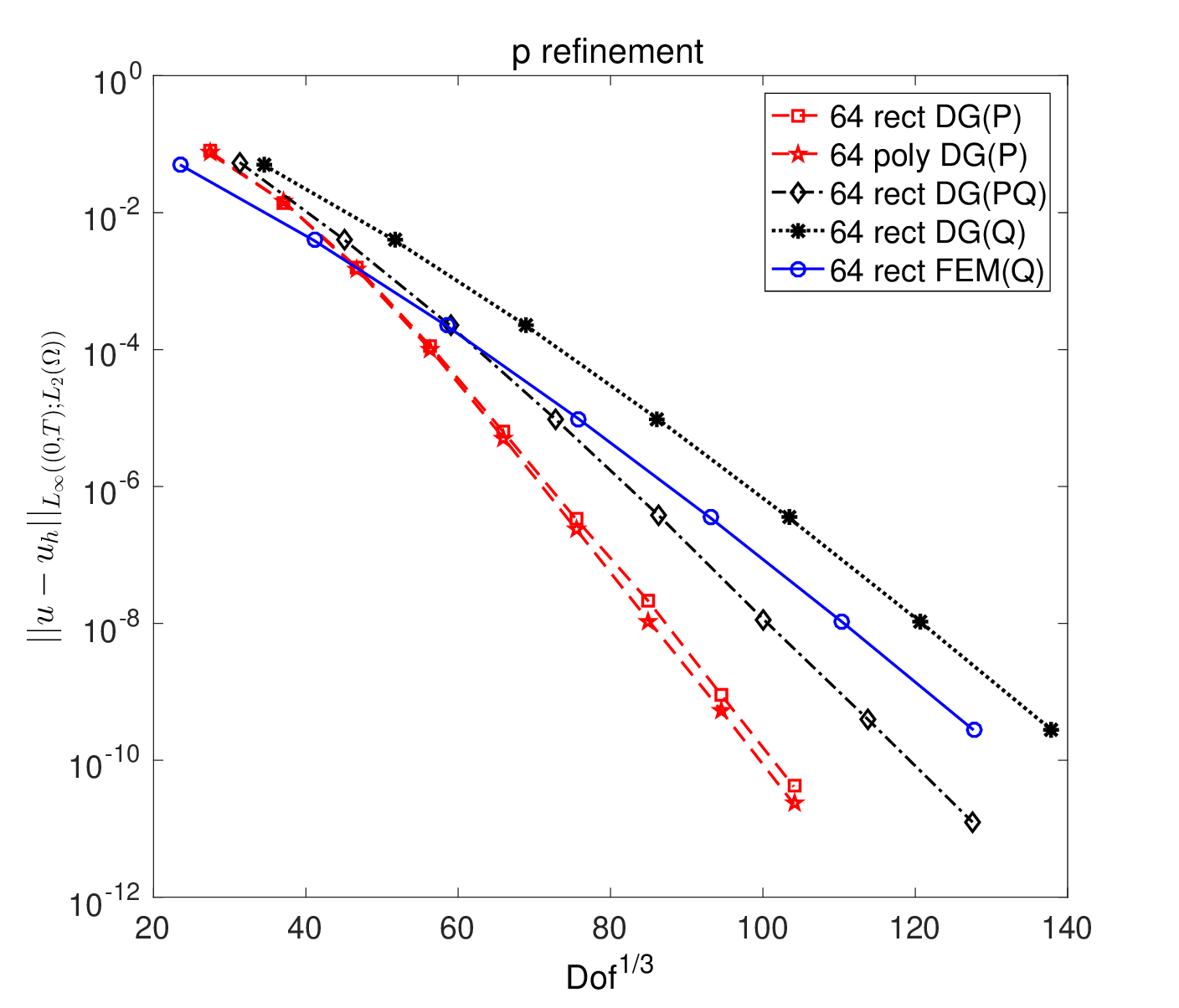}&
\hspace{-0.8cm}
\includegraphics[scale=0.29]{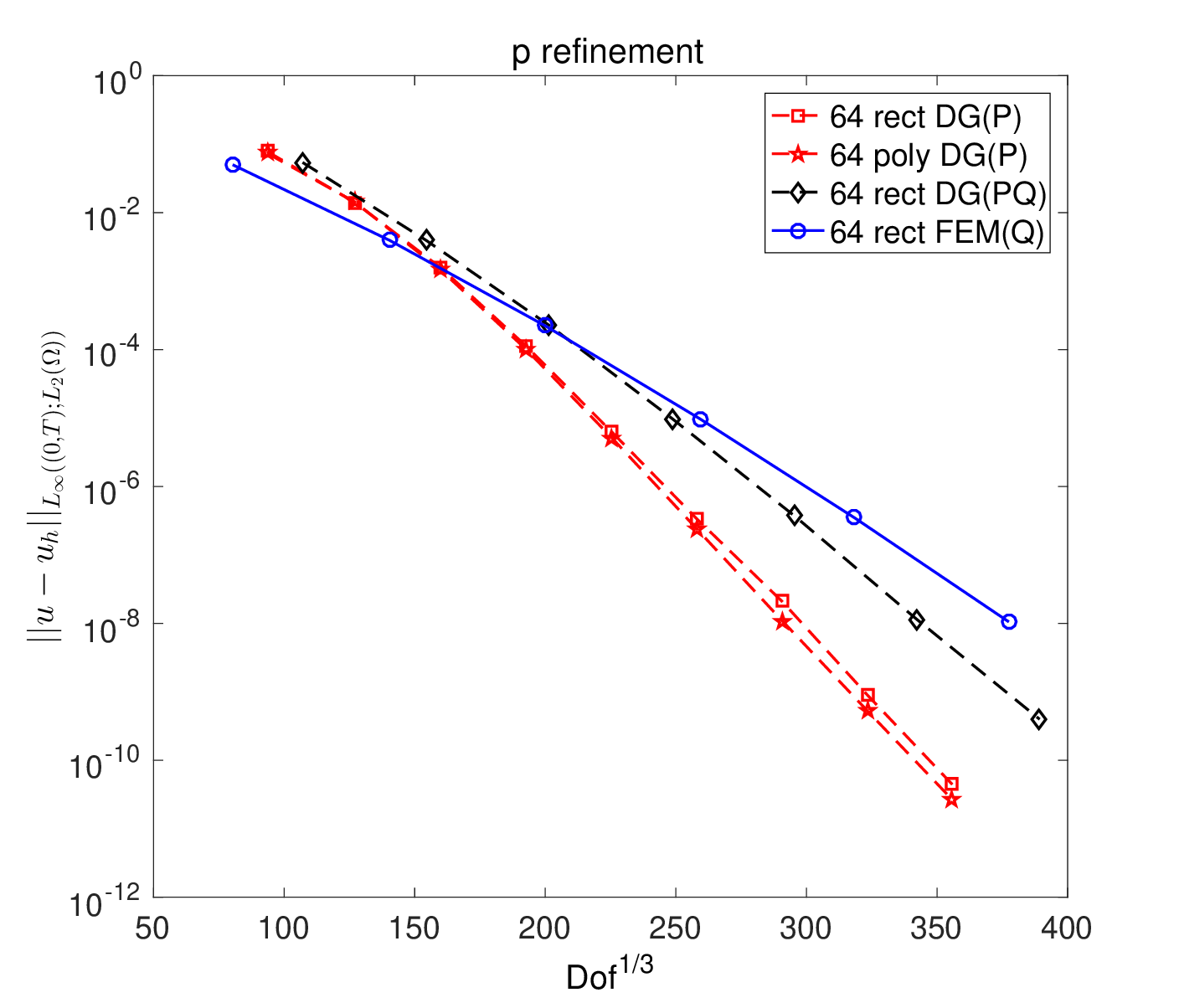} \\
\end{tabular}
\end{center}
\caption{Example 1. Convergence under $p$--refinement for $T=1$ with $80$ time steps (left) and for $T=40$ with $3200$ time steps (right) for three different norms. For (left) figures,  DG(P) with  $p=1,\dots,9$, DG(PQ) with $p=1,\dots,8$, DG(Q) and FEM(Q) with $p=1,\dots,7$.} \label{Ex1_p_refinement}
\vspace{-0.6cm}
\end{figure}

The left three plots in  Figure \ref{Ex1_h_refinement}, show the rate of convergence for the proposed dG scheme using $\mathcal{P}_p$ basis, for  $p=1,2,\dots,6$, on each 3-dimensional space-time element, against the total space-time degrees of freedom (Dof). This will be referred to as `DG(P)' for short, with `rect' meaning spatial rectangular elements and `poly' referring to general polygonal spatial elements in the legends. The observed rates of convergence are also given in the legends. The error appears to decay at essentially the same rate for both rectangle and polygonal spatial meshes, with very similar constants. Indeed, the DG(P) scheme  appears to converge at an optimal rate $\mathcal{O}(h^{p})$ in the $L_2(J;H^1(\Omega,\mathcal{T}))$--norm for $p = 1, 2,\dots, 6$ (cf.~Corollary \ref{cor_conv_l2h1}),  while the convergence appears to be slightly sub-optimal,  $\mathcal{O}(h^{p+1/2})$, in the $L_2(J;L_2(\Omega))$-- and $L_\infty(J;L_2(\Omega))$--norms. Nonetheless, the observed $L_2(J;L_2(\Omega))$--norm convergence rate is in accordance with the a priori bound of Theorem \ref{L2_L2 norm}.

To assess whether this marginal deterioration in the $h$-convergence rates for the DG(P) method, is an acceptable trade-off with respect to the number of degrees of freedom (Dof) gained by the use of reduced cardinality space-time local elemental basis, we present a comparison between $4$ different space-time schemes over rectangular space-time meshes in the right plots of Figure \ref{Ex1_h_refinement}. More specifically, we compare the proposed DG(P) method, against the time-dG method with: 1) discontinuous tensor-product space-time bases consisting of $\mathcal{P}_p$-basis in space ('DG(PQ)' for short), 2) full discontinuous tensor-product $\mathcal{Q}_p$ basis in space  ('DG(Q)' for short) and, 3) the standard finite element method with conforming tensor-product $\mathcal{Q}_p$ basis in space  ('FEM(Q)' for short) \cite{thomee1984galerkin,schotzau2000time}. Unlike the proposed DG(P) scheme, the three other methods achieve the optimal $h$-convergence rate in the three different norms: $\mathcal{O}(h^{p+1})$ in $L_2(J;L_2(\Omega))$-- and $L_\infty(J;L_2(\Omega))$--norms and $\mathcal{O}(h^{p})$ in $L_2(J;H^1(\Omega,\mathcal{T}))$--norm, respectively. Nevertheless, plotting the error against the total degrees of freedom, a more relevant measure of computational effort,  we see, for instance, that DG(P) with $p=2$ use less Dofs compared to the other $3$ methods with $p=1$, to achieve the same level of accuracy, at least for relatively large number of space time elements. {More pronounced gains are observed when comparing DG(P) with $p=5,6$ with  the other methods with $p=4$, across all mesh sizes and error norms.} Analogous results hold for DG(P) with $p=3,4$.

Moving on to $p$-version, Figure \ref{Ex1_p_refinement} shows the error for all four methods in the three different norms for fixed space-time mesh size under $p$-refinement. The left three plots are with final time $T=1$, for fixed $64$ spatial elements and $80$ time steps. As expected, exponential convergence is observed since the solution to \eqref{example1} is analytic over the computational domain. However, the convergence slope for DG(P) with both rectangular and polygonal spatial elements appears to be steeper compared to the other $3$ methods. Indeed, DG(P) achieves the same level of accuracy for $p\ge 3$ with less number of Dofs in all $3$ different norms.

The right three plots for the same computation run for a longer time interval with final time $T=40$, that is $3200$ time-steps. Since DG(P) use less Dofs per space-time element compared to the other three methods, the acceleration of $p-$convergence for the DG(P) is expected to be more pronounced for long time computations. Again DG(P) achieves the same level of accuracy with fewer degrees of freedom for $p\ge 3$. For instance, the total DG(P) Dofs for this problem are about $45$ million when $p=9$, compared to about $53$ million Dofs with $p=6$ for FEM(Q), while the error for DG(P) is about $100$ times smaller than the error of FEM(Q) in all three norms.

\begin{figure}[t]
\begin{center}
\begin{tabular}{cc}
\vspace{-0.5cm}
\hspace{-0.4cm} \includegraphics[scale=0.29]{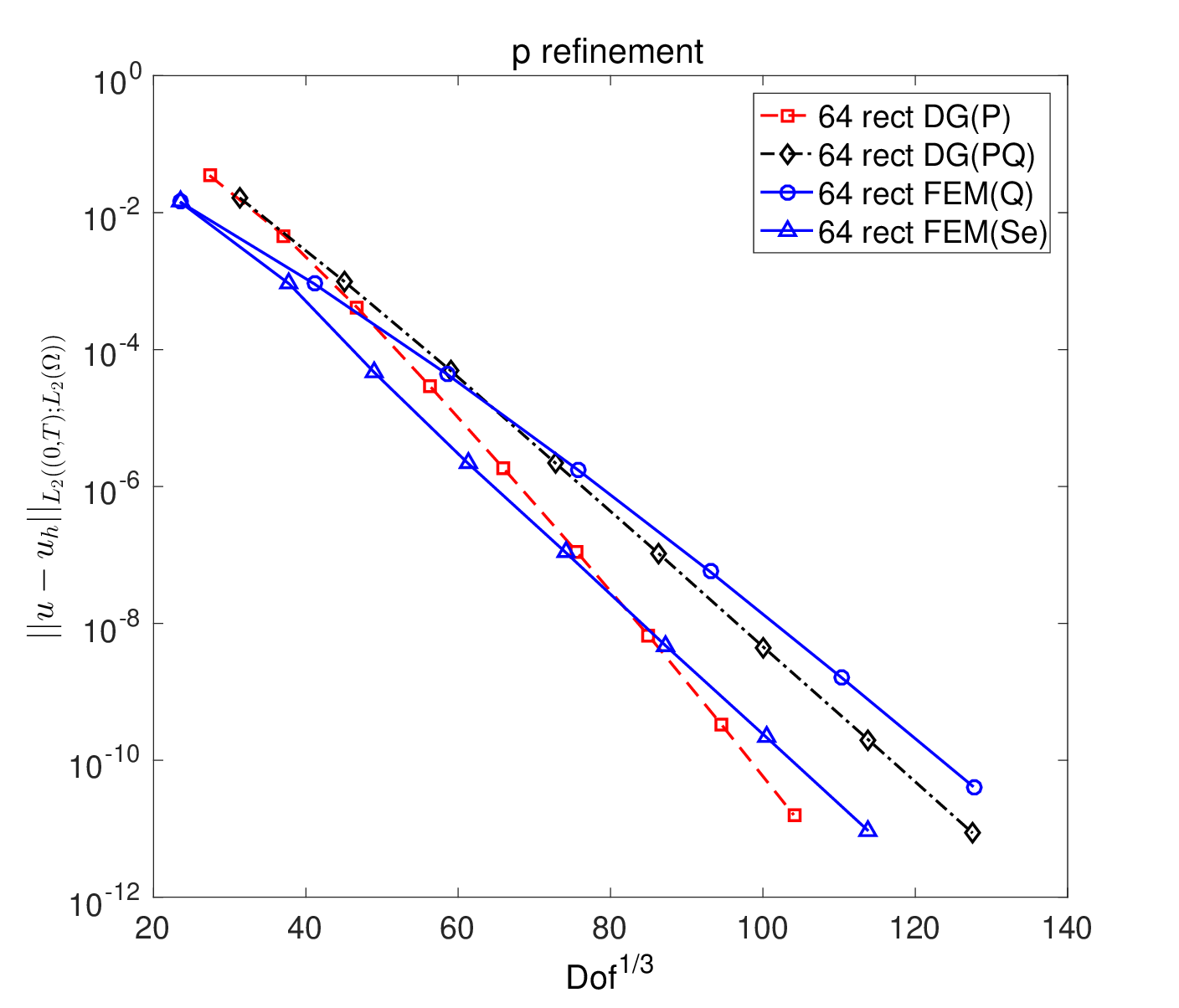}&
\hspace{-1cm}
\includegraphics[scale=0.29]{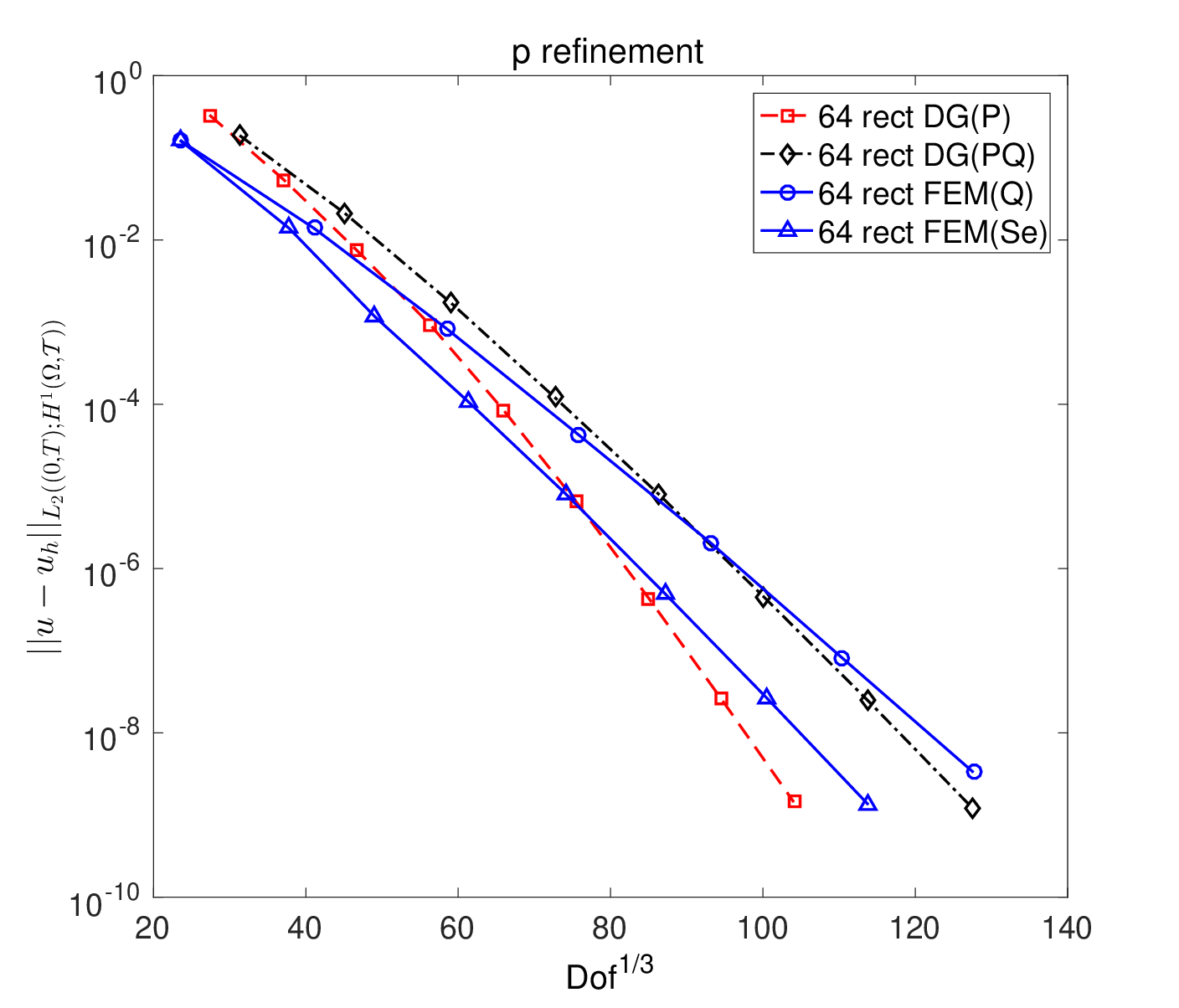} \\
\end{tabular}
\end{center}
\caption{Example 1. Convergence under $p$--refinement for $T=1$ with $80$ time steps for two different norms. For above figures,  DG(P) with  $p=1,\dots,9$, DG(PQ) with $p=1,\dots,8$, FEM(Q) with $p=1,\dots,7$ and FEM(Se) with $p=1,\dots,8$.} \label{Ex1_p_refinement_serendipity}
\vspace{-0.3cm}
\end{figure}

Next, we investigate the convergence performance of the proposed approach against dG-time stepping spatially conforming FEM with the cheaper conforming serendipity elements in space on hexahedral space-time meshes. Numerical results under $p$-refinement are given in Figure~\ref{Ex1_p_refinement_serendipity}, with FEM(Se) standing for the latter method. We note that for $d=2$, the cardinality of the local serendipity space equals the cardinality of $\mathcal{P}_p$--basis plus two more Dofs. We observe that  the convergence slope of FEM(Se) is steeper than that of FEM(Q) and almost parallel to DG(PQ), but it is still not steeper than the convergence slope of DG(P). We observe that  DG(P) with $p=7$ gives smaller error against Dofs than FEM(Se) with $p=6$.  Noting that serendipity basis in three dimensions uses considerably more Dofs compared to total degree $\mathcal{P}_p$-basis, it is expected that  DG(P) will achieve smaller error for the same Dofs than FEM(Se) with lower order that $7$ polynomials for $d=3$.

\begin{figure}[t]
\begin{center}
\begin{tabular}{cc}
\vspace{-0.5cm}
\hspace{-0.4cm} \includegraphics[scale=0.29]{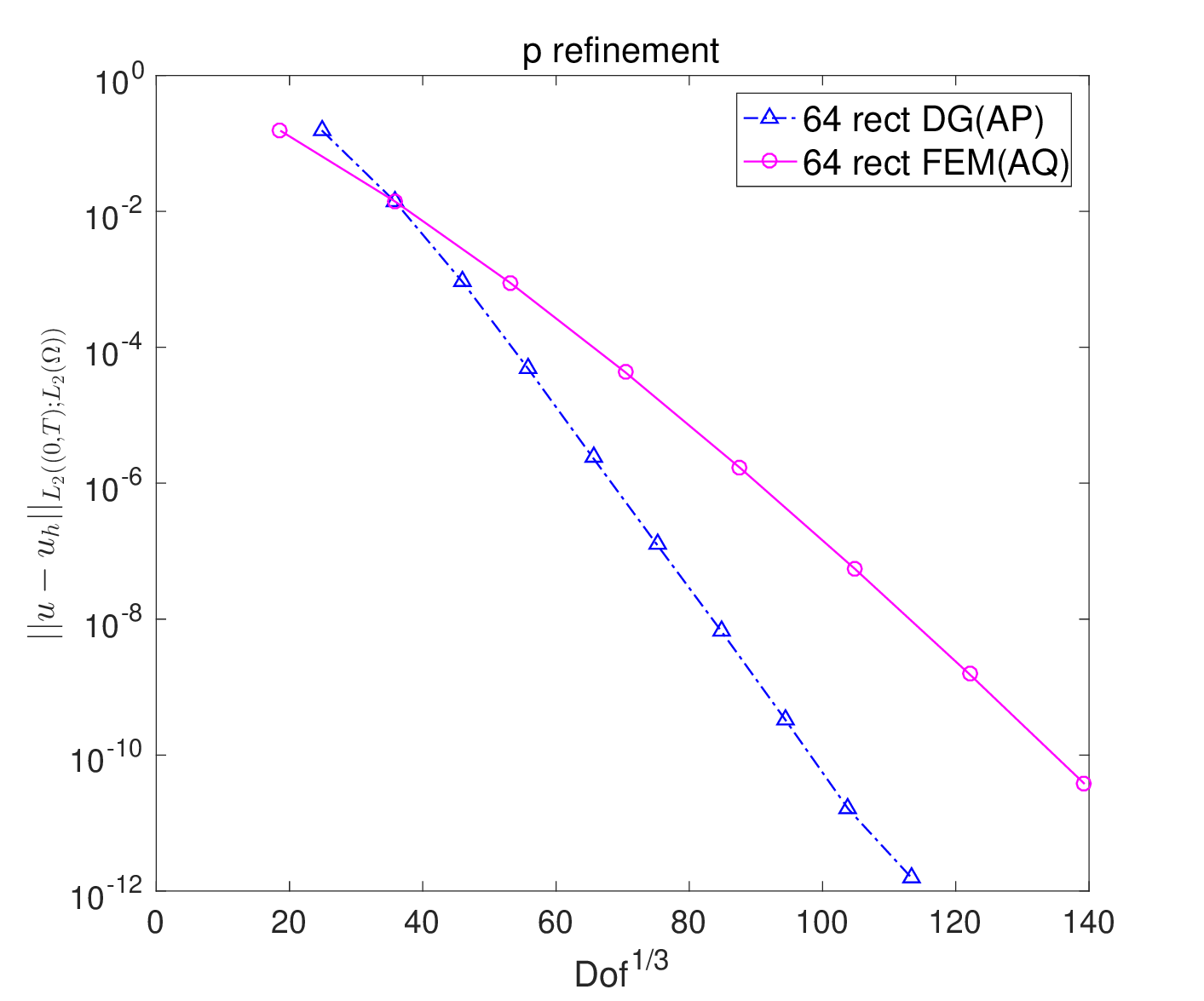}&
\hspace{-1cm}
\includegraphics[scale=0.29]{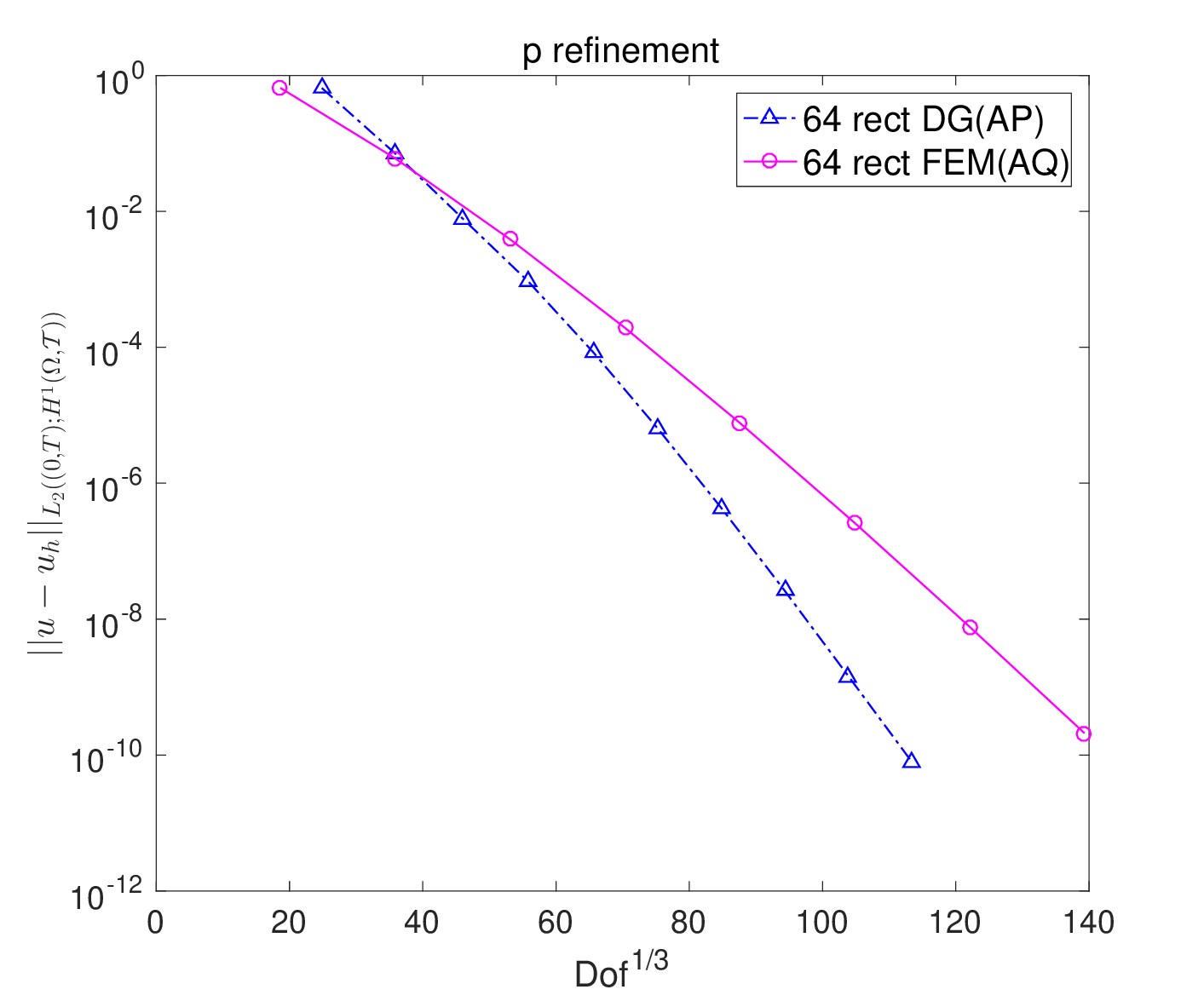} \\
\end{tabular}
\end{center}
\caption{Example 1. Convergence under $p$--refinement for $T=1$ with $80$ time steps for two different norms.For above  figures,  DG(AP) with  $p=1,\dots,10$ and  FEM(AQ) with $p=1,\dots,8$.} \label{Ex1_p_refinement_anisotropic}
\vspace{-0.3cm}
\end{figure}

Finally, we investigate the convergence performance of the proposed dG scheme with anisotropic space-time $\mathcal{P}_p$ basis against dG-time stepping spatially conforming FEM with anisotropic space-time $\mathcal{Q}_p$ basis  on hexahedral space-time meshes. Numerical results under $p$-refinement are given in Figure~\ref{Ex1_p_refinement_anisotropic}. Here, DG(AP) uses a reduced space-time $\mathcal{P}_p$ basis where the basis function with order $p$ in the temporal variable is removed, and FEM(AQ) uses the tensor-product space-time basis with order $p-1$ in the temporal variable and $\mathcal{Q}_p$ basis on spatial variables. We observe that  the convergence slope of DG(AP) is steeper than that of FEM(AQ). DG(AP)  achieves the same level of accuracy with fewer degrees of freedom for $p>2$.

\begin{figure}[t!]
\begin{center}
\begin{tabular}{cc}
\vspace{-0.15cm}
\hspace{-0.5cm}
\includegraphics[scale=0.28]{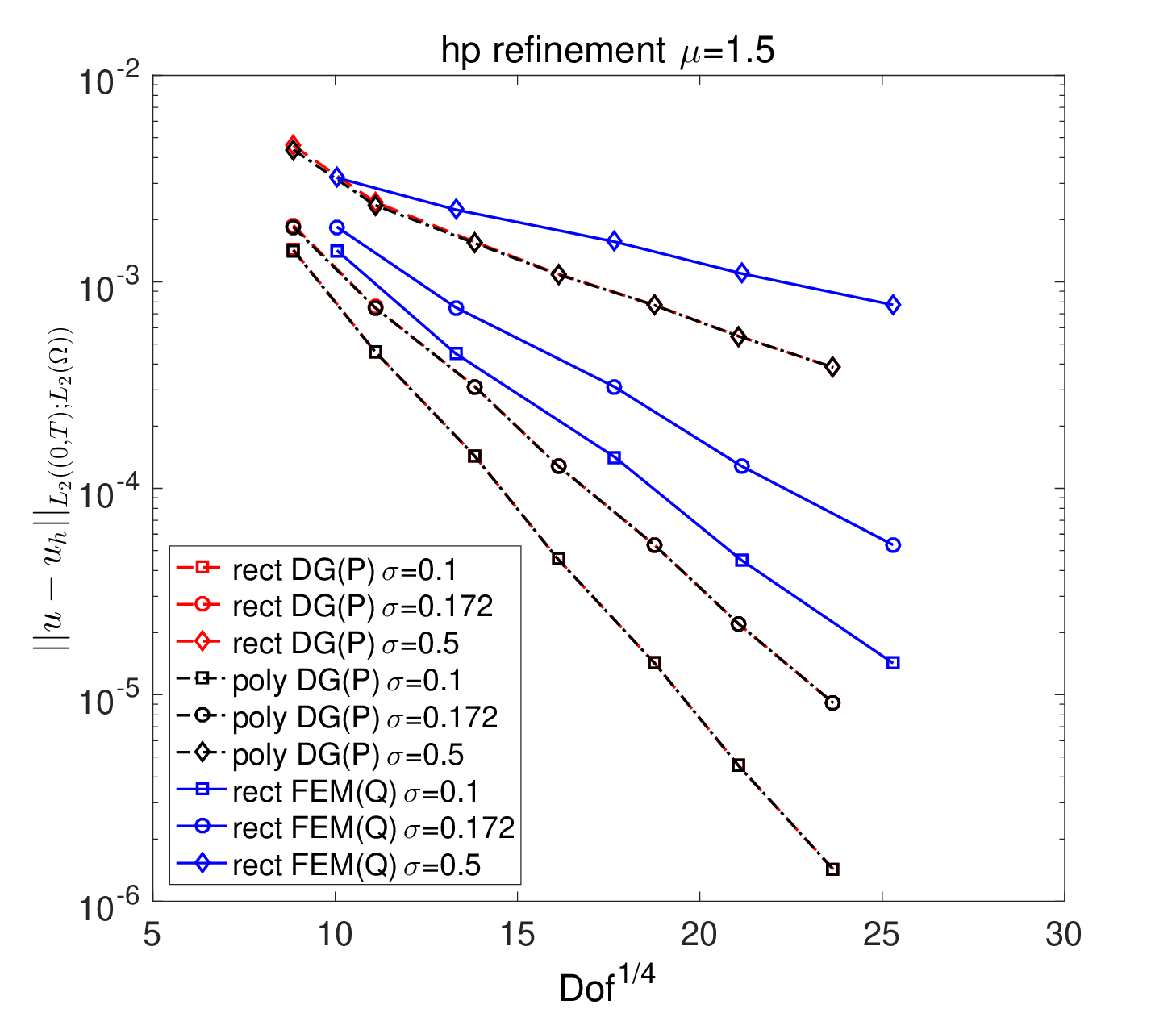}&
\hspace{-0.8cm}
\includegraphics[scale=0.28]{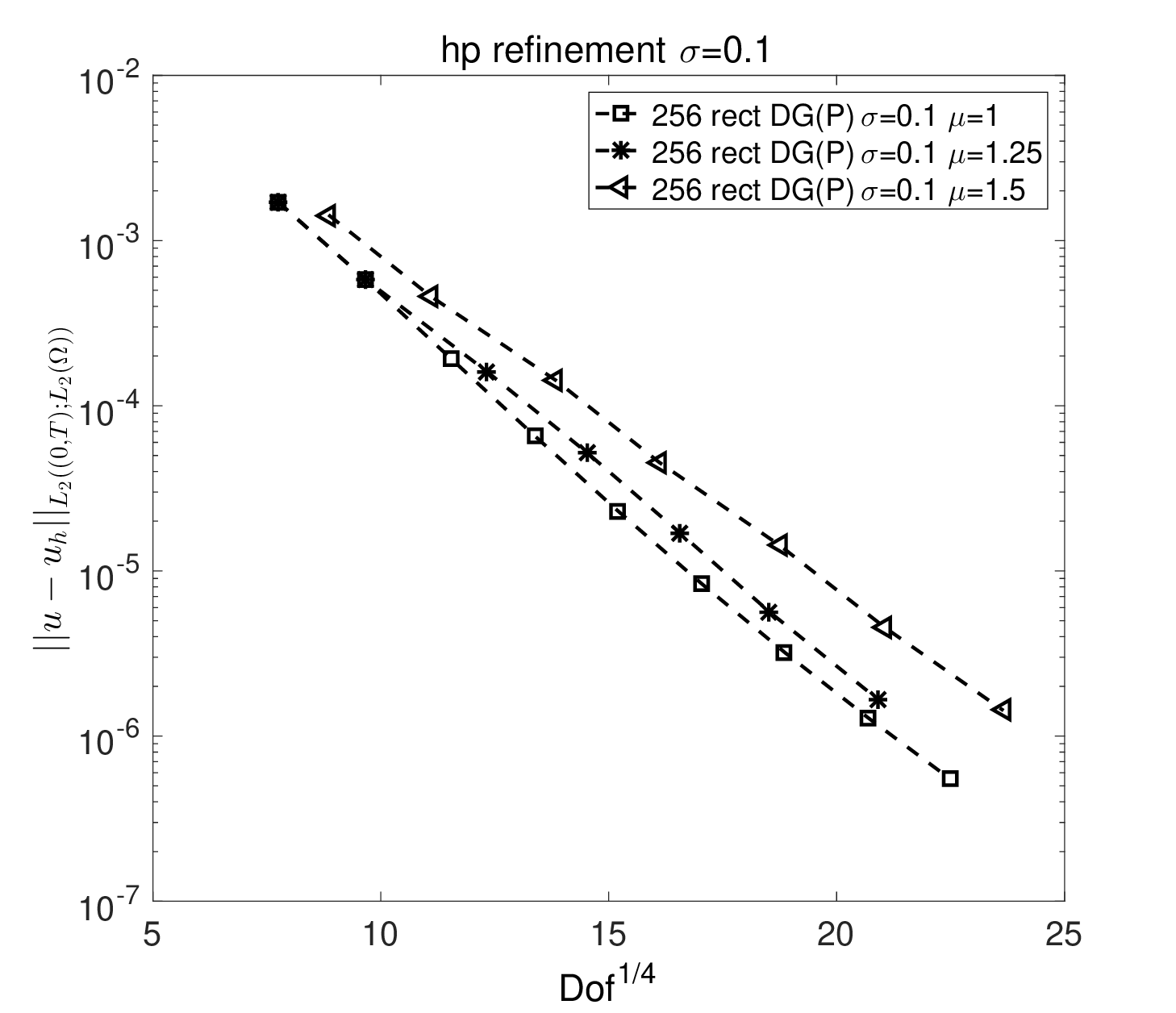} \\
\vspace{-0.15cm}
\hspace{-0.5cm}
\includegraphics[scale=0.28]{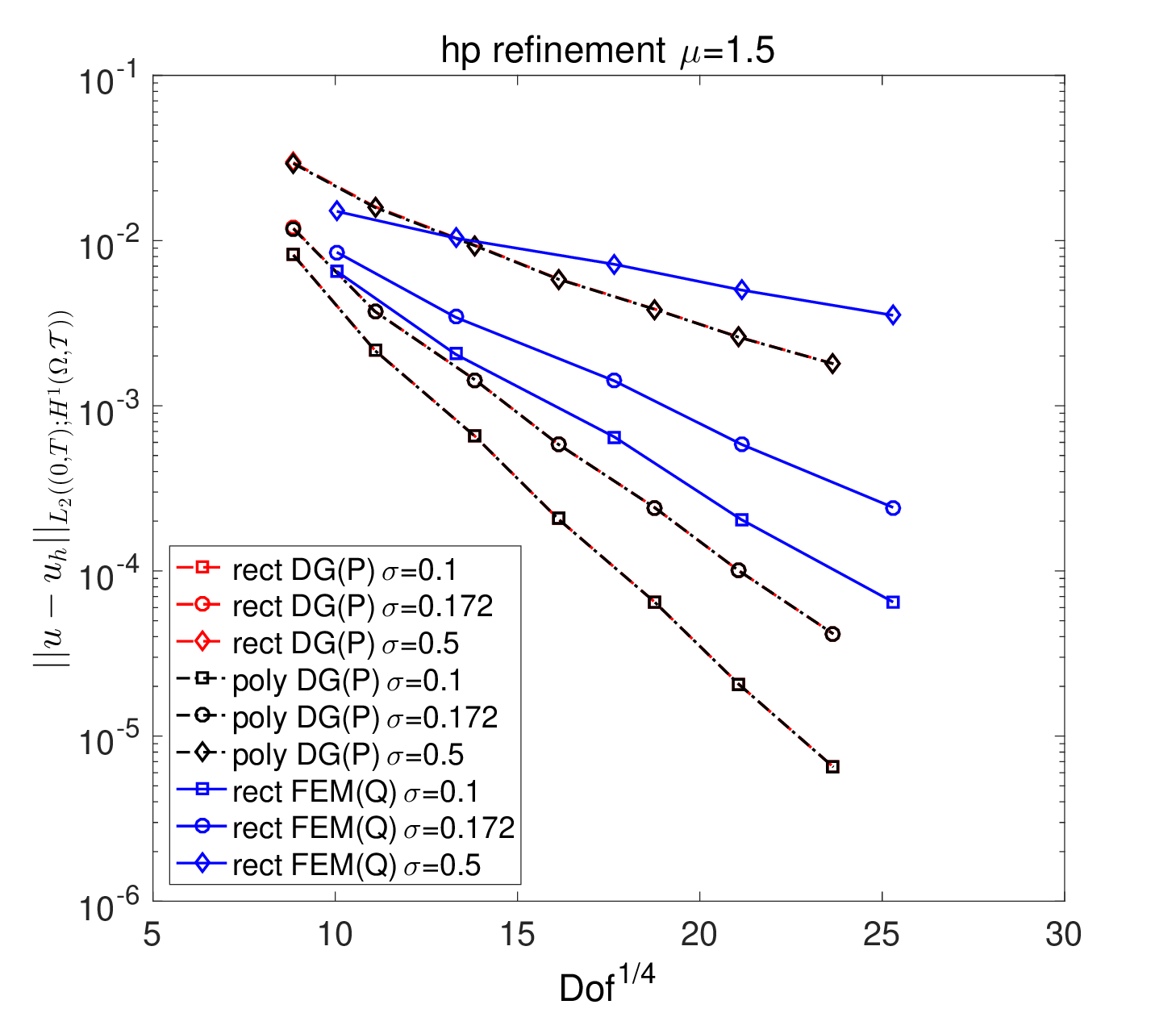}&
\hspace{-0.8cm}
\includegraphics[scale=0.28]{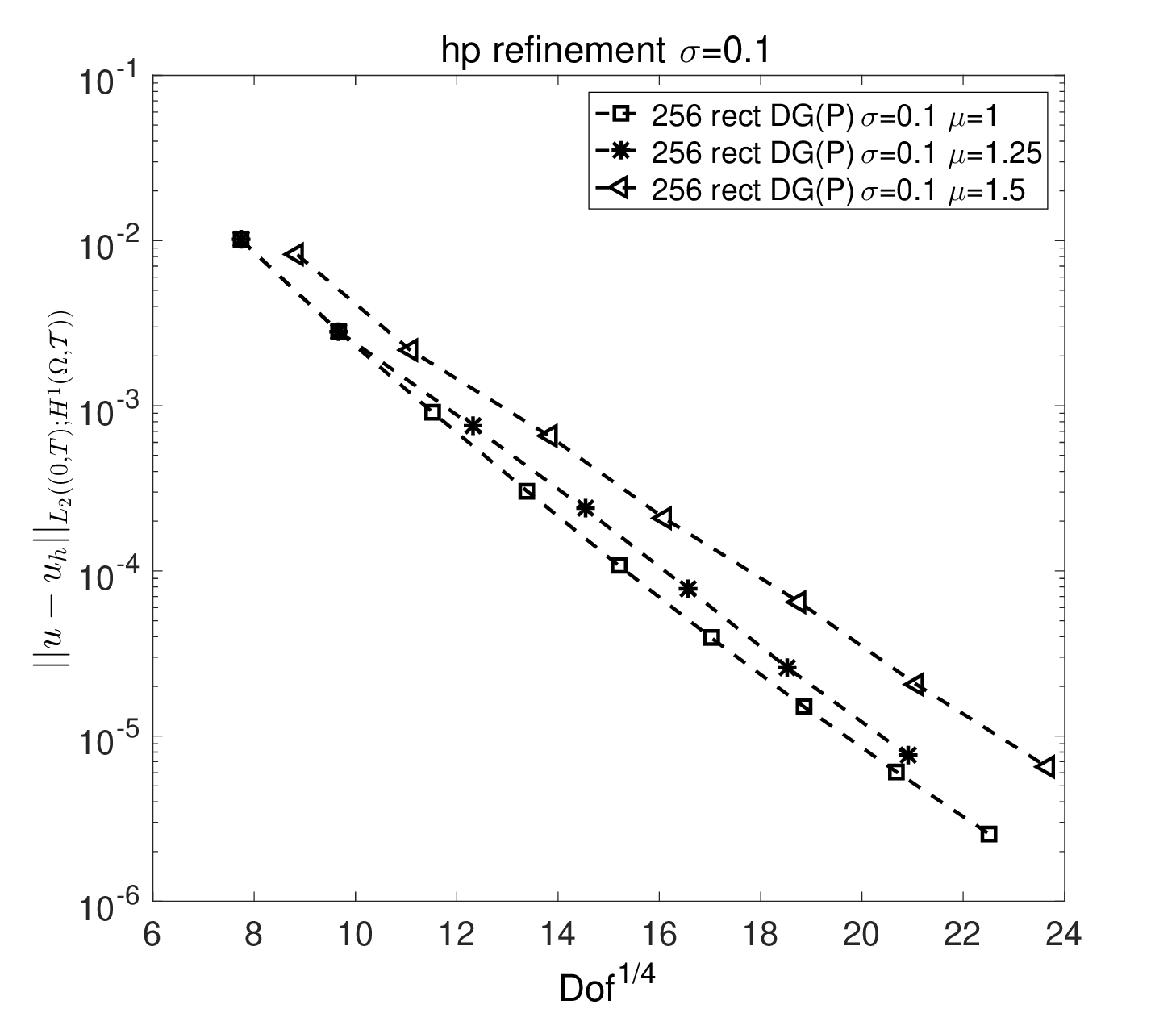} \\
\vspace{-0.15cm}
\hspace{-0.5cm}
\includegraphics[scale=0.28]{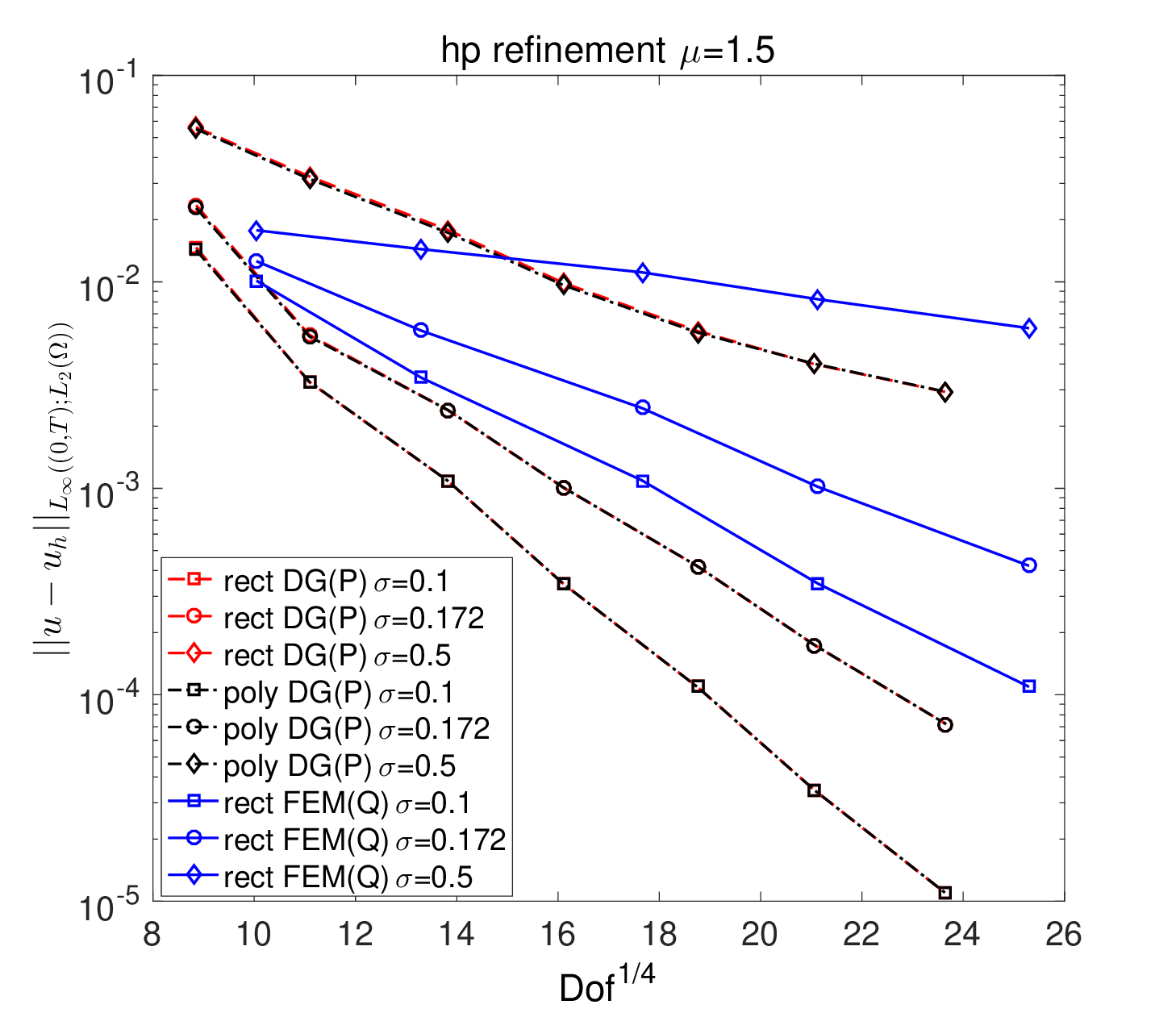}&
\hspace{-0.8cm}
\includegraphics[scale=0.28]{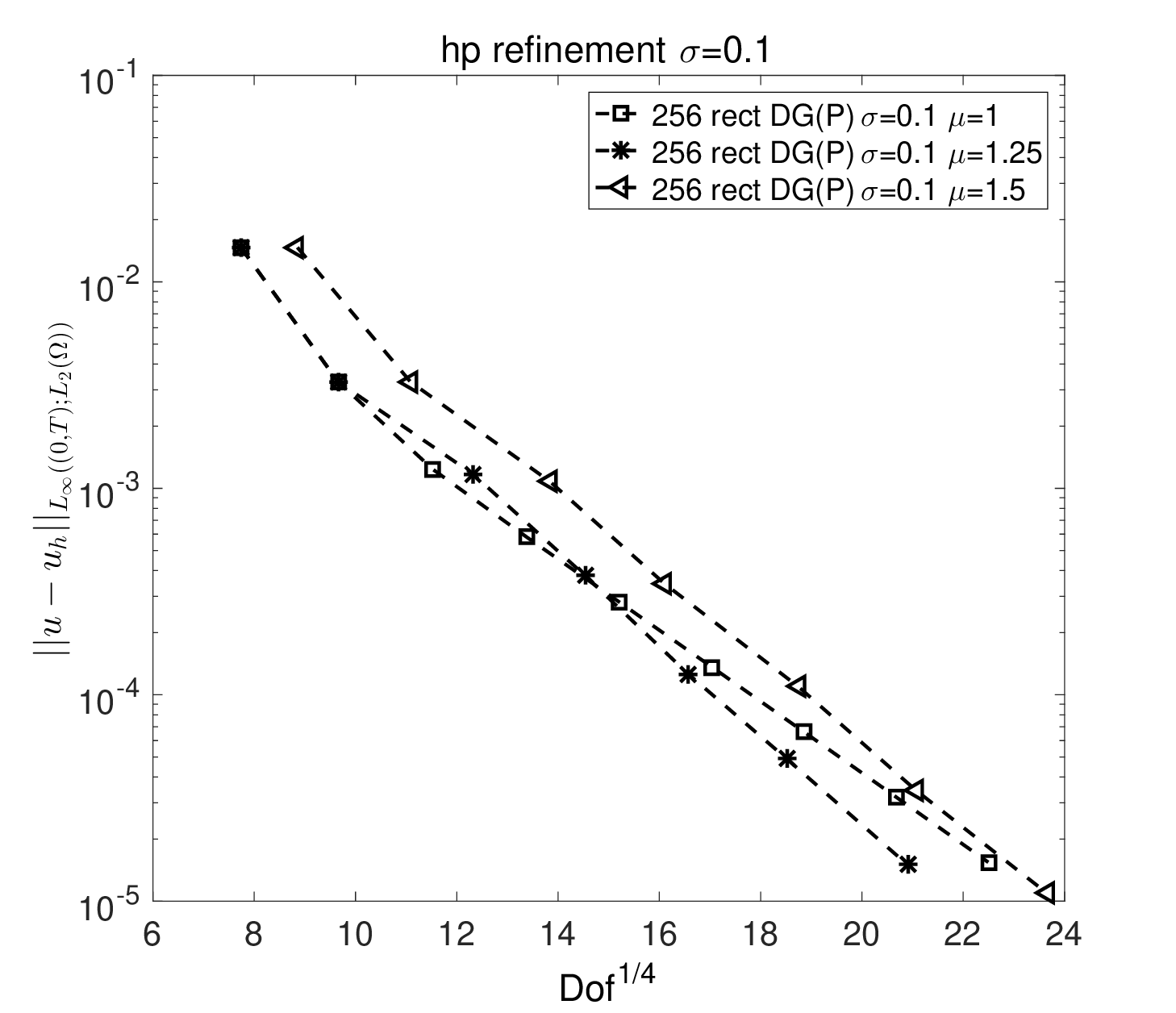} \\
\end{tabular}
\end{center}
\caption{Example 2: Convergence under $hp$--refinement with fixed $\mu=1.5$ (left); with fixed $\sigma=0.1$ (right) for three different norms.}\label{Ex2_hp_refinement}
\vspace{-.5cm}
\end{figure}

\subsection{Example 2}
We shall now assess the performance of the $hp$-version of the proposed method for a problem with initial layer. For $\bold{a}(x,y,t)$ being the identity matrix, $u_0$ and $f$ are chosen so that the exact solution of \eqref{Problem} is given by
\begin{equation}\label{example2}
u(x,y,t)= t^{\alpha}\sin(\pi x)\sin(\pi y) \quad \text{in } J \times  \Omega ,
\end{equation}
for $J =(0,0.1)$ and $\Omega = (0,1)^2$. We set $\alpha=1/2$, so that $u\in H^{1-\epsilon}(J; L_2(\Omega))$, for all $\epsilon>0$. This problem is analytic over the spatial domain, but has low regularity at $t=0$. To achieve exponential rates of convergence, we use temporal meshes, geometrically graded towards $t=0$, in conjunction with temporally varying polynomial degree $p$, starting, from $p=1$ on the elements belonging to the initial time slab, and linearly increasing $p$ when moving away from $t=0$; see \cite{schwab,schotzau2013hp,schotzau2000time} for details. Following \cite{schotzau2000time}, we consider a short time interval with $T = 0.1$. Let $0<\sigma<1$ be the mesh grading factor which defines a class of temporal meshes $t_n=\sigma^{N-n}\times 0.1$, for $n=1, \dots, N$. Let also $\mu$ be the polynomial order increasing factor determining the  polynomial order over different time steps by $p_{\k_n} := \ujump{\mu n}$ for for $n=1, \dots, N$.

The three left plots in Figure \ref{Ex2_hp_refinement} show the convergence history for DG(P) and FEM(Q) for this problem. All computations are performed over $256$ spatial elements with geometrically graded temporal meshes based on $3$ different grading factors $\sigma=0.1, 0.172, 0.5$ and fixed $\mu=1.5$. The error for both DG(P) and FEM(Q) appears to decay exponentially under the $hp$ refinement strategy described above for all three grading factors considered. The choice of $\sigma= 0.5$, is motivated by the meshes constructed in standard adaptive algorithms; $\sigma = 0.172$, is classical in that it was shown that it is the optimal grading factor for one-dimensional functions with $r^\alpha$-type singularity \cite{gui1986h}, while $\sigma = 0.1$ appears to be a better choice in the current context.  We also note that the convergence rate of DG(P) appears to be steeper than FEM(Q) under the same mesh and polynomial distribution. Furthermore, performing the same experiments on general polygonal spatial meshes, we observe that the error decay does not appear to depend on the shape of the spatial elements. This is expected, as the error in the time variable dominates in this example.

For completeness, we also report on how the choice of the  polynomial order increasing factor $\mu$ influences the exponential error decay for DG(P) with  fixed mesh grading factor $\sigma=0.1$; these are given in the three right plots in Figure \ref{Ex2_hp_refinement}. For both $L_2(J;L_2(\Omega))$-- and $L_2(J;H^1(\Omega,\mathcal{T}))$--norms, the results show that $\mu=1$ gives the fastest convergence, while $\mu=1.25$ gives the fastest error decay in the $L_\infty(J;L_2(\Omega))$--norm.

\subsection{Example 3}

Finally, we consider an example with a rough boundary to highlight the flexibility in domain approximation offered by the use of polytopic meshes within the context of the proposed dG scheme. Let $\bold{a}(x,y,t)$ be the identity matrix and let $f\equiv 1$. The domain $\Omega$, illustrated in Figure \ref{different meshes}, is constructed by removing small triangular regions attached to the boundary of a square domain. We set $u=0$ on $\partial \Omega$, $u|_{t=0} = 0$ and $J=(0,1)$.  The problem's solution is not known. As reference solution, we shall use the DG(P) solution on a fine uniform mesh made of $15624$ triangles and $256$ time steps with $p=3$.

We apply the proposed DG(P) scheme on two different meshes built as follows.
In both cases, we start from the reference uniform triangular mesh, fine enough to resolve the small scale structures on the boundary. Each mesh is iteratively coarsened, keeping into consideration that the micro-structures of the boundary need more resolution than the interior of the domain.  The first mesh is conforming, it consists of $616$ triangles and $900$ quadrilaterals; see Figure~\ref{different meshes} (left). The second mesh is polygonal and is made up of $225$ quadrilateral elements and $124$ polygonal elements; see  Figure~\ref{different meshes} (right). The second mesh achieves a similar resolution of the boundary with an overall coarser subdivision of the computational domain. Note that here the square elements neighbouring the finer polygonal elements in the second mesh are treated as polygons with more than four faces, some of which being co-linear, rather than traditional square elements with hanging nodes.  Numerical results obtained with these two meshes for $16$ and $32$ time steps are reported in Table 1. 
For the first mesh, linear basis functions are used, while for the second, polygonal, mesh quadratic basis functions are employed. We observe that the polygonal mesh is more accurate using smaller number of degrees of freedom. We note that the error is dominated by the spatial error in the vicinity of the boundary as can be seen by comparing the results obtained with $16$ and $32$ time steps, respectively. These results suggest that the use of general meshes and appropriate polynomial spaces may have a potential in achieving the same level of accuracy with fewer degrees of freedom, as they permit a more aggressive grading towards complicated features of the computational domain.

\begin{figure}[t]
\begin{center}
\begin{tabular}{cc}
\vspace{-0.5cm}\hspace{-1cm}
\includegraphics[scale=0.39]{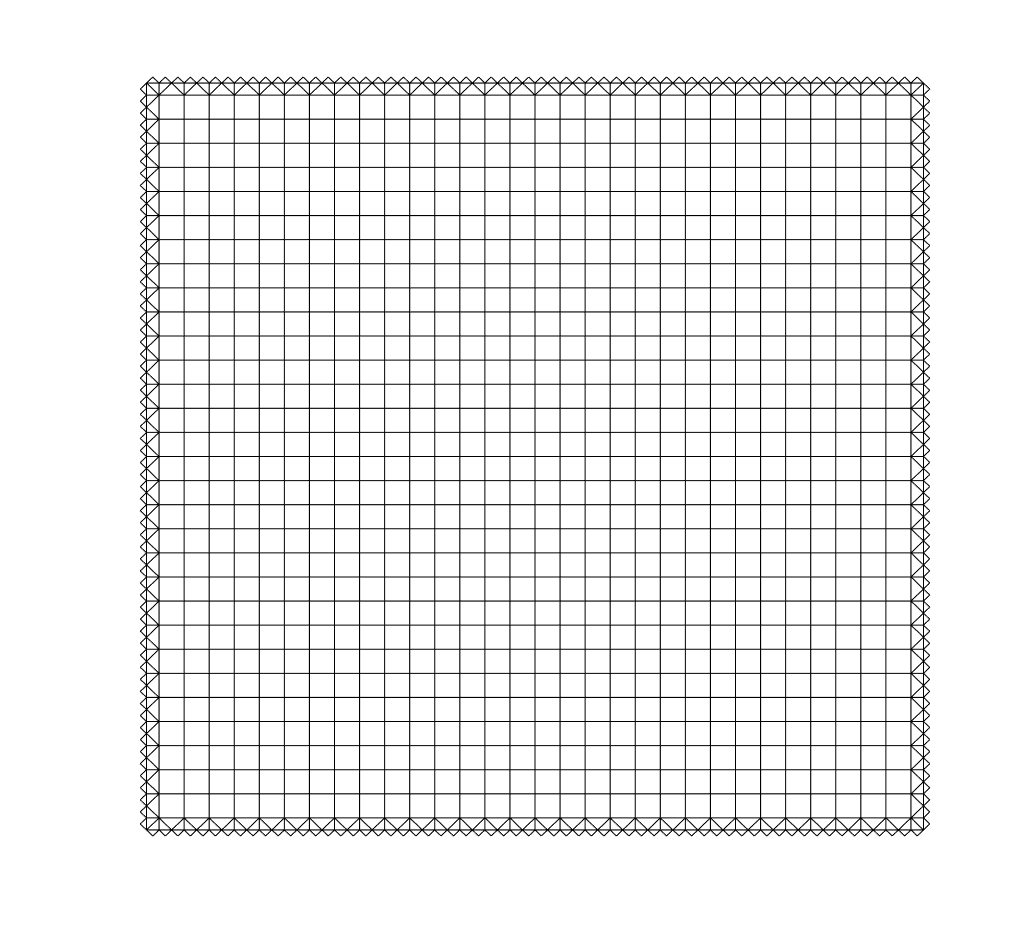}&
\includegraphics[scale=0.39]{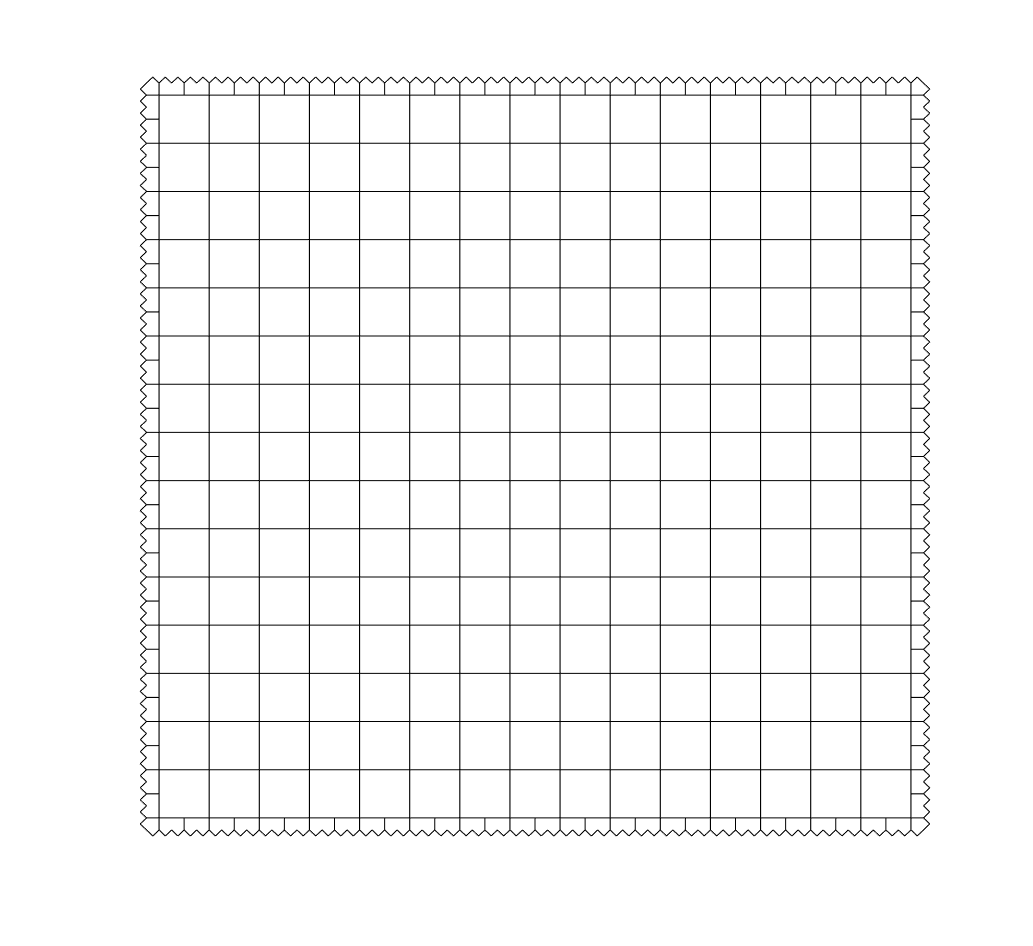}
\end{tabular}
\end{center}
\caption{Example 3. Hybrid triangular and rectangular mesh with $1516$ elements (left). Polygonal mesh with $349$ elements (right).} \label{different meshes}
\vspace{-0.4cm}
\end{figure}

 \begin{table}[!htb]
\begin{center}
\begin{tabular}{|c|c|c|c|c|}
 \hline
Mesh and basis & \multicolumn{2}{|c|}{$349$ elements P$2$ basis}   &\multicolumn{2}{|c|}{$1516$ elements P$1$ basis} \\
   \hline
Time steps  & {$n=16$}   &  {$n=32$  } & {$n=16$}   &  {$n=32$  }\\
 \hline
Total degrees of freedom &	55840	&	111680	&	97024 &	194048	\\
 \hline
$||u-u_h||_{L_2(J;L_2(\Omega))}$	&	2.2445e-04	&	1.8833e-04	&	1.0661e-03	&	7.4505e-04	\\
  \hline
$||u-u_h||_{L_2(J;H^1(\Omega,\mathcal{T}))}$	&	9.8263e-03	&	9.7722e-03	&	1.3131e-02	&	1.0131e-02	\\
 \hline
\end{tabular}
 \end{center}
\caption{Example 3. Numerical results corresponding to the meshes of Figure~\ref{different meshes}.}
 \end{table}

\section*{Acknowledgements}
The authors wish to express their gratitude to Oliver Sutton (University of Leicester) for his help in generating agglomerated meshes. AC acknowledges partial support from the EPSRC (Grant EP/L022745/1). EHG acknowledges support from The Leverhulme Trust (Grant RPG-2015-306).


\bibliographystyle{siam}
\bibliography{space_time_DG_final}

\end{document}